\documentclass[11pt]{amsart}
\usepackage{euscript}
\usepackage{amssymb}
\usepackage{amsmath}
\usepackage{graphics}
\usepackage{epsfig}
\usepackage{color}

\usepackage{amscd,euscript}
\usepackage[frame,cmtip,curve,arrow,matrix,line,graph]{xy}

\usepackage[english]{babel}
\usepackage{tikz} 
\usetikzlibrary{arrows,decorations.pathmorphing,decorations.pathreplacing,backgrounds,fit,positioning,shapes.symbols,chains,snakes}

\numberwithin{equation}{section}
\AtBeginDocument{\def\MR#1{}}

\newtheorem{thm}{Theorem}[section]
\newtheorem{Theorem}[thm]{Theorem}
\newtheorem*{Theorem*}{Theorem}
\newtheorem{Corollary}[thm]{Corollary}

\newtheorem*{corollary*}{Corollary}
\newtheorem{Lemma}[thm]{Lemma}

\newtheorem{Proposition}[thm]{Proposition}

\newtheorem*{conjecture*}{Conjecture}

\newtheorem*{question*}{Question}

\newtheorem{Definition}[thm]{Definition}

\newtheorem*{definitions*}{Definitions}

\newtheorem*{rem*}{Remark}
\newtheorem{Remark}[thm]{Remark}

\newtheorem*{remark*}{Remark}

\newtheorem*{remarks*}{Remarks}
\newtheorem*{example*}{Example}
\newtheorem{Example}[thm]{Example}
\newtheorem*{examples*}{Examples}

\newtheorem*{convention*}{Convention}
\newtheorem*{conventions*}{Conventions}

\newtheorem*{exercise*}{Exercise}
\newtheorem*{bibliographical-note*}{Bibliographical note}

\newcommand{\Acknowledgements}{{\em Acknowledgements.} }

\newcommand{\scrA}{\EuScript{A}}

\newcommand{\calF}{\mathcal{F}}

\newcommand{\calI}{\mathcal{I}}

\newcommand{\scrX}{\EuScript{X}}

\newcommand{\scrY}{\EuScript{Y}}

\newcommand{\scrP}{\EuScript{P}}

\newcommand{\scrW}{\EuScript{W}}

\newcommand{\scrL}{\EuScript{L}}

\newcommand{\bG}{\mathbb{G}}
\newcommand{\bF}{\mathbb{F}}
\newcommand{\bR}{\mathbb{R}}
\newcommand{\bZ}{\mathbb{Z}}
\newcommand{\bQ}{\mathbb{Q}}
\newcommand{\bC}{\mathbb{C}}
\newcommand{\bN}{\mathbb{N}}
\newcommand{\bP}{\mathbb{P}}

\newcommand{\id}{\mathrm{id}}

\newcommand{\im}{\mathrm{im}}
\renewcommand{\ker}{\mathrm{ker}}

\newcommand{\Hom}{\mathrm{Hom}}
\newcommand{\End}{\mathrm{End}}

\newcommand{\rk}{\mathrm{rk}}
\newcommand{\Log}{\mathrm{Log}}

\newcommand{\Aut}{\mathrm{Aut}}
\newcommand{\Ham}{\mathrm{Ham}}
\newcommand{\Symp}{\mathrm{Symp}}
\newcommand{\Diff}{\mathrm{Diff}}

\newcommand{\Ob}{\mathrm{Ob}}

\newcommand{\scrF}{\EuScript{F}}

\newcommand{\scrG}{\EuScript{G}}

\newcommand{\bK}{\mathbb{K}}

\newcommand{\A}{\mathcal{A}}

\numberwithin{equation}{section}
\renewcommand{\leq}{\leqslant}
\renewcommand{\geq}{\geqslant}

\newcommand{\C}{\mathbb C}

\newcommand{\Z}{\mathbb{Z}}
\newcommand{\R}{\mathbf R}

\newcommand{\Tw}{\operatorname{Tw}}

\newcommand{\Coh}{\operatorname{Coh}}

\newcommand{\Ker}{\operatorname{Ker}}
\newcommand{\Ext}{\operatorname{Ext}}

\newcommand{\Def}{\operatorname{Def}}

\newcommand{\Auteq}{\mathrm{Auteq}}

\title{Fukaya categories of surfaces, spherical objects, and mapping class groups}
\author{Denis Auroux}
\address{Department of Mathematics, Harvard University, Cambridge MA 02138, USA.}
\email{auroux@math.harvard.edu}
\author{Ivan Smith}
\address{Centre for Mathematical Sciences, University of Cambridge, CB3 0WB, England.}
\email{is200@cam.ac.uk}

\begin{document}
\maketitle \thispagestyle{empty}

\parindent0em
\parskip0.5em

\raggedbottom

\begin{abstract} We prove that every spherical object in the derived Fukaya category of a closed surface of genus at least two whose Chern character represents a non-zero Hochschild homology class is quasi-isomorphic to a simple closed curve equipped with a rank one local system.  (The homological hypothesis is necessary.)  This  largely answers a question of Haiden, Katzarkov and Kontsevich.  It follows that there is a  natural surjection from the autoequivalence group of the Fukaya category to the mapping class group. The proofs appeal to and illustrate numerous recent developments: quiver algebra models for wrapped categories, sheafifying the Fukaya category,  equivariant Floer theory for finite and continuous group actions, and homological mirror symmetry. An application to high-dimensional symplectic mapping class groups is included.
\end{abstract}





\section{Introduction}

The mapping class group $\Gamma_g$ of a closed surface $\Sigma_g$ arises naturally in different contexts: in algebra as the outer automorphism group $\mathrm{Out}(\pi_1\,\Sigma_g)$, in topology as the component group $\pi_0\Diff^+(\Sigma_g)$, in algebraic geometry as the orbifold fundamental group $\pi_1^{orb}(\mathcal{M}_g)$ of the moduli space of curves.  In Floer theory, and mirror symmetry, a symplectic manifold $Z$ appears through the Fukaya category $\scrF(Z)$ and  its derived category $D^{\pi}\scrF(Z)$, a formal algebraic enlargement of $\scrF(Z)$ introduced to have better homological-algebraic properties. The natural symmetry group of a surface in that context is the group of autoequivalences $\Auteq(D^{\pi}\scrF(\Sigma_g))$.   This comes with a map $\Gamma_g \to \Auteq(D^{\pi}\scrF(\Sigma_g))$ (which depends on additional choices), which has no obvious instrinsic categorical or Floer-theoretic  description; in many cases, we know of autoequivalences of Fukaya categories which are not geometric \cite{Abouzaid-Smith:plumbing}. This paper shows that the mapping class group arises naturally from the Fukaya category.

Let $(\Sigma_g, \omega)$ denote a closed surface of genus $g \geq 2$, equipped with an area form of area $1$. 
Let $\scrF(\Sigma_g)$ denote the Fukaya category of $\Sigma_g$, which is a $\bZ/2$-graded 
$A_{\infty}$-category, linear over the one-variable Novikov field $\Lambda = \bC(\!(q^{\bR})\!)$.  Objects of the category are immersed unobstructed closed curves, equipped with  a finite rank local system and auxiliary brane data (including a choice of spin structure on the underlying curve).  We will denote by $(\xi,\gamma)$ the object associated to an immersed closed curve $\gamma: S^1 \to \Sigma_g$ and local system $\xi \rightarrow S^1$ on the domain of $\gamma$.

 We denote by $D^{\pi}\scrF(\Sigma_g) = \scrF(\Sigma_g)^{per}$ the category of perfect modules over $\scrF(\Sigma_g)$, equivalently the split-closure of twisted complexes, which is triangulated in the classical sense; write $\simeq$ for quasi-isomorphism in this category. The composite of the Chern character and the open-closed map defines a class
 \[
 ch(A) \in HH_{0}(D^{\pi}\scrF(\Sigma_g), D^{\pi}\scrF(\Sigma_g)) \cong H_1(\Sigma_g;\Lambda)
 \]
 for any object $A\in D^{\pi}\scrF(\Sigma_g)$.  Recall that an object $A \in D^{\pi}\scrF(\Sigma_g)$ is \emph{spherical} if 
 \[
H^*( hom_{D^{\pi}\scrF(\Sigma_g)}(A,A))  \cong H^*(S^1;\Lambda).\]

Lemma \ref{Cor:HH_integral} shows that $ch(X)$ is an integral class when $X$ is spherical.
Our main result is the following geometricity theorem for spherical objects.

\begin{Theorem}\label{Thm:Main}
If $X \in D^{\pi}\scrF(\Sigma_g)$ is spherical and $ch(X)$ is non-zero, then there is a simple closed curve $\gamma \subset \Sigma_g$ and a rank one local system $\xi \rightarrow \gamma$ with $X \simeq (\xi,\gamma)$.  
\end{Theorem}

En route, we prove the corresponding result for surfaces with non-empty boundary
(Corollary \ref{Cor:spherical_on_punctured}).
When $g=1$, Theorem \ref{Thm:Main} is a consequence of homological mirror symmetry for elliptic curves  and Atiyah's classification of bundles on such curves. When $g>1$, 
there are spherical objects with vanishing Chern character which are \emph{not} quasi-isomorphic to any simple closed curve with local system, so the result is in some sense sharp; see Lemma  \ref{Lem:bad}.  Theorem \ref{Thm:Main}  largely answers \cite[Problem 2]{HKK}.  

Let $\Gamma(\Sigma_g)$ denote the symplectic mapping class group of $\Sigma_g$.
There is a homomorphism $\Gamma(\Sigma_g) \to \Auteq(D^{\pi}\scrF(\Sigma_g))$ (this depends on choices, cf. Section \ref{Sec:Monodromy}).

\begin{Corollary} \label{Cor:Main}
There is a natural surjective homomorphism $\Auteq(D^{\pi}\scrF(\Sigma_g)) \to \Gamma(\Sigma_g)$. 
\end{Corollary}

The same conclusion also holds for surfaces of genus $\geq 1$ with boundary, by combining
Corollary \ref{Cor:spherical_on_punctured}
with the argument of Proposition \ref{Prop:split}.

It follows that the map $\Gamma(\Sigma_g) \to \Auteq(D^{\pi}\scrF(\Sigma_g))$ splits.   The existence of such a split homomorphism  has cohomological implications; for instance,  the autoequivalence group has infinite-dimensional second bounded cohomology, and admits families of unbounded quasimorphisms.

We conjecture that, for $g\geq 2$ (so the flux group is trivial),  the kernel of the natural map from autoequivalences to the mapping class group is generated by tensoring by flat unitary line bundles and the actions of symplectomorphisms of non-trivial flux, i.e. that
\[
\Auteq(D^{\pi}\scrF(\Sigma_g)) \stackrel{?}{=} H^1(\Sigma_g;\Lambda^*) \rtimes \Gamma(\Sigma_g),
\]
where the map $\Gamma(\Sigma_g) \to \Auteq(D^{\pi}\scrF(\Sigma_g))$ is only well-defined up to the action of $H^1(\Sigma_g;\bR)$.
In Section \ref{Sec:Schmutz} we outline an argument suggesting that the
subgroup $H^1(\Sigma_g;\Lambda^*)$ is normal, which is consistent with the speculation.

The proof of Theorem \ref{Thm:Main} is surprisingly involved, and breaks into the following steps (the main text treats these in somewhat different order). Let $X \in D^{\pi}\scrF(\Sigma_g)$ be a spherical object.

\begin{enumerate}
\item The open-closed image $ch(X) \in HH_0(\scrF(\Sigma_g)) = H_1(\Sigma_g;\Lambda)$ defines an integral class, i.e. lies in the image of $H_1(\Sigma_g;\bZ)$.  \emph{Assume henceforth this class is non-zero.}

\item A non-zero integral class  $a \in H^1(\Sigma_g;\Z)$ defines a $\bG_m$-action on $\scrF(\Sigma_g)$.

\item If $\langle a, ch(X) \rangle = 0$, then $X$ defines a $\bG_m$-equivariant object, hence a $\bZ/N$-equivariant object for any finite $\bZ/N \leq \bG_m$. The $\bZ/N$-equivariant Fukaya category of $\Sigma_g$ is the Fukaya category of an $N$-fold cover $\tilde{\Sigma}$ of $\Sigma_g$.

\item   A choice of equivariant structure on $X$ defines a lift $\hat{X}$ of $X$ to $\tilde{\Sigma}$, and for large enough $N$, there are disjoint homologically independent simple closed curves $\gamma_1, \gamma_2 \subset \tilde{\Sigma}$ with $H^*(\hom_{\scrF(\tilde{\Sigma})}(\hat{X}, \gamma_i)) = 0$.

\item There are annular neighbourhoods $\gamma_i \subset A_i \subset \tilde{\Sigma}$ for which $\hat{X}$ lifts to define a spherical object (which we still call) $\hat{X}$ in the wrapped category $\scrW(C)^{per}$ of  $C = \hat{\Sigma} \backslash (A_1\cup A_2)$.

\item   $\hat{X}$ is represented by a strictly unobstructed immersed closed curve $\sigma \subset C$ equipped with a unitary local system.

\item If it is not embedded, the immersed curve $\sigma$ supporting $\hat{X}$  bounds an embedded bigon.

\item Bigons on $\sigma$ may be ``emptied" and then ``cancelled", so $\hat{X}$ is quasi-isomorphic to a simple closed curve with rank one local system in $\scrF(C)$.

\item $X$ is quasi-isomorphic to a simple closed curve with rank one local system in $\scrF(\Sigma_g)$.
\end{enumerate}

Steps (2)-(4) rely on work of Seidel on equivariant Floer theory \cite{Seidel:Am_milnor_fibres, Seidel:categorical_dynamics}, and ideas of family Floer theory \`a la Abouzaid and Fukaya.  Step (5) relies on H.~Lee's restriction technology \cite{HLee} for sheafifying wrapped categories of surfaces. Step (6) appeals to the work of Haiden, Katzarkov and Kontsevich \cite{HKK} on quiver algebra models for wrapped categories of punctured surfaces, and to a split-closure result for such wrapped categories which we infer from homological mirror symmetry  and a K-theoretic characterisation of split-closure for derived categories of singularities due to Abouzaid, Auroux, Efimov, Katzarkov and Orlov \cite{AAEKO}.  Step (7) invokes delicate classical  results of Steinitz \cite{Steinitz2, Schleimer_et_al} and  Hass and Scott \cite{Hass-Scott} in surface topology. The ``cancelling" move for bigons in Step (8) uses the homological hypothesis from Step (1).  Corollary \ref{Cor:Main} follows by considering the action of autoequivalences of $\scrF(\Sigma_g)$ on a Floer-theoretic ``Schmutz graph" of non-separating curves \cite{SchmutzSchaller}.

As an application of Corollary \ref{Cor:Main}, in Section \ref{Sec:application} we prove:

\begin{Theorem} \label{Thm:largeMCG}
There is a smooth manifold $Z$ with symplectic forms $\omega_{\delta}$, $\delta \in (0,1]$, for which $\pi_0\Symp(Z,\omega_{\delta})$ surjects to a free group of rank $N(\delta)$ where $N(\delta) \to \infty$ as $\delta \to 0$.
\end{Theorem}

Explicitly, $Z$ is the product of $\Sigma_2$ with the blow-up of the four-torus at a point, but equipped with an irrational perturbation (of size $\delta$) of the standard K\"ahler form.  The free group quotients arise from subgroups of the genus two Torelli group.  Gromov \cite{Gromov}, Abreu-McDuff \cite{Abreu-McDuff} and others showed that the topology of the symplectomorphism group can vary rather wildly as one continuously deforms the symplectic form, but this seems to be the first example in which a fixed degree homotopy group is known to have unbounded rank. 

\noindent \emph{Notation:} Throughout the paper, surfaces are connected, and  (immersed) curves are assumed to be homotopically non-trivial  unless explicitly stated otherwise.

\Acknowledgements Thanks to Paul Seidel for suggesting we use equivariance and for several related explanations; and to Mohammed Abouzaid, Fabian Haiden,
Yank\i{} Lekili, Dhruv Ranganathan and Henry Wilton for helpful conversations. We are grateful to the anonymous referees for their comments and suggestions. 
DA is partially supported by NSF grant DMS-1937869 and by Simons Foundation grant
\#385573 (Simons Collaboration on Homological Mirror Symmetry).
IS is partially supported by Fellowship EP/N01815X/1 from the Engineering and Physical Sciences Research Council, U.K.


\section{The Fukaya category of a surface\label{subsec:closed_fukaya}}

This section collects background on Floer theory and the Fukaya category for a two-dimensional surface.  A reader wanting a more comprehensive treatment might consult \cite{Seidel:FCPLT, Seidel:HMSgenus2} for $A_{\infty}$-algebra and the general construction of the Fukaya category, and \cite{Abouzaid:higher_genus, AAEKO, HLee} for related discussions of Fukaya categories of  surfaces (note some of the latter references do not take the split-closure).

\subsection{Background}
Fix a coefficient field $k$.  Let $\bK= \Lambda_{k}$ denote the single-variable Novikov field, with formal variable $q$,  of formal series $\sum_i a_i q^{t_i}$ with $a_i \in k$, $t_i \in \bR$ and $\lim t_i = +\infty$.  The valuation map
\[
val: \Lambda \to \bR \cup \{\infty\}, \quad val(0) = +\infty
\]
associates to a non-zero element of $\Lambda$ its smallest $q$-power. 
The subring $\Lambda_{\geq 0} = val^{-1}[0,\infty]$ comprises series with $t_i \geq 0$ for every $i$, and there is a homomorphism 
\[
\eta: \Lambda_{\geq 0} \rightarrow k
\]
which extracts the constant coefficient.  The kernel of this homomorphism is denoted $\Lambda_{>0}$. 
The unitary subgroup $U_{\Lambda} = val^{-1}(0)$ is the subgroup of elements $a + \sum_{t_i>0} a_i q^{t_i}$ where $a\in k^*$ is non-zero. The field $\Lambda_{\bC}$ is algebraically closed of characteristic zero.

A non-unital $A_{\infty}$-category $\scrA$ over the field $\bK$ comprises: a set of objects Ob$\,\scrA$; for each $X_0, X_1 \in$ Ob$\,\scrA$ a $\bZ/2$-graded $\bK$-vector space $hom_{\scrA}(X_0, X_1)$; and $\bK$-linear composition maps, for $k \geq 1$,
\[
\mu_{\scrA}^k: hom_{\scrA}(X_{k-1},X_k) \otimes \cdots \otimes hom_{\scrA}(X_0,X_1) \longrightarrow hom_{\scrA}(X_0,X_k)[2-k]
\]
(here $[j]$ denotes downward shift by $j \in \bZ/2$, and all degrees are mod 2 degrees; we write $2-k$ rather than $-k$ since $\bZ$-graded categories may be more familiar).  The maps $\{\mu^k\}$ satisfy a hierarchy of quadratic equations
\[
\sum_{m,n} (-1)^{\maltese_n} \mu_{\scrA}^{k-m+1} (a_k,\ldots,a_{n+m+1}, \mu_{\scrA}(a_{n+m},\ldots,a_{n+1}),a_n \ldots, a_1) = 0
\]
with $\maltese_n = \sum_{j=1}^n |a_j|-n$ and where the sum runs over all possible compositions: $1 \leq m \leq k$, $0\leq n \leq k-m$.  If the characteristic of $\bK$ is not equal to $2$, the signs in the $A_{\infty}$-associativity equations depend on the mod 2 degrees of generators, and the existence of a $\bZ/2$-grading on $\scrA$ is essential.

We denote by $\Tw^{\pi}(\scrA) = \scrA^{per}$ the idempotent completion of the category of twisted complexes $\Tw(\scrA)$ \cite{Seidel:FCPLT}.    If the smallest split-closed triangulated $A_{\infty}$-category containing a subcategory $\scrA' \subset \scrA$  is $\scrA^{per}$, then we will say that $\scrA'$ split-generates $\scrA$. 

In a curved $A_{\infty}$-category each object $A$ comes with a curvature $\mu^0 \in hom^{ev}_{\scrA}(A,A)$, where $hom^{ev}_{\scrA}$ denotes the subgroup of morphisms of even degree;  the quadratic associativity relations are modified to take account of $\mu^0$.  In general the resulting sums could \emph{a priori} be infinite, so some geometric conditions are required to ensure convergence.  The following result controls the curvature of mapping cones on non-closed morphisms. 

\begin{Lemma} \label{Lem:Curvature_of_cone}
If $\scrA$ is a curved $A_{\infty}$-category and $X_1 \stackrel{f}{\longrightarrow} X_2$ is a mapping cone in $\Tw(\scrA)$ between objects $X_i$ with $\mu^0(X_i) = 0$, then 
\[
\mu^0_{\scrA}(X_1 \stackrel{f}{\longrightarrow} X_2)  = \mu^1_{\scrA}(f).\]
\end{Lemma}

\begin{proof} Standard. \end{proof}

The Hochschild cohomology $HH^*(\scrA,\scrA) = \Hom_{(\scrA-mod-\scrA)}(\id_{\scrA},\id_{\scrA})$
of an (uncurved) $A_\infty$-category is the endomorphisms of the diagonal bimodule. 
Hochschild cohomology of an $A_{\infty}$-category is invariant under passing to the derived category,  so we will not distinguish notationally between $HH^*(\scrA,\scrA)$ and $HH^*(D^{\pi}\scrA,D^{\pi}\scrA)$. 

We say that an $A_{\infty}$-category $\scrA$ is \emph{homologically smooth} if the diagonal bimodule  is perfect (split-generated by Yoneda bimodules), and \emph{proper}\footnote{Some authors call this  `locally proper', and say $\scrA$ is proper if  it also admits a compact generator.} if it is cohomologically finite.

\subsection{Immersed curves}

Let $\Sigma$ be a closed oriented surface of genus $g(\Sigma) = g \geq 2$ equipped with an area form $\omega$ of unit total area, $\int_{\Sigma} \omega = 1$.  

The Fukaya category $\scrF(\Sigma)$ has objects which are called Lagrangian branes, and which comprise an immersed curve $\iota: S^1 \rightarrow \Sigma$ with the following properties and additional data.  First note that if $\gamma = \iota(S^1)$ is immersed, it may bound ``teardrop discs", i.e. images of holomorphic discs with a unique boundary puncture where the puncture is mapped to a self-intersection point of $\gamma$.  Let $\mu^0(\gamma) \in CF^0(\gamma,\gamma)$ be the algebraic sum of all such discs, which is the obstruction term in the self-Floer complex.  We insist:

\begin{itemize}
\item $\iota(S^1)= \gamma$ is unobstructed\footnote{An `unobstructed' Lagrangian is usually defined to be a pair $(L,b)$ where $b \in CF^*(L,L;\Lambda_{>0})$ solves the $A_{\infty}$-Maurer Cartan equation $\sum_{k\geq 0} \mu^k(b,\ldots,b) = 0$.  Unobstructedness for us is the special case in which $b=0$ is a Maurer-Cartan solution, which is all that we will require.}, meaning that $\mu^0(\gamma) \in CF^0(\gamma,\gamma)$  vanishes;
\item all self-intersections of $\gamma$ are transverse;
\item the domain $S^1$ is equipped with a flat unitary local system of $\Lambda$-vector spaces;
\item $\gamma$ is equipped with a $spin$ structure.
\end{itemize}

Unitarity of the local system means that the monodromy takes values in the
subgroup of valuation-preserving elements of $GL(n,\Lambda)$, i.e., square matrices
with entries in $\Lambda_{\geq 0}$ and such that, after discarding
positive powers of $q$, the constant terms form an invertible matrix.

Note that the first condition is that $\mu^0(\gamma)$ vanishes identically (i.e. at each self-intersection point).  We could equally well  insist that objects of $\scrF(\Sigma)$ are tautologically unobstructed in the sense of bounding no teardrop
discs; in fact the two conditions are equivalent:

\begin{Lemma} \label{l:unobstructed}
If $\gamma$ is a homotopically non-trivial immersed curve, then
the following are equivalent: $(1)$ $\mu^0(\gamma)=0$;
$(2)$ $\gamma$ does not bound any teardrop discs; $(3)$ 
$\gamma$ lifts to a properly embedded arc in the universal cover of
$\Sigma$.
\end{Lemma}

\begin{proof}
Since $\gamma$ is homotopically non-trivial, it lifts to a properly immersed
arc $\tilde{\gamma}$ in the universal cover $\tilde{\Sigma}$.
Teardrop discs bounded by $\gamma$ lift to teardrop discs in
$\tilde{\Sigma}$ with
boundary on $\tilde{\gamma}$; thus (3) implies (2), and (2) implies (1).

Conversely, assume that $\tilde{\gamma}$ is not embedded, and let $a\in
\tilde{\Sigma}$ be any self-intersection of $\tilde{\gamma}$. If the arc
$\tilde{\gamma}_a$ consisting of the portion of $\tilde{\gamma}$ which connects $a$ to itself is not embedded, then there
exists another ``nested'' self-intersection $b$ such that the arc $\tilde{\gamma}_b$ connecting $b$
to itself is a strict subset of $\tilde{\gamma}_a$.
Considering $b$ instead of $a$, and repeating the process if needed, we
can assume that $\tilde{\gamma}_a$ is embedded in
$\tilde{\Sigma}$; the portion $D_a$ of $\tilde{\Sigma}$ enclosed by
$\tilde\gamma_a$ is then an embedded teardrop disc with a corner at $a$.

The teardrop $D_a$ may be either locally convex or locally concave near its
corner, meaning it occupies either 1 or 3 of the quadrants delimited by
the two branches of $\tilde{\gamma}$ intersecting at $a$. Recall that only
locally convex teardrops contribute to $\mu^0(\gamma)$.  We claim that,
if $D_a$ has a locally concave corner at $a$, then there are
smaller embedded teardrop discs contained inside it. Indeed, in the locally
concave case, the portions of $\tilde{\gamma}$ just before and after the arc
$\tilde{\gamma}_a$ lie inside $D_a$. Continuing along $\tilde{\gamma}$ until
it exits $D_a$ (which must eventually happen since
$\tilde{\gamma}$ is properly immersed), we find another self-intersection
$b$ lying on the boundary of $D_a$, such that the arc $\tilde{\gamma}_b$ 
connecting $b$ to itself is entirely contained in $D_a$ (and is not the
entire boundary of $D_a$).  If $\tilde{\gamma}_b$ is
not embedded, then we replace $b$ by a nested self-intersection as above;
this allows us to assume that $\tilde{\gamma}_b$ is embedded, while still
ensuring that $\tilde{\gamma}_b$ is entirely contained in $D_a$.
The region of $\tilde{\Sigma}$ bounded by $\tilde{\gamma}_b$ then gives a
new embedded teardrop disc $D_b$ which is strictly contained in $D_a$.

Repeating the process, we conclude that if $\tilde{\gamma}$ is not embedded
then it must bound an embedded teardrop disc with a locally convex
corner.  The generator of $CF^0(\gamma,\gamma)$ corresponding to this
self-intersection must then have a non-zero coefficient in
$\mu^0(\gamma)$, since there are no other teardrops with the same corner.
Thus, if $\mu^0(\gamma)=0$ then $\tilde\gamma$ must be embedded.
\end{proof}

\begin{Remark}
Regular homotopies of immersed curves do not
preserve the absence of teardrops or the vanishing of $\mu^0$.
For instance, pushing an embedded arc sideways then back through itself to create 
a pair of self-intersections gives rise to a teardrop (and a bigon).
Thus, in the arguments below we will take care to only consider
regular homotopies which preserve unobstructedness.
\end{Remark}

\begin{Remark}
On a two-dimensional surface, rigid $J$-holomorphic discs with pairwise distinct boundary conditions are immersed polygons with convex corners, which are purely combinatorial. One can set up the Fukaya category either via moving Lagrangian boundary conditions and honest $J$-holomorphic curves or via (the more usual approach with) Hamiltonian perturbation terms in the Floer equation.   See \cite{Seidel:HMSgenus2} for an implementation.
\end{Remark}

\begin{Remark}
The results of this section apply, \emph{mutatis mutandis}, to the case of a surface $S$ with non-empty boundary. The main difference is that it is sometimes necessary to consider immersed curves which live in the (infinite area) completion of $S$, rather than in $S$ itself. We will pay careful attention to this issue when it arises in the sequel. 
\end{Remark}

\subsection{Isotopies and twists}

For each $a\in H^1(\Sigma;\bR)$ there is a symplectomorphism $\phi_a$ with flux $a$, and one can move Lagrangian submanifolds by such symplectomorphisms.  This obviously preserves unobstructedness, since it preserves all teardrop discs and their areas.  There is a more general statement for isotopies not induced by global symplectomorphisms.

Two curves on a surface meet \emph{minimally} if they meet transversely in their geometric intersection number of points.

\begin{Lemma}  \label{Lem:find_sigma} Let $\gamma \subset \Sigma$ be an immersed curve and $\sigma$ a simple closed curve. One can isotope $\gamma$ by a regular homotopy so as to meet $\sigma$ minimally.
Moreover, the regular homotopy can be chosen to preserve the unobstructedness of $\gamma$.
\end{Lemma}

\begin{proof}
Suppose the intersection is not minimal, so there is a not necessarily embedded bigon bound by $\sigma\cup\gamma$.  Pull back $\sigma$ to the domain of such a bigon; changing the choice of bigon to one that is ``innermost" if necessary, we can assume that there is a bigon $H$ which does not meet $\sigma$ in its interior. The boundary of this bigon therefore contains a proper subset of $\gamma$, since the boundary arc of the bigon lying along $\gamma$ connects two intersections in $\gamma \cap \sigma$ which are consecutive as read along $\gamma$. (Note that the $\sigma$-edge of this bigon might loop fully around $\sigma$ even with the bigon being innermost; see Figure \ref{Fig:lemmafind_sigma}.)  Now push the arc of $\gamma$ lying along $H$ across $H$, to decrease the total intersection number of $\gamma$ and $\sigma$. This can be iterated to reduce to a minimal situation.

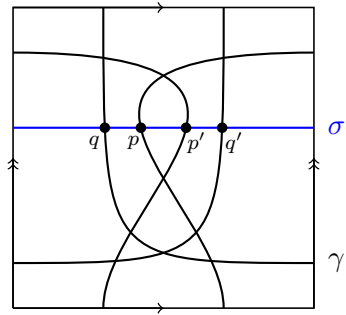
\begin{figure}[t]
\begin{center} 
\begin{tikzpicture}[scale = 2]
\draw[semithick] (0,0)--(1,0)--(2,0)--(2,2)--(0,2)--(0,0);
\draw[semithick,->](0,0)--(1,0);
\draw[semithick,->](0,2)--(1,2);
\draw[semithick,->>](0,0)--(0,1);
\draw[semithick,->>](2,0)--(2,1);
\draw[thick, blue] (0,1.2)--(2,1.2);
\draw[blue] (2.15,1.2) node {$\sigma$};
\draw[thick] (0,1.7) .. controls (2.2,1.7) and (0.6,0.7) .. (0.6,0);
\draw[thick] (0.6,2) .. controls (0.6,0.3) .. (2,0.3);
\draw[thick] (2,1.7) .. controls (-0.2,1.7) and (1.4,0.7) .. (1.4,0);
\draw[thick] (1.4,2) .. controls (1.4,0.3) .. (0,0.3);
\draw(2.15,0.3) node {$\gamma$};
\fill(0.61,1.2) circle (0.035); 
\fill(0.85,1.2) circle (0.035); 
\fill(1.39,1.2) circle (0.035); 
\fill(1.15,1.2) circle (0.035); 
\draw (0.54,1.1) node {\SMALL $q$};
\draw (0.8,1.1) node {\SMALL $p$};
\draw (1.22,1.1) node {\SMALL $p'$};
\draw (1.47,1.1) node {\SMALL $q'$};
\end{tikzpicture}
\end{center}
\caption{$\sigma$ and $\gamma$ bound four bigons in $T^2$; the boundaries of
the innermost bigons (those not meeting $\sigma$ in their interior:
above $p-p'$ and below $q-q'$) loop slightly more than once around $\sigma$.}
\label{Fig:lemmafind_sigma}
\end{figure}

By Lemma \ref{l:unobstructed}, $\gamma$ is unobstructed if and
only if its lift $\tilde\gamma$ to the universal cover of $\Sigma$ is
embedded. Lifting to $\tilde\Sigma$ the bigon $H$ along which we slide $\gamma$
in the above argument, we obtain a bigon $\tilde{H}$ with boundary on 
$\tilde\gamma$ and on a lift $\tilde\sigma$ of $\sigma$, whose interior is moreover disjoint
from all lifts of $\sigma$. If the interior of $\tilde{H}$ is disjoint from
$\tilde{\gamma}$ then sliding $\gamma$ across $H$ preserves unobstructedness.
Otherwise, by pulling back $\tilde\gamma$ to the domain of the bigon $\tilde{H}$ 
(which yields a disjoint collection of embedded arcs, since $\tilde\gamma$
is embedded) and changing the choice of bigon to one that is innermost, we obtain
a bigon with boundary on $\tilde\gamma \cup \tilde\sigma$ whose interior is 
disjoint from $\tilde\gamma$ and from all lifts of $\sigma$. Projecting back
to $\Sigma$ and sliding across this bigon decreases the intersection number
without affecting unobstructedness.
\end{proof}

\begin{Remark} \label{Rmk:find_sigma}
This argument also gives the following generalization of 
Lemma \ref{Lem:find_sigma}: let $S$ be a
compact surface with boundary, $\sigma$ a simple closed curve, and 
$\gamma$ either an immersed curve
or an immersed arc with ends in $\partial S$. Then one can isotope $\gamma$
by a regular homotopy so as to meet $\sigma$ minimally, in a way that
preserves the absence of teardrops with boundary on $\gamma$.  
\end{Remark}

\begin{Lemma} \label{Lem:sweeparea}
For any immersed curve $\gamma \subset \Sigma$ with $\mu^0(\gamma) = 0$, and $a\in \bR$, there is a regular isotopy $\gamma_t$  of $\gamma$ through (generically) immersed unobstructed curves which sweeps area $a$. 
\end{Lemma}

\begin{proof}
Pick a non-separating simple closed curve $\sigma$ which meets $\gamma$ minimally and with non-zero geometric intersection number.  Such a curve can be obtained from  Lemma \ref{Lem:find_sigma}. (The construction of $\sigma$ in that Lemma  involves an isotopy of $\gamma$ which might now have non-trivial flux, but the conclusion of the current Lemma for the modified curve would then yield the same result for the initial one, just with a different base-point to the one-parameter family).  
One obtains $\gamma_t$ by sliding a portion of $\gamma$ along $\sigma$, introducing a cancelling pair of self-intersections each time $\gamma$ is being pushed across itself. The resulting curves bound no teardrops other than those bound by $\gamma$.  
Indeed, in the universal cover, we can choose lifts of $\gamma$ and $\sigma$ which meet exactly once, by the minimal intersection assumption; and then the family $\gamma_t$ which slides along $\sigma$ bounds no new teardrops since its lift to the universal cover acquires no new self-intersections.
\end{proof}

\begin{Remark} \label{Rmk:sweeparea}
Let $\gamma$ be an unobstructed immersed curve in a compact surface with boundary $S$.
If $\gamma$ is not boundary-parallel
(i.e., it cannot be homotoped to a curve contained in a collar neighborhood of
$\partial S$), then there exists
a simple closed curve which has non-zero geometric intersection number
with $\gamma$, and the above argument shows that there are unobstructed
regular isotopies of $\gamma$ in $S$ which sweep arbitrary area $a\in \bR$. 
On the other hand, if $\gamma$ is boundary-parallel then
the total area between $\gamma$
and $\partial S$ is finite, and any regular homotopy sweeping more
than this amount of area must introduce a teardrop.
\end{Remark}

\begin{Lemma} \label{Lem:space}
Let $\gamma \subset \Sigma$ be an embedded simple closed curve. If $\gamma$ is non-separating, then there are embedded simple closed curves $\gamma_t$ smoothly isotopic to $\gamma$ and obtained by an isotopy of flux $t$, for every $t \in \bR$.
\end{Lemma}

\begin{proof}
Fix a simple closed curve $\sigma$ with non-zero algebraic intersection number with $\gamma$, and fix a small embedded cylinder centred on $\sigma$. Let $\sigma^{\pm}$ be disjoint embedded curves in this cylinder which bound a subcylinder of area $\varepsilon>0$.  The curve 
\[
(\tau_{\sigma^+} \circ \tau_{\sigma^-}^{-1})(\gamma)
\]
is then smoothly isotopic to $\gamma$ but differs from it by a flux of area $\varepsilon$ times the algebraic intersection number. Since it is the image of $\gamma$ under a diffeomorphism, it is obviously embedded.  By varying $\varepsilon$ and iterating the map $\tau_{\sigma^+} \circ \tau_{\sigma^-}^{-1}$ or its inverse, one obtains embedded curves differing from $\gamma$ by arbitrary real values of flux.
\end{proof}

The analogue of Lemma \ref{Lem:space} does not hold for separating simple closed curves; in that case, if one tries to move $\gamma$ by an isotopy sweeping a flux larger than the area of a subsurface bound by $\gamma$, one may have to introduce self-intersections, cf. Figure \ref{Fig:bad} below. 

For embedded curves, isotopies that sweep zero area are induced by
Hamiltonian diffeomorphisms, and the invariance of Floer theory is classical.
We record the following consequence (and note that the result also holds for
surfaces with boundary):

\begin{Corollary} \label{Cor:rankHF}
Let $\gamma_1,\gamma_2$ be two embedded simple closed curves with non-zero
algebraic intersection number, and let $X_1,X_2$ be the objects of
$\scrF(\Sigma)$ obtained by equipping $\gamma_1,\gamma_2$ with rank one
local systems. Then the rank of the Floer cohomology group $HF^*(X_1,X_2)$
is equal to the geometric intersection number of $\gamma_1$ and
$\gamma_2$.
\end{Corollary}

\begin{proof}
Move $\gamma_2$ by an isotopy in order to obtain a simple closed curve
$\gamma'_2$ which intersects $\gamma_1$ minimally, so that there are no
bigons bound by $\gamma_1\cup \gamma'_2$. The isotopy from $\gamma_2$ to
$\gamma'_2$ may sweep a non-zero amount of flux, but this can be remedied by
sliding $\gamma'_2$ along $\gamma_1$: applying the construction in the proof
of Lemma \ref{Lem:space} to $\gamma'_2$, 
taking $\sigma=\gamma_1$, yields $\gamma''_2$ which is Hamiltonian isotopic
to $\gamma_2$ and intersects $\gamma_1$ minimally. Replacing $X_2$ by the
quasi-isomorphic object $X''_2$ given by $\gamma''_2$ with the appropriate
local system, we find that the Floer complex $CF^*(X_1,X''_2)$ has rank
equal to the geometric intersection number and vanishing Floer differential.
\end{proof}

Invariance of Floer cohomology under isotopies that sweep zero area holds more generally for
immersed curves, even when we allow the areas of the regions bounded by
the curve to vary, or if we allow isotopies through a curve that does not
self-intersect transversely, creating or cancelling self-intersections.
We note that such regular homotopies are still induced by a
(time-dependent) Hamiltonian on the {\em domain} of the immersion, and
thus we will abusively refer to them as Hamiltonian isotopies. The following is closely related to \cite[Proposition 4.1]{Abouzaid:higher_genus}.

\begin{Lemma}\label{Lem:hamiltonian_isomorphic}
If unobstructed immersed curves $\gamma_0$ and $\gamma_1$ are regular homotopic through 
(generically self-transverse) unobstructed immersed curves by an isotopy that 
sweeps zero area (and which identifies their $Spin$ structures and local
systems), then 
$\gamma_0$ and $\gamma_1$ define quasi-isomorphic objects of $\scrF(\Sigma)$. 
\end{Lemma}

\begin{proof} By concatenation, this follows from the result for $C^1$-small
isotopies sweeping zero area.  When $\gamma_0$ and $\gamma_1$ are
$C^1$-close to each other, there exists a small time-independent Morse function
$H$ on the domain of the immersion whose Hamiltonian vector field generates
the isotopy. Each critical point of $H$ gives rise to an intersection of
$\gamma_0$ and $\gamma_1$; considering separately the maxima and minima, we
let $$p=\sum_{p_i \in \min(H)} q^{H(p_i)} p_i\in CF^0(\gamma_0,\gamma_1)
\quad \text{and} \quad p'=\sum_{p_j\in \max(H)} q^{-H(p_j)} p_j \in
CF^0(\gamma_1,\gamma_0).$$
When $\gamma_0$ and $\gamma_1$ are embedded, a classical argument
considering the bigons that connect the consecutive minima and maxima of $H$
shows that $\mu^1(p)=0$, $\mu^1(p')=0$, i.e.\ $p$ and $p'$ are Floer cocycles, 
and $\mu^2(p',p)=\mathrm{id}_{\gamma_0}$,
$\mu^2(p,p')=\mathrm{id}_{\gamma_1}$, i.e.\ $p$ and $p'$ provide an
explicit isomorphism between $\gamma_0$ and $\gamma_1$. When $\gamma_0$ and
$\gamma_1$ are immersed, there are additional generators of
$CF(\gamma_0,\gamma_1)$ and $CF(\gamma_1,\gamma_0)$ near the
self-intersections. However, by considering lifts of $\gamma_0$ and
$\gamma_1$ to the universal cover (which are embedded by our assumption of 
unobstructedness), we find that the lifts of holomorphic polygons which 
contribute to $\mu^1(p)$, $\mu^1(p')$, $\mu^2(p',p)$ and $\mu^2(p,p')$ 
cannot involve the generators coming from the self-intersections, and the
outcome of the calculation is exactly the same as in the embedded case.
\end{proof}

Any simple closed curve $\sigma \subset \Sigma$ has an associated Dehn twist $\tau_{\sigma}$. The identity $\id$ and $\tau_{\sigma}$ both define $A_{\infty}$-equivalences of $\scrF(\Sigma)$, and viewed as functors, there is a distinguished morphism $\Phi_\sigma: \id \to \tau_{\sigma}$ in $\Hom_{nu-fun}(\id,\tau_{\sigma})$ obtained from the count of sections of a Lefschetz fibration over a disc  with one interior critical point and vanishing cycle $\sigma$.

\begin{Lemma} \label{Lem:determines}
If $\sigma$ is non-separating, then $\tau_{\sigma}$ determines $\sigma$ up to Hamiltonian isotopy.  If $\sigma$ is separating, then the pair $(\tau_{\sigma}, \Phi_\sigma)$ determines $\sigma$ up to Hamiltonian isotopy.
\end{Lemma}

\begin{proof}
It is straightforward to see that $\tau_{\sigma}$ determines the smooth isotopy class of $\sigma$. 
Deforming a simple closed curve $\sigma$  by an isotopy that sweeps  area $\alpha$, the Dehn twist $\tau_{\sigma}$ changes by a non-Hamiltonian isotopy whose flux is $\alpha \cdot \mathrm{PD}([\sigma])$. If $[\sigma]\neq 0$ the result follows. 

If $[\sigma] = 0$, then $\sigma$ separates $\Sigma$ into two subsurfaces. 
Pick a simple closed curve $\gamma$ whose geometric intersection number 
with $\sigma$ is two, and let $\gamma'$ be a simple closed curve which is
Hamiltonian isotopic to $\tau_\sigma(\gamma)$ and intersects $\gamma$ and $\sigma$ minimally,
as in Figure \ref{Fig:separatingtwist}.

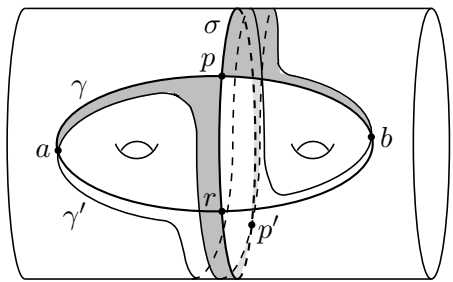
\begin{figure}[ht]
\begin{center} 
\begin{tikzpicture}[scale = 0.6]
\fill[gray!25!white] (0.1,-3) arc (-90:-42:0.9 and 3) arc (-42:-90:0.4 and 3);
\fill[gray!25!white] (1.3,3) arc (90:180:0.3 and 3) arc (0:-42:0.9 and 3)
arc (-42:90:0.4 and 3);
\fill[gray!50!white] (-3.5,-0.2) arc (187:87:3.5 and 1.5)
 arc (150:270:0.4 and 3) -- (0.1,-3) arc (270:180:0.5 and 3)
arc (0:60:0.6 and 1.15) -- (-1,1.15) arc (107:180:3.5 and 1.5);
\fill[gray!50!white] (3.5,0.2) arc (7:88:3.5 and 1.5)
arc (150:90:0.4 and 3) -- (1.3,3) arc (90:0:0.1 and 0.7) arc
(180:240:0.1 and 0.85) arc (65:0:3.5 and 1.5);
\draw[semithick] (-4.2,3) -- (4.8,3);
\draw[semithick] (-4.2,-3) -- (4.8,-3);
\draw[semithick] (4.8,0) ellipse (0.4 and 3);
\draw[semithick] (-4.2,3) arc (90:270:0.4 and 3);
\draw[semithick] (-2.2,0) arc (200:340:0.5);
\draw[semithick] (1.7,0) arc  (200:340:0.5);
\draw[semithick] (-2.07,-0.15) arc (150:30:0.4);
\draw[semithick] (1.83,-0.15) arc (150:30:0.4);
\draw[thick, dashed] (0.5,-3) arc (-90:90:0.4 and 3);
\draw[thick] (0.5,3) arc (90:270:0.4 and 3);
\draw[thick] ellipse (3.5 and 1.5);
\draw[semithick, rounded corners] (0.1,-3) arc (270:180:0.5 and 3) 
arc (0:60:0.6 and 1.15) -- (-1,1.15) arc (107:255:3.5 and 1.5) 
-- (-0.7,-2.3) arc (220:270:0.4 and 1.9);
\draw[semithick, dashed] (0.1,-3) arc (-90:0:0.9 and 3)
arc (180:90:0.3 and 3);
\draw[semithick, dashed] (-0.5,-3) arc (-90:0:0.9 and 3)
arc (180:90:0.4 and 3);
\draw[semithick, rounded corners] (1.3,3) arc (90:0:0.1 and 0.7) arc
(180:240:0.1 and 0.85) arc (65:-67:3.5 and 1.5) -- (1.1,-0.5) 
arc (0:90:0.3 and 3.5);
\draw (-3,1.2) node {$\gamma$};
\draw (-0.05,2.6) node {$\sigma$};
\draw (-3.1,-1.6) node {$\gamma'$};
\fill (-3.47,-0.15) circle (0.08);
\fill (3.47,0.15) circle (0.08);
\fill (0.16,1.5) circle (0.08); 
\fill (0.16,-1.5) circle (0.08);
\fill (0.82,-1.8) circle (0.08);
\draw (-3.8,-0.15) node {$a$};
\draw (3.8,0.15) node {$b$};
\draw (-0.15,1.8) node {$p$};
\draw (-0.1,-1.25) node {$r$};
\draw (1.2,-1.9) node {$p'$};
\end{tikzpicture}
\end{center}
\caption{Dehn twisting about a separating curve \label{Fig:separatingtwist}}
\end{figure}

By the work of Seidel \cite{Seidel:FCPLT}, there is an exact
triangle in $\scrF(\Sigma)$ taking the form
\[ \xymatrix{ CF(\sigma,\gamma)\otimes \sigma \ar[r]^-{e} & \gamma
\ar[r]^{f} & \ar@/^1.5pc/[ll] \gamma'} \\[1em] {} \]

where $e$ is a tautological evaluation map, which can be written
in terms of the generators $p,r$ of $CF(\sigma,\gamma)$ as 
$e=(p,r):\sigma\oplus \sigma[1]\to \gamma$, and $f\in CF^0(\gamma, \tau_\sigma(\gamma))$ is given
by the natural transformation $\Phi_\sigma$. Since $\sigma\oplus \sigma[1]$
is not isomorphic to $\gamma \oplus \gamma'[1]$, the morphism $f$ is a non-zero
linear combination of the generators $a,b$ of $CF^0(\gamma,\gamma')$,
and determined uniquely up to scaling by the property that 
$\mu^2(f,p)=\mu^2(f,r)=0$. 

Denoting by $A$ and $B$ the symplectic areas
of the two shaded triangles on Figure \ref{Fig:separatingtwist}, we find
that $\mu^2(a,p)=\pm q^A p'$ and $\mu^2(b,p)=\mp q^B p'$, so the vanishing of
$\mu^2(f,p)$ implies that $f$ is proportional to $q^B a + q^A b$.
(As expected, $\mu^2(f,r)$ then vanishes as well: since $\gamma'$ is Hamiltonian isotopic
to $\tau_\sigma(\gamma)$, the areas of the two triangles contributing to $\mu^2(f,r)$
differ by the same amount $B-A$).

Keeping $\gamma$ and $\gamma'$ fixed, when $\sigma$ moves by
an isotopy of flux $\alpha$ (without creating new intersections) the quantity $B-A$ changes by $\alpha$, so that
the class $[f]=\Phi_\sigma(\gamma)\in HF^0(\gamma,\tau_\sigma(\gamma))$ must change as well. 
This implies that $\Phi_\sigma$ detects the Hamiltonian isotopy class of $\sigma$.
\end{proof}

An object $Y \in \scrF(\Sigma)^{per}$ is \emph{spherical} if $H^*(hom_{\scrF^{per}}(Y,Y)) = H^*(S^1;\Lambda)$. A spherical object $Y$ has an associated twist functor $T_Y$.  

\begin{Lemma} \label{Lem:twists_forces_objects}
Let $Y, Y'$ be spherical objects on a surface $S$. Suppose the twist functors $T_Y$ and $T_{Y'}$ are quasi-isomorphic.  If $Y$ is a homologically non-trivial simple closed curve, then $Y$ and $Y'$ are quasi-isomorphic in $\Tw^{\pi}\scrF(S)$.
\end{Lemma}

\begin{proof}
Write $T$ for $T_Y \simeq T_{Y'}$. Let $\delta$ be a simple closed curve with geometric intersection number $1$ with $Y$. The group $HF(\delta, T(\delta))$ has rank one, so up to quasi-isomorphism there are only two distinct mapping cones
\[
Y \simeq \{\delta \stackrel{x}{\longrightarrow} T(\delta)\} \quad \mathrm{and} \quad \delta \oplus T(\delta)[1] \simeq \{\delta \stackrel{0}{\longrightarrow} T(\delta)\}.
\]
Since the functor $T=T_{Y'}$ is a cone over an evaluation functor which has image in the subcategory with objects $V\otimes Y'$ for graded vector spaces $V$, one of these mapping cones is isomorphic to direct sums of copies of $Y'$.  The sum $\delta \oplus T(\delta)[1]$ is a sum of two non-isomorphic indecomposables, so by the Krull-Schmidt property cannot be a sum of copies of a single indecomposable $Y'$.  Therefore $Y$ is isomorphic to a direct sum of copies of $Y'$, and indecomposability of $Y$ implies $Y \simeq Y'$ as required. 
\end{proof}

\subsection{Generation}

Consider an $A_{2g}$-chain of $2g$ curves $\{\zeta_i\}_{1 \leq i\leq 2g}$ on $\Sigma$, as depicted in Figure \ref{Fig:genus2} when $g=2$. 

\begin{figure}[ht]
\setlength{\unitlength}{1cm}
\begin{picture}(5,2)(0,-1)
\qbezier[200](0,0)(0,1.2)(1.5,1)
\qbezier[200](1.5,1)(2.5,0.85)(3.5,1)
\qbezier[200](3.5,1)(5,1.2)(5,0)
\qbezier[200](0,0)(0,-1.2)(1.5,-1)
\qbezier[200](1.5,-1)(2.5,-0.85)(3.5,-1)
\qbezier[200](3.5,-1)(5,-1.2)(5,0)
\qbezier[60](1,0)(1,0.3)(1.5,0.3)
\qbezier[60](2,0)(2,0.3)(1.5,0.3)
\qbezier[60](1,0)(1,-0.3)(1.5,-0.3)
\qbezier[60](2,0)(2,-0.3)(1.5,-0.3)
\qbezier[60](3,0)(3,0.3)(3.5,0.3)
\qbezier[60](4,0)(4,0.3)(3.5,0.3)
\qbezier[60](3,0)(3,-0.3)(3.5,-0.3)
\qbezier[60](4,0)(4,-0.3)(3.5,-0.3)
\put(1.5,0){\circle{1.1}}
\put(3.5,0){\circle{1.1}}
\qbezier[60](0,0)(0,-0.1)(0.5,-0.1)
\qbezier[60](1,0)(1,-0.1)(0.5,-0.1)
\qbezier[20](0,0)(0,0.1)(0.5,0.1)
\qbezier[20](1,0)(1,0.1)(0.5,0.1)
\qbezier[60](2,0)(2,-0.1)(2.5,-0.1)
\qbezier[60](3,0)(3,-0.1)(2.5,-0.1)
\qbezier[20](2,0)(2,0.1)(2.5,0.1)
\qbezier[20](3,0)(3,0.1)(2.5,0.1)
\put(-0.35,0){$\zeta_1$}
\put(0.8,0.5){$\zeta_2$}
\put(2.1,-0.35){$\zeta_3$}
\put(3.9,0.5){$\zeta_4$}
\end{picture}
\caption{Split-generating curves for $\scrF(\Sigma_2)$\label{Fig:genus2}}
\end{figure}
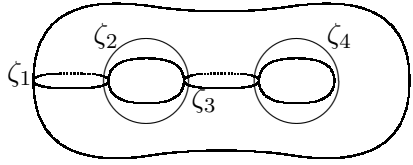

\begin{Proposition} \label{Prop:Split_generators}
The curves $\zeta_i$ split-generate $\scrF(\Sigma)$.
\end{Proposition}
\begin{proof}
This is a consequence of \cite[Proposition 3.8]{Smith:HFQuadrics}, a variant of 
\cite[Lemma 6.4]{Seidel:HMSgenus2}. Briefly, if $u = \prod_{i=1}^{2g} \tau_i$,  the Dehn twists $\tau_i$ in the curves $\zeta_i$ satisfy the positive relation $u^{4g+2} = 1 \in \Gamma_g$. The square of this relation defines a Lefschetz fibration $X \rightarrow \bP^1$ with fibre $\Sigma$, with $2g(8g+4)$ critical fibres, and for which every section of the fibration has square $\leq -2$.   

Let $\delta \in \scrF(\Sigma)$ be an arbitrary curve (equipped with a local system, which we suppress from the notation). There are exact triangles associated to the Dehn twists $\tau_i$,  on the cohomological category $H(\scrF(\Sigma))$ taking the form
\[
\xymatrix{ HF(\delta, \phi(\delta)) \ar[r]^-{p} & HF(\delta,\tau_i\circ\phi(\delta)) \ar[r] & \ar@/^1.5pc/[ll] HF(\delta,\zeta_i) \otimes HF(\zeta_i, \phi(\delta))} \\[1em]
{}
\]

for a subword $\phi = \prod_{j<i} \tau_{j}$ of the monodromy.   This triangle is induced from an exact triangle in $\scrF(\Sigma)$ of the form
\[
\xymatrix{ \phi(\delta)\ar[r]^-{p} & \tau_i\circ\phi(\delta)) \ar[r]   & \ar@/^1.5pc/[ll]  V\otimes \zeta_i} \\[1em]
\] 
where $V=HF(\zeta_i, \phi(\delta))$ is a $\bZ/2$-graded vector space and the arrow $p$ is multiplication by the section count $\Phi_{\tau_i}$. Concatenating such triangles for all the twists in $u^{8g+4}$ yields 
\begin{equation}\label{eqn:concatenate}
\delta \stackrel{\hat{p}}{\longrightarrow} u^{8g+4}(\delta) \cong \delta
\end{equation}
The morphism $\hat{p}$ counts sections of $X\rightarrow \bP^1$, and there are no holomorphic such for generic almost complex structure,  since all sections have square $\leq -2$ and live in moduli spaces of virtual dimension $<0$. Therefore the arrow $\hat{p}$ vanishes, and $\delta$ is exhibited as a summand in a triangle whose third entry is a twisted complex on the vanishing cycles $\zeta_i$.
\end{proof}

\begin{Corollary} \label{Cor:QH}
The closed-open map $H^*(\Sigma;\Lambda) \rightarrow HH^*(\scrF(\Sigma), \scrF(\Sigma))$ is an isomorphism; similarly for the open-closed map $HH_*(\scrF(\Sigma),\scrF(\Sigma)) \to H_*(\Sigma;\Lambda)$.
\end{Corollary}

\begin{proof}
This is a special case of \cite[Corollary 3.11]{Smith:HFQuadrics}, where the required hypotheses are obtained from Proposition \ref{Prop:Split_generators}.
\end{proof}

\begin{Corollary} \label{Cor:nice_split_generators}
Given a simple closed curve $\sigma \subset \Sigma$, there are curves $\{\xi_1,\ldots,\xi_{2g}\}$ which are split-generators for $\scrF(\Sigma)$ and which meet $\sigma$ minimally.
\end{Corollary}

\begin{proof}
Up to automorphism, there are only finitely many possibilities for $\sigma$, so the result follows by inspection of the pattern of curves in Figure \ref{Fig:genus2}.
\end{proof}

The same methods that underlie Corollary \ref{Cor:QH} also show:

\begin{Lemma} \label{Lem:proper_and_smooth}
If $\Sigma$ is a closed surface, then $\scrF(\Sigma)^{per}$ is homologically smooth and proper.
\end{Lemma}

\begin{proof}[Sketch] Properness  is immediate since the objects are closed Lagrangians. Smoothness follows from the fact that one can resolve the diagonal on $\Sigma\times\Sigma$ by product Lagrangians, which follows for instance from the argument of \cite[Section 3.4]{Smith:HFQuadrics}.  (An alternative is to use that $\scrF(\Sigma)^{per}$ is equivalent to a category of matrix factorizations of an isolated hypersurface singularity, and such categories are always smooth and proper\footnote{Also in the stronger sense of admitting a compact generator.}.)
\end{proof}

\subsection{Mukai pairing}

Let $\scrA$ be a proper (i.e. cohomologically finite) $A_{\infty}$-category, linear over $\bK$.  The \emph{Chern character} is a map
\[ ch: K_0(\scrA) \to HH_0(\scrA),\]
whilst the \emph{Mukai pairing}  is a graded bilinear pairing
\begin{align*}
\langle \bullet, \bullet \rangle: HH_*(\scrA) \otimes HH_*(\scrA) & \to \bK.
\end{align*}
These were introduced by Shklyarov in the case of $dg$-categories \cite{Shklyarov}; since any $A_\infty$ category is quasi-equivalent to a $dg$-category, their definitions and basic properties extend to the $A_{\infty}$-setting.  
Shklyarov proved that 
\begin{equation} \label{eqn:non-comm-rr}
\langle ch(X) , ch(Y) \rangle  = -\chi(X,Y).
\end{equation}

In the case of $\scrF(\Sigma)$, these notions are quite explicit, using
the open-closed map to identify $HH_0(\scrF(\Sigma))$ with $H_1(\Sigma;\Lambda)$.
If $\gamma$ is an unobstructed immersed curve
(equipped with a rank one local system), the absence of non-constant holomorphic discs
with boundary on $\gamma$ implies that the image of
$ch(\gamma)=[\id_\gamma]\in HH_0(\scrF(\Sigma))$ under the open-closed map is exactly
$[\gamma]\in H_1(\Sigma;\bZ)\subset H_1(\Sigma;\Lambda)$.
Comparing \eqref{eqn:non-comm-rr} with the classical identity
$\chi \, HF(\gamma,\gamma')=-[\gamma]\cdot [\gamma']$, we find that
the Mukai pairing is simply the intersection pairing on $H_1(\Sigma;\Lambda)$.

\begin{Lemma} \label{Cor:HH_integral}
Let $X \in \scrF(\Sigma)^{per}$. Then $ch(X) \in H_1(\Sigma;\Lambda)$ represents an integral class, i.e. $ch(X) \in \mathrm{image}\{H_1(\Sigma;\bZ) \to H_1(\Sigma;\Lambda)\}$.
\end{Lemma}

\begin{proof}
By \eqref{eqn:non-comm-rr}, $\langle ch(X), ch(\gamma)\rangle = -\chi \, HF(X,\gamma) \in \bZ$ for any simple closed curve $\gamma \subset \Sigma$.  
As noted above, under the open-closed map $ch(\gamma)$ maps to $[\gamma] \in H_1(\Sigma;\bZ)$.
It follows that $ch(X)$ has integral pairing with all of $H_1(\Sigma;\bZ)$, which is only possible if the class is integral.
\end{proof}



\begin{Corollary} Let $X \in \scrF(\Sigma)^{per}$ be spherical. There is a non-zero class $a\in H^1(\Sigma;\bZ)$ with $\langle a, ch(X)\rangle = 0$.
\end{Corollary}

\begin{proof} Evident from Lemma \ref{Cor:HH_integral}. \end{proof}

\subsection{Balancing\label{Sec:Monodromy}}

There is a natural map 
\[
\Symp(\Sigma) \to \mathrm{nu{\text-}fun}(\scrF(\Sigma),\scrF(\Sigma))
\]
which takes any Hamiltonian symplectomorphism to an equivalence which is quasi-isomorphic to the identity\footnote{The notation $\mathrm{nu{\text-}fun}$ follows \cite{Seidel:FCPLT}; the Fukaya category admits cohomological but not strict units, and the functors are therefore not strictly unital; they are however cohomologically unital.}. This yields a map $\Symp(\Sigma)/\Ham(\Sigma) \to \Auteq(D^{\pi}\scrF(\Sigma))$, where the domain is viewed as a discrete group.  For a surface of genus $\geq 2$, 
\[
\Symp(\Sigma)/\Ham(\Sigma) = H^1(\Sigma;\bR) \rtimes \Gamma(\Sigma)
\]
by Moser's theorem and the vanishing of the flux group.
To build a homomorphism $\Gamma_g \to \Auteq(D^{\pi}\scrF(\Sigma_g))$ requires some additional choice.  Suppose $g\geq 2$, and fix a primitive $\theta$ for the pullback of $\omega_{\Sigma}$ to the unit tangent bundle $S(T\Sigma)$. For any simple closed curve $\sigma \subset \Sigma$, a choice of orientation of $\sigma$ defines a canonical lift $\sigma \subset S(T\Sigma)$, and we then have a real number $t_{\sigma} = \int_{\sigma} \theta$.  Say $\sigma$ is balanced if $t_{\sigma} = 0$, and define a \emph{balanced} symplectomorphism $f: \Sigma \to \Sigma$ to be one which takes balanced curves to balanced curves, i.e. for which $t_{f(\sigma)} = t_{\sigma}$ for every oriented simple closed curve $\sigma \subset \Sigma$.  

\begin{Lemma}
$\{\mathrm{Balanced \ symplectomorphisms}\} / \Ham(\Sigma) \ \simeq \ \Gamma_g$.
\end{Lemma}

\begin{proof} See \cite{Seidel:HMSgenus2}. \end{proof}

It follows that the choice of $\theta$ defines a map 
\begin{equation} \label{eqn:exists}
\Gamma_g \to \Auteq(D^{\pi}\scrF(\Sigma_g)).
\end{equation} 

Given two primitives $\theta, \theta'$ for $p^*\omega_{\Sigma}$, with $p: S(T\Sigma) \to \Sigma$, one obtains a class $[\theta-\theta'] \in H^1(S(T\Sigma);\bR)$. 
Note that, since $\chi(\Sigma)\neq 0$, the pullback $p^*$ induces an isomorphism on first cohomology, so we can think
of $[\theta-\theta']$ as an element of $H^1(\Sigma;\bR)$. 
Changing $\theta$ to $\theta'$ conjugates the image of \eqref{eqn:exists}  by the action of this element of $H^1(\Sigma;\bR)$.

\begin{Remark} The analogous construction for punctured surfaces may be more
familiar to the reader: the choice of an exact symplectic structure on a
punctured surface $S$ (i.e.\ a primitive $\theta$ for the
symplectic form itself, rather than its lift to the unit tangent bundle)
determines a homomorphism from the mapping class group of $S$ to $\Auteq(D^\pi\scrF(S))$
by considering exact symplectomorphisms of $S$, i.e.\ those which take exact curves to exact
curves, up to Hamiltonian isotopy.
\end{Remark}

\subsection{Cautionary examples}

Despite its generally elementary character, there are some surprises in Floer theory for curves on surfaces.

\begin{figure}[ht]
\begin{center} 
\begin{tikzpicture}[scale = 0.6]

\draw[semithick] ellipse (4.5 and 3);
\draw[semithick] (-2,0) arc (200:340:0.5);
\draw[semithick] (1,0) arc  (200:340:0.5);
\draw[semithick] (-1.87,-0.15) arc (150:30:0.4);
\draw[semithick] (1.13,-0.15) arc (150:30:0.4);
\draw[fill] (-3,0.5) circle (0.1); 
\draw[fill] (-2.5,0.5) circle (0.1);
\draw[semithick] (-3,-0.55) -- (-3,1.5);
\draw[semithick] (-2.5,-1) -- (-2.5,0.7);
\draw[semithick, rounded corners] (-3,1.3) -- (-3,2) -- (2,2) -- (2,1) -- (0.2,1) -- (-2.5,1) -- (-2.5,0.7);
\draw[semithick,rounded corners] (-2.5,-0.9) -- (-2.5,-1.1) -- (0.5,-1.1) -- (0.5,0.5) -- (-4.4,0.5);
\draw[semithick, rounded corners, dashed] (-4.4,0.5) -- (-4,0.7) -- (0.2,0.7) -- (0.2,-0.8) -- (-4.3,-0.8);
\draw[semithick, rounded corners] (-4.3,-0.8) -- (-4,-0.6) -- (-3,-0.6) -- (-3,-0.5);
\draw (0.5,1.5) node {$C$};
\draw (-3.5,0) node {$A$}; 

\end{tikzpicture}
\end{center}
\caption{An exotic spherical object when $C>A$\label{Fig:bad}}
\end{figure}
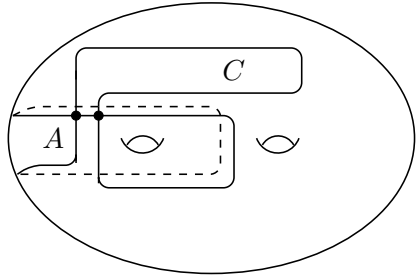

\begin{Lemma} \label{Lem:bad} There are spherical objects in $\scrF(\Sigma)$ which are not quasi-isomorphic to any simple closed curve with local system.
\end{Lemma}

\begin{proof}[Sketch]
See Figure \ref{Fig:bad}, which shows an immersed curve obtained by pushing a separating simple closed curve through itself to create a single bigon.  The region labelled $C$ contributes to a non-trivial Floer differential, so this immersed curve is spherical. However, the hypothesis that the area $C > A$ implies that one cannot deform the curve to be embedded through an isotopy which has trivial flux: the end result would have to separate $\Sigma$ into two regions, one of area $A-C < 0$.  It is not hard to see that, if this immersed curve was quasi-isomorphic to a simple closed curve, that curve would have to be in the same homotopy class, and then consideration of Lemma \ref{Lem:determines} would show that the natural transformation from the identity to
the associated twist functor would involve a different linear combination of Floer generators.
\end{proof}

\begin{Lemma} There is an immersed curve $S^1 \to \Sigma$ with rank one local system which admits non-trivial idempotents.
\end{Lemma}

\begin{proof} Take an immersed curve $\gamma$ which is homotopic to the double cover of a simple closed curve $\sigma$. Then a rank one local system on $\gamma$ defines an object quasi-isomorphic to a rank two local system $\xi$ on $\sigma$, and rank one sub-local-systems of $\xi$ define idempotents.
\end{proof}

\begin{Lemma} \label{Lem:not_cone} Let $p$ be a transverse intersection point of curves $\gamma,\gamma'$ such that $p$ is a Floer cocycle in $CF^*(\gamma,\gamma')$. The immersed curve resulting from surgery at $p$ need not be quasi-isomorphic to the mapping cone $\gamma \stackrel{p}{\longrightarrow} \gamma'$.
\end{Lemma}

\begin{figure}[ht]
\begin{center} 
\begin{tikzpicture}[scale=0.8]

\draw[semithick] (-3,0) -- (3,0);
\draw[semithick] (-1,-1) parabola (-2,1.5);
\draw[semithick] (-1,-1) parabola (0,0);
\draw[semithick] (1,1) parabola (0,0);
\draw[semithick] (1,1) parabola (2, -1.5);
\draw[fill] (0,0) circle (0.1);
\draw[fill] (-1.65,0) circle (0.1);
\draw[fill] (1.65,0) circle (0.1);
\draw (-2,-0.3) node {$r_2$};
\draw (0.2,-0.3) node {$p$};
\draw (2,-0.3) node {$r_1$};
\draw (-3.3,0) node {$\gamma$};
\draw (1.6,-1.3) node {$\gamma'$};

\draw[semithick, rounded corners] (4,0) -- (6.7,0) -- (7,0.3);
\draw[semithick, rounded corners] (10,0) -- (7.3,0) -- (7,-0.3);
\draw[semithick] (8,1) parabola (7,0.3);
\draw[semithick] (8,1) parabola (9,-1.5);
\draw[semithick] (6,-1) parabola (7,-0.3);
\draw[semithick] (6,-1) parabola (5,1.5);
\draw[fill] (5.35,0) circle (0.1);
\draw[fill] (8.65,0) circle (0.1);
\draw (5,-0.3) node {$r_2$};
\draw (9,-0.3) node {$r_1$};
\draw (10.3,0) node {$\sigma$};

\end{tikzpicture}
\end{center}
\caption{The mapping cone is not the surgery\label{Fig:bad_surgery}}
\end{figure}
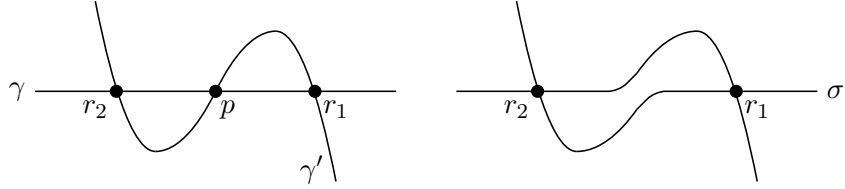

\begin{proof}
Consider Figure \ref{Fig:bad_surgery}, where in the first image the Floer complex $CF^*(\gamma,\gamma')$ has 
\[
dr_1 = q^{\alpha_1} \cdot p, \ dr_2 = -q^{\alpha_2} \cdot p
\]
with $\alpha_1,\alpha_2$ the areas of the bigons; we assume there are no holomorphic strips other than those in the picture. Then $p$ is an \emph{exact} Floer cocycle, so 
\[
\mathrm{Cone}(p) \simeq \gamma[1] \oplus \gamma'.
\]
In this case, the Lagrange surgery  $\sigma$ at $p$ is quasi-isomorphic to 
the cone on $q^{\alpha_1}r_1$ viewed as a (closed)
morphism in the reverse direction, from $\gamma'$ to $\gamma$ 
(or equivalently, $q^{\alpha_2} r_2$, which is cohomologous). However,
there are also examples where the Lagrange surgery at an exact Floer cocycle
between a pair of simple closed curves $\gamma,\gamma'$ yields an immersed
curve $\sigma$ which lies outside of their triangulated 
envelope (Figure \ref{Fig:bad_surgery2}). In all these examples, the surgered curve $\sigma$ remains cobordant to $\gamma \cup \gamma'$,
but the Lagrangian cobordism between them is obstructed, and the cobordism
only yields an exact triangle
after deforming $\sigma$ by a suitable bounding cochain (which amounts geometrically to smoothing
a self-intersection of $\sigma$, to obtain a curve
in a different homotopy class -- in fact, homotopic to $\gamma \cup \gamma'$ rather than $\sigma$).
This is a purely one-dimensional phenomenon -- for instance, in higher dimensions, there would be no rigid strip passing through the neck region after surgery, in contrast to the visible strip in the right image of Figure \ref{Fig:bad_surgery}.
\end{proof}

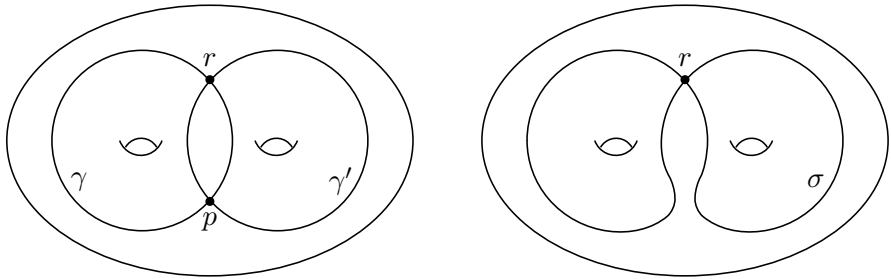
\begin{figure}[ht]
\begin{center} 
\begin{tikzpicture}[scale = 0.6]
\draw[semithick] ellipse (4.5 and 3);
\draw[semithick] (-2,0) arc (200:340:0.5);
\draw[semithick] (1,0) arc  (200:340:0.5);
\draw[semithick] (-1.87,-0.15) arc (150:30:0.4);
\draw[semithick] (1.13,-0.15) arc (150:30:0.4);
\draw[semithick] (-1.5,0) ellipse (2 and 2);
\draw[semithick] (1.5,0) ellipse (2 and 2);
\draw (-2.9,-0.9) node {$\gamma$};
\draw (2.9,-0.9) node {$\gamma'$};
\fill (0,-1.35) circle (0.1);
\fill (0,1.35) circle (0.1);
\draw (0,-1.8) node {$p$};
\draw (0,1.8) node {$r$};
\end{tikzpicture}
\qquad
\begin{tikzpicture}[scale = 0.6]
\draw[semithick] ellipse (4.5 and 3);
\draw[semithick] (-2,0) arc (200:340:0.5);
\draw[semithick] (1,0) arc  (200:340:0.5);
\draw[semithick] (-1.87,-0.15) arc (150:30:0.4);
\draw[semithick] (1.13,-0.15) arc (150:30:0.4);
\draw[semithick, rounded corners] (-3.5,0) arc (180:-30:2) -- (0.214,-1.53) arc
(-130:210:2) -- (-0.214,-1.53) arc (310:180:2);
\draw (2.9,-0.9) node {$\sigma$};
\fill (0,1.35) circle (0.1);
\draw (0,1.8) node {$r$};
\end{tikzpicture}

\end{center}
\caption{The mapping cone differs from the surgery by a bounding cochain\label{Fig:bad_surgery2}}
\end{figure}

\begin{Remark}
In the situation of Lemma \ref{Lem:not_cone}, if $\gamma \cup \gamma'$ bounds no (immersed) bigons, then the mapping cone and the surgery do agree, cf. \cite[Section 5]{Abouzaid:higher_genus}.
\end{Remark}

\section{Immersed curves and bigons}\label{sec:bigons}

\subsection{Bigons}

In this section, $S$ denotes a compact surface which may have empty or non-empty boundary.  Let $\gamma \subset S$ be an
unobstructed immersed curve, which we always assume has only transversal self-intersections. An \emph{embedded bigon} with boundary on $\gamma$ is a map $u: D \to S$ from the closed disc to $S$ which takes $\pm 1 \in \partial D$ to self-intersection points of $\gamma$, which takes the boundary $\partial D \backslash \{\pm 1\}$  to  $\gamma$, and which is an embedding $D \to S$.  Note that there will in general be arcs of $\gamma$ which meet the interior of the bigon. If there are no such arcs, so $\gamma \cap u(D) = u(\partial D)$, then we say the bigon is \emph{empty}. 

\begin{Lemma}
An immersed curve $\gamma \subset S$ bounds at most finitely many embedded bigons.
\end{Lemma}

\begin{proof} For a given pair of intersection points, there are only finitely many possible boundary arcs in $\gamma$ between them.  A pair of arcs which cuts out a disc in $S$ defines a unique bigon.
\end{proof}

\begin{Lemma}[Hass, Scott]
Let  $\gamma \subset \Sigma$ be an immersed closed curve which does not
bound any teardrops. If $\gamma$ is homotopic to a simple closed curve, but is not embedded, then $\gamma$ bounds an embedded bigon.
\end{Lemma}

\begin{proof} This is proved in \cite{Hass-Scott}. \end{proof}

Hass and Scott show by examples that one can not in general assume that $\gamma$ bounds an empty bigon. We introduce a combinatorial move on immersed curves:
\begin{itemize}
\item The \emph{triple-point move} changes a configuration of three intersecting arcs which cut out a small downwards-pointing triangle to one defining a small upwards-pointing triangle, cf. Figure \ref{Fig:triangle_move}.
\end{itemize}

\begin{figure}[ht]
\begin{center} 
\begin{tikzpicture}[scale=0.6]
\draw[semithick] (2,-2) -- (-2,2);
\draw[semithick] (-2,-2) -- (2,2);
\draw[semithick] (-2.5,0) arc (135:45:3.7);

\draw[semithick] (8,-2) -- (4,2);
\draw[semithick] (4,-2) -- (8,2);
\draw[semithick] (3.5,0) arc (-135:-45:3.7);
\end{tikzpicture}
\end{center}
\label{triple_point_via_bigons}
\caption{The triple-point move\label{Fig:triangle_move}}
\end{figure}
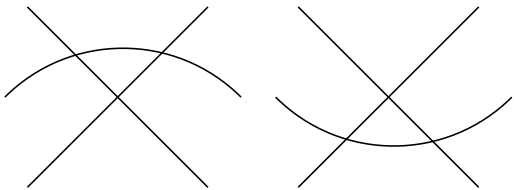

\begin{Lemma} The triple-point move preserves the total number of self-intersections of $\gamma$, and does not increase the total number of embedded bigons on $\gamma$.  
\end{Lemma}

\begin{proof} Evident.  \end{proof}

\begin{Lemma}[Steinitz]
If $\gamma$ bounds an embedded bigon, then after a finite sequence of triple-point moves, $\gamma$ bounds an empty bigon.
\end{Lemma}

\begin{proof} This is proved in \cite{Steinitz, Steinitz2}, see also \cite{Schleimer_et_al}. \end{proof}

\subsection{Removing bigons\label{Sec:bigons}}

Let $v: \gamma \to S$ be an immersion with $[v(\gamma)] \in H_1(S;\bZ)$ a non-zero class.

%
%

\begin{Lemma} \label{Lem:shrink_area}
If $v(\gamma)$ bounds an empty embedded polygon, then there is a 
Hamiltonian isotopy of $v(\gamma)$ which decreases the area of this polygon 
to be arbitrarily small without creating any self-intersections.
However, if $S$ has non-empty boundary and $[v(\gamma)]$ vanishes in
$H_1(S,\partial S;\bZ)$ then the isotopy may require enlarging the surface $S$.
\end{Lemma}

\begin{proof} 
Sliding the portions of $v(\gamma)$ that bound the empty polygon gives
a smooth isotopy that decreases its area as required. All that is
required, then, is to correct the isotopy by the flow of a symplectic vector
field on $S$ in order to ensure that it sweeps zero area (so that it is induced by a 
Hamiltonian on the domain of the immersion). If $v(\gamma)$ represents a
non-zero class in $H_1(S,\partial S;\bZ)$ then this can be achieved exactly
as in  Lemma \ref{Lem:space}, by considering a simple closed curve $\sigma$
with non-zero algebraic intersection number with $v(\gamma)$ and sliding
around a thin cylinder centred on $\sigma$. If $[v(\gamma)]$ vanishes in
relative cohomology, then instead we find a properly embedded arc with
non-zero algebraic intersection number with $v(\gamma)$ and slide along it; this may push
$v(\gamma)$ outside of $S$ and require us to enlarge the surface.
\end{proof}

\begin{Lemma} \label{Lem:bigon-quasiiso}
If $\gamma$ bounds an empty bigon, then $\gamma$ is quasi-isomorphic to
a curve $\gamma'$, with two fewer self-intersections, obtained by cancelling the bigon. 
\end{Lemma}

\begin{proof}
The previous Lemma shows that we can deform $\gamma$ by a Hamiltonian
isotopy to decrease the area of the bigon to be arbitrarily small. 
Once the bigon is sufficiently small, we can cancel the pair of
self-intersections by a regular homotopy that sweeps zero area, without
creating any other intersections; the result then follows from
Lemma \ref{Lem:hamiltonian_isomorphic}.
\end{proof}

\begin{Lemma} \label{Lem:triangle-quasiiso}
If $\gamma$ bounds an empty triangle, then $\gamma$ is
quasi-isomorphic to a curve $\gamma'$ obtained by performing a triple point
move.
\end{Lemma}

\begin{proof} The argument is the same as for bigons: Lemma
\ref{Lem:shrink_area} shows we can shrink the triangle to have arbitrarily
small area by an isotopy that sweeps zero area, and once the triangle is
sufficiently small we can perform the triple point move by a regular
homotopy that sweeps zero area. The result then follows from
Lemma \ref{Lem:hamiltonian_isomorphic}.
\end{proof}

\subsection{Analyticity of Floer cohomology\label{Sec:analyticity}}

Let $(X,\omega)$ be a symplectic manifold and $\{F_b\}_{b\in B}$ a family of 
unobstructed Lagrangian submanifolds parametrized by a smooth manifold $B$.  (The prototypical situation in the literature would be that $B$ is a subset of a tropical SYZ base, and we will sometimes refer to the $F_b$ as fibres.) Recall that $\xi \to F_b$ denotes a rank one $U_{\Lambda}$-local system over $F_b$. 

Over a small disk $b_0\in P\subset B$, the fibres $F_b$ are graphs of closed one-forms $\alpha_b$ over $F_{b_0}$;  the Hamiltonian isotopy class of $F_b$ depends only on the de Rham cohomology class of the one-form $\alpha_b$. The space of choices $(b,\xi\to F_b)$ is therefore naturally a  domain centred on the base-point $b_0=(0,1)$ inside $(\Lambda^*)^k = H^1(F_{b_0};\Lambda^*) = H^1(F_{b_0};\bR) \times H^1(F_{b_0}; U_{\Lambda})$, where $k= \rk_{\bZ}\,H^1(F_{b_0};\bZ)$.
To be more explicit, fixing a basis $a_1,\dots,a_k$ for $H_1(F_{b_0};\bZ)$, we have
$\Lambda^*$-valued co-ordinates $z_1,\dots,z_k$ given by
$z_i(b,\xi)=q^{[\alpha_b]\cdot a_i} \mathrm{hol}_\xi(a_i)$. More
intrinsically, to every element
$\gamma=\sum \gamma_i a_i\in H_1(F_b;\bZ)$ corresponds a monomial
$$z_{b,\xi}^\gamma=z_1^{\gamma_1} \dots z_k^{\gamma_k}=q^{[\alpha_b]\cdot \gamma}\, \mathrm{hol}_\xi(\gamma).$$

Consider a Lagrangian submanifold $L\subset X$ transverse to the fibres $F_b$ over $P$.
Over $P$, $L$ defines an unbranched cover of $P$, so the intersection points $\{L \pitchfork F_b\}$ may be identified with $L \pitchfork F_{b_0}$. Fix a local section of the family $\{F_b\}$, so that for each fibre $F_b$ we can fix a smoothly varying base-point $\star \in F_{b}$. In $F_{b_0}$, fix a (homotopy class of) path $\gamma_x$ from $x$ to $\star$, for each $x \in L\pitchfork F_{b_0}$.   We fix a rank one local system $\xi \to F_{b_0}$ and an arbitrary identification
$\xi_{\star} \cong \Lambda$; by parallel transport along $\gamma_x$, this identifies $\xi_x \cong \Lambda$ for each intersection
$x\in L \cap F_{b_0}$.

Now consider another point $b \in P$. Assuming $P$ is convex, over a path joining $b$ and $b_0$ the path $\gamma_x$ sweeps a two-chain $\Gamma_x$, yielding an area $a_x(b) = \int_{\Gamma_x} \omega$.  By a trick due to Fukaya (essentially the observation that taming is an open condition on almost complex structures), a rigid holomorphic strip $u$ with boundaries on $L$ and $F_{b_0}$ will deform, for $P$ sufficiently small, to a rigid holomorphic strip $u'$ with boundaries on $L$ and $F_b$. By concatenating the boundary arcs of $u'$ with the reference paths, we get a well-defined element $[\partial u'] \in H_1(F_b;\bZ)$. 

\begin{Lemma} \label{Lem:analytic}
Suppose the arc of the boundary of $u$ connects intersection points $x,y$.  We have 
\[
q^{E(u')} \mathrm{hol}_{\xi}(\partial u') = q^{a_y(b)-a_x(b)} \cdot q^{E(u)} z_{b,\xi}^{[\partial u']}.
\]
\end{Lemma}

\begin{proof}
See \cite{Abouzaid:ICM}. 
\end{proof}

The key point is that this expression involves the monomial $z_{b,\xi}^{[\partial u']}$,
whereas the quantity $q^{E(u)}$ is \emph{constant}, depending only on the reference point $b_0$.  In particular, under the rescaling
\[
CF^*(L,(\xi\to F_{b_0})) \longrightarrow CF^*(L, (\xi \to F_b)), \qquad x \mapsto q^{a_x(b)}\cdot x = x'
\]
one finds that the Floer differential becomes analytic as a function of the co-ordinates $z_{b,\xi}$; schematically, 
\[
\langle \mu^1(x'), y'\rangle \ = \ \sum_{[u]} \# \mathcal{M}([u]) \cdot q^{E([u])}
z_{b,\xi}^{[\partial u]}.
\]
The dependence on the choice of base-point $b_0$ and of the homotopy classes of the paths $\gamma_x$ are also analytic; the former changes the values $a_x(b)$ by some fixed constant, rescaling $x'$ by a value which does not depend on
$z_{b,\xi}$, whilst the latter rescales $x'$ by a monomial.

\begin{Corollary} \label{Cor:analytic}
There is an affinoid neighbourhood $P$ of $b_0 \in (\Lambda^*)^k$ such that the Floer cohomology groups $HF^*(L,(\xi\to F_b))$ are the fibres of an analytic sheaf over $P$.
\end{Corollary}

The same argument would apply to a family of immersed Lagrangians equipped with analytically varying bounding cochains.

Up to this point, we have assumed that the Lagrangian $L$ is transverse to all the fibres $\{F_b\}_{b\in P}$.  To obtain a more global statement, given a family of Lagrangians $\{F_b\}_{b\in B}$ one chooses a finite set of Hamiltonian perturbations $L_i$ of $L$ for which the corresponding caustics of the projections $L_i \to B$ have empty intersection.  This finite set of Hamiltonian perturbations can be spanned by a simplex of Hamiltonian perturbations, and there are (higher) continuation maps on Floer cochains associated to the edges and higher-dimensional facets of this simplex.  These yield a module for a Cech complex of the corresponding covering of $B$, and (for a sufficiently fine cover) gluing maps of the local analytic sheaves over affinoids in $B$ constructed previously.  Summing up:

\begin{Theorem}[Abouzaid] \label{Thm:Floer_analytic}
Let $\{F_b\}_{b\in B}$ be a family of unobstructed Lagrangians pa\-ram\-e\-trized by an open $B \subset H^1(F_b;\bR)$. There is an analytic $dg$-sheaf over the annulus $B\times U_{\Lambda} \subset H^1(F_b;\Lambda^*)$ with stalk quasi-isomorphic to $CF^*(L, (\xi \to F_b))$.  Furthermore, this association yields an $A_{\infty}$-functor from the tautologically unobstructed Fukaya category of $(X,\omega)$ into the $dg$-category of complexes of sheaves over $B\times U_{\Lambda}$.
\end{Theorem}

\begin{proof} See \cite{Abouzaid:FFFF}. \end{proof}

\begin{Remark} 
Suppose we fix a class $a \in H^1(F_b;\bR)$ and consider the corresponding real one-parameter family of Lagrangians $F_b^t$ given by moving $F_b$ by symplectomorphisms of flux $t\cdot a$. Then for a test Lagrangian $L$, there is a discrete set of values of $t\in\bR$ where the transversality condition $L \pitchfork F_b^t$ fails. Picking Hamiltonian perturbations, the Cech complex above amounts to a zig-zag diagram of quasi-isomorphisms associated to neighbouring such intervals; generically, these quasi-isomorphisms  furthermore correspond to the simplest birth-death bifurcations of Floer complexes, adding or subtracting an acyclic subcomplex with two generators and a unique minimal area bigon.  
\end{Remark}

\subsection{Spherical objects}
Let $S$ be a surface which may be closed or have non-empty boundary. Let $\gamma \subset S$ be an
unobstructed immersed curve and $\xi \to \mathrm{domain}(\gamma)$ a local system.  Equip $S$ with a symplectic form $\omega$ and compatible complex structure $j$.

\begin{Lemma} If $(\xi,\gamma)$ defines a spherical object of $\scrF(S)$, then $\gamma$ is
regular homotopic to a simple closed curve.
\end{Lemma}

\begin{proof} 
Equipping $S$ with a hyperbolic metric, it is a classical fact that there
is a unique geodesic $\eta$ in the homotopy class of $\gamma$.
In fact, performing the homotopy by a suitable
curve-shortening flow (see e.g.\ \cite{Hass-Scott-Shortening}), one finds
that $\gamma$ is regular homotopic to $\eta$ among generically 
self-transverse unobstructed immersed curves. We claim that $(\xi,\gamma)$ being
spherical implies that the geodesic $\eta$ is a simple closed curve.

Recall that the Floer complex $CF^*((\xi,\gamma),(\xi,\gamma))$ splits
into a direct sum of complexes corresponding to the various 
lifts of $\gamma$ to the universal cover of $S$ which have non-trivial
intersection with a fixed lift $\tilde{\gamma}$.
(This is because Floer generators which correspond to intersections between
different pairs of lifts of $\gamma$ cannot be connected by bigons.)
The summand which corresponds to the trivial homotopy class (i.e., intersections of
$\tilde{\gamma}$ with a small Hamiltonian
perturbation of itself) contributes $H^*(S^1;\mathrm{End}(\xi))$, which
has rank at least two (considering the identity endomorphism of $\xi$).
Thus, all other summands must be acyclic.

If the geodesic $\eta$ is multiply covered, then by homotoping $\gamma$
to $\eta$, then translating along the underlying simple geodesic, and
homotoping back, we find that there is a regular homotopy of $\gamma$ to
itself which sweeps zero area and turns the chosen lift $\tilde{\gamma}$ 
into a different lift. Hamiltonian isotopy invariance (cf.\ the proof of
Lemma \ref{Lem:hamiltonian_isomorphic})
then implies that this pair of lifts also contributes non-trivially to the
Floer cohomology of $(\xi,\gamma)$ to itself, which contradicts the
assumption that $(\xi,\gamma)$ is spherical. Hence $\eta$ cannot be multiply
covered.

Finally, observe that each lift of $\gamma$ to the universal cover of $S$
lies within bounded distance of a lift of the geodesic $\eta$.
If two lifts of $\eta$ intersect each other (necessarily at
a unique point, since they are hyperbolic geodesics), then the 
corresponding lifts of $\gamma$ also have algebraic intersection number equal to 1, 
and hence they must contribute non-trivially to Floer cohomology, which 
contradicts the assumption. Thus, the lifts of $\eta$ to the universal 
cover of $S$ are pairwise disjoint, which implies that $\eta$ is embedded.
\end{proof}

\begin{Lemma} \label{lem:rank_one} Let $\xi \rightarrow \gamma$ be an indecomposable local system over a simple closed curve $\gamma$. The endomorphism ring $H^*(\hom_{\scrF}((\xi,\gamma), (\xi,\gamma))) = H^*(\gamma; \End(\xi))$ has rank $2$ if and only if $\xi$ has rank $1$.
\end{Lemma}

\begin{proof}
An indecomposable rank $r$ local system $\xi$ is determined by its monodromy $A$.  The fibre $\Lambda^r$ is cyclic as
a $\Lambda[t^{\pm 1}]$-module (where $t$ acts by $A$), generated by any vector, and can thus be
identified with $\Lambda[t^{\pm 1}] / \langle \chi_A(t)\rangle$, where $\chi_A$ is the characteristic polynomial of the monodromy.  
Then $H^0(\End(\xi))$ contains the maps $\Lambda^r \to \Lambda^r$ which commute with the monodromy, 
i.e.\ $\Lambda[t^{\pm 1}]$-module maps. Since module maps
\[
\Lambda[t^{\pm 1}]/(\chi_A(t)) \longrightarrow \Lambda[t^{\pm 1}]/(\chi_A(t))
\]
are determined by the image of 1, which can be any element of
$\Lambda[t^{\pm 1}]/(\chi_A(t))\simeq \Lambda^r$, we conclude that
$H^0(\End(\xi))$ has rank $r$ over $\Lambda$; the same is then true for $H^1$ by considering Euler characteristic.
\end{proof}

\begin{Corollary}  \label{Cor:immersed_spherical} If $(\xi,\gamma) \subset S$ is an immersed curve with local system which defines a spherical object $X \in \scrF(S)$, and if $[\gamma] \in H_1(S;\Z)$ is non-zero, then $\gamma$ is quasi-isomorphic to an embedded simple closed curve and $\xi$ has rank one.
(If $S$ has non-empty boundary and $[\gamma]$ vanishes in $H_1(S,\partial S;\bZ)$ then we may need to enlarge $S$.) \end{Corollary}

\begin{proof} Being spherical implies that $\gamma$ is homotopic to a simple closed curve, hence bounds an embedded bigon.  By a sequence of triple point moves, we may find an empty bigon bound by $\gamma$, which we may then shrink by Hamiltonian isotopy and cancel to obtain a new immersed curve $\gamma'$ in the same quasi-isomorphism class. 
By repeated applying Lemmas \ref{Lem:bigon-quasiiso} and
\ref{Lem:triangle-quasiiso}, we eventually arrive at a simple closed curve.
Finally, Lemma \ref{lem:rank_one} implies that $\xi$ must have rank one.
\end{proof}

 \section{Geometrization on punctured surfaces}

\subsection{The wrapped Fukaya category of a non-compact surface}

Let $(S,\partial S)$ be a symplectic surface with non-empty boundary; fix a finite subset $\Lambda \subset \partial S$ of
boundary marked points (the ``stops'') and a
homotopy class of line field $\eta\subset TS$. 
Associated to this data is a $\bZ$-graded partially wrapped Fukaya 
category $\scrW(S,\Lambda,\eta)$.
If one does not make a choice of line field, there is also a $\bZ/2$-graded category $\scrW(S,\Lambda)$, where the $\bZ/2$-grading is given by orientation. When the set $\Lambda \subset \partial S$ of stops is empty, we will simply write $\scrW(S)$ or $\scrW(S,\eta)$. 

If the line field $\eta$ is orientable (i.e. lifts from a section of $\bP(TS)$ to the unit sphere bundle of $TS$), there is a forgetful functor $\scrW(S,\Lambda,\eta) \to \scrW(S,\Lambda)$ which forgets the grading structure. There are also  localization functors $\scrW(S,\Lambda) \to \scrW(S,\Lambda')$  which forget some of the stops whenever $\Lambda' \subset \Lambda$; in particular there are ``acceleration" functors $\scrW(S,\Lambda) \to \scrW(S)$ which are (by definition) surjective on objects.


The objects of $\scrW(S,\Lambda)$ (resp.\ $\scrW(S,\Lambda,\eta)$) are (graded) unobstructed properly
immersed curves or arcs with boundary in $\partial S\setminus \Lambda$,
equipped with local systems. While the construction is usually carried out
in the exact setting, we work over the Novikov field and allow non-exact
objects into our category. Recall that generators of the wrapped Floer
complex arise not only from intersection points but also from
(positively oriented) boundary chords in $\partial S\setminus \Lambda$
connecting the end points of a pair of arcs.
Fixing a Liouville structure on $S$, the structure maps of the wrapped
category count isolated solutions of Floer's equation with a suitable
Hamiltonian perturbation in the Liouville completion of $S$. These counts are
weighted by the topological energy of the
solutions and by holonomy terms. In the case of
arcs, the weights can be cancelled out by trivializing the local systems and rescaling generators by
their Floer action; it is only for non-exact objects that
Novikov coefficients are necessary.  Note that
the wrapped category only depends on the Liouville completion of
$(S,\Lambda)$; in particular it is independent of the choice of Liouville 
structure.

\begin{Remark}
As in the case of Fukaya categories of closed Riemann surfaces,
the structure maps of $\scrW(S,\Lambda)$ can be determined
combinatorially in terms of immersed polygons with convex corners.
Indeed, solutions to Floer's equation which
do not lie entirely in the cylindrical ends of the completion can be
reinterpreted as immersed polygons in $S$ whose
boundary lies partially on the given Lagrangians and partially
along chords in $\partial S\setminus \Lambda$; and solutions which
lie entirely in the cylindrical ends only contribute a ``classical'' term to
$\mu^2$ which concatenates two boundary chords with a common end point.
\end{Remark}
 
 \subsection{Geometricity of twisted complexes}

We will say that an object $Y$ of $\Tw^{\pi}\scrW(S,\Lambda)$ is \emph{geometric} if it is quasi-isomorphic to a union of  immersed arcs or curves with local systems
in $S$ or its Liouville completion.   The category $\Tw\scrW(S,\Lambda,\eta)$ has a combinatorial model, due to \cite{HKK}, using which they prove the remarkable:

\begin{Theorem}[Haiden, Katzarkov, Kontsevich]  \label{Thm:HKK} Let $(\Sigma,\partial \Sigma, \Lambda)$ be a surface with non-empty boundary and a (possibly empty) collection of boundary marked points $\Lambda \subset \partial\Sigma$. Let $Y \in \Tw\scrW(\Sigma,\Lambda)$ be a $\bZ/2$-graded twisted complex. Then $Y$ is geometric.
\end{Theorem}

We include a brief discussion of the proof, to illustrate why its ingredients do not readily generalise to the case of closed surfaces treated in this paper, and to clarify its applicability to the $\bZ/2$-graded case
and to the non-exact setting.

\begin{proof}[Sketch]   Suppose that $\Lambda \cap C \neq \emptyset$ for each component $C \subset \partial \Sigma$; the general case will follow from this by localization.  
A ``full formal arc system" is a collection $\{a_i\}$ of disjoint embedded arcs 
with boundary in $\partial \Sigma \setminus \Lambda$ which
decompose $\Sigma$ into a union of discs each containing exactly one
point of $\Lambda$. Any such system of arcs generates the category 
$\scrW(\Sigma,\Lambda)$. The
$A_{\infty}$-algebra of endomorphisms of the collection of objects $\{a_i\}$ has
a particularly simple description: it is formal, and the only
non-trivial products correspond to the concatenation of boundary chords on
each component of $\partial \Sigma\setminus \Lambda$.
This can be described by a nilpotent quiver algebra, with vertices the arcs 
and arrows the boundary chords connecting successive end points along each
component of $\partial \Sigma\setminus \Lambda$. 

Since the arcs $a_i$ generate the wrapped category, any object of 
$\Tw \scrW(\Sigma,\Lambda)$ can be expressed as a twisted complex 
$A=(\bigoplus V_i\otimes a_i, \delta)$ for some ($\bZ$ or $\bZ/2$) graded vector
spaces $V_i$. This twisted complex can be viewed as a representation of 
a ``net", i.e.\ a collection of vector spaces which carry two filtrations, 
together with prescribed isomorphisms between certain pieces of the 
associated gradeds.  

To a component $c$ of $\partial \Sigma\setminus \Lambda$, 
containing end points of the arcs $a_{i_1}, \dots,a_{i_k}$ in that order,
we associate the ($\bZ$ or $\bZ/2$) graded vector space 
$$V_c:=V_{i_1}[d_1]\oplus \dots \oplus 
V_{i_k}[d_k],$$ where the shifts $d_1,\dots,d_k$ reflect the gradings of
the arcs near the relevant end points. This carries two filtrations.
One comes from the ordering of the arcs along the boundary of $\Sigma$:
\[
V_{i_1}[d_1]\subset V_{i_1}[d_1]\oplus V_{i_2}[d_2]\subset \dots
\subset V_c.
\] 
The other comes from the part of the differential of the
twisted complex which involves
boundary chords lying along $c$, viewed as an endomorphism $\delta_c
\in \mathrm{End}(V_c)$ which squares to zero, giving the filtration
$$\im(\delta_c) \subset \ker(\delta_c) \subset V_c.$$
In the latter filtration, $\delta_c$ induces isomorphisms 
$V_c/\ker(\delta_c)\simeq \im(\delta_c)$, whereas in the former,
each $V_i$ appears twice in the associated gradeds (once for each end point
of the arc $a_i$). 

Because the language of nets is formulated for
ungraded vector spaces, the argument of \cite[Section 4.4]{HKK} 
actually splits the vector spaces $V_c$ according to cohomological degree:
in the $\bZ$-graded case, the indexing set for the net is $\pi_0(\partial\Sigma\setminus\Lambda)\times
\bZ$, and one considers the collection of vector spaces $V_c^d$ for all $c\in
\pi_0(\partial\Sigma\setminus \Lambda)$ and $d\in \bZ$ (each
equipped with the two filtrations described above).

With this understood, a classification theorem
generalising results of \cite{Nazarova-Roiter} implies that any indecomposable 
representation of a net is pushed forward from an indecomposable 
representation of a net of ``height 1", i.e.\ one in which all the filtrations
have length 1. These correspond to twisted complexes ``locally'' built from pieces
that involve a single arc among the $a_i$, and connecting differentials that
are isomorphisms between multiplicity vector spaces that are concentrated in
a single degree and correspond to a single boundary chord (not a
linear combination). Such twisted complexes look like either
\[ \xymatrix{ 
V_{j_1}\otimes a_{j_1} \ar[r] & V_{j_2}\otimes a_{j_2} & \ar[l] \ \ \dots \ \ \ar[r]
& V_{j_\ell}\otimes a_{j_\ell}} \qquad \text{or}
\]
\[ \xymatrix{ V_{j_1}\otimes a_{j_1} \ar[r] \ar@/_1.5pc/[rrr] & V_{j_2}\otimes a_{j_2} & \ar[l] \ \ \dots \ \ \ar[r]
&  V_{j_\ell}\otimes a_{j_\ell}} \qquad \phantom{or} \\[1em] {} \]
where the vector spaces $V_{j_i}$ are all isomorphic up to grading shift,
and the arrows between $V_{j_i}\otimes a_{j_i}$ and $V_{j_{i+1}}\otimes
a_{j_{i+1}}$ can point in either direction and each correspond to a 
single boundary chord.
Interpreting the mapping cone of a boundary chord geometrically as a boundary
connected sum surgery, these two kinds of twisted complexes correspond respectively to
immersed arcs and curves equipped with local systems.
(See \cite{HKK}; see also \cite{Burban-Drozd} for an earlier classification of objects of the derived category of a cycle of rational curves based on the same algebraic formalism.) 

Since we are working over the Novikov field, there is one more subtlety that
arises: when
using boundary connected sum surgeries to build a geometric object out of
the second kind of indecomposable twisted complex,
one arrives at an
immersed closed curve carrying a local system which is not necessarily
unitary. However, irreducibility implies that the holonomy has a single
eigenvalue, whose valuation can be adjusted by modifying the boundary connected sum
construction by an isotopy that sweeps a suitable amount of flux.
(The isotopy may however require replacing $\Sigma$ with a larger domain 
inside the Liouville completion; as noted in Remark \ref{Rmk:sweeparea}, for boundary-parallel curves this is
unavoidable, whereas for all other curves one can find enough space
within $\Sigma$ by applying the trick of Lemma \ref{Lem:sweeparea}).
After performing an isotopy to ensure that the eigenvalue has valuation zero, 
a suitable choice of basis of the local system (e.g., reducing to the Jordan normal form)
ensures that the holonomy is a valuation-preserving element of $GL_n(\Lambda)$.
We then arrive at an immersed curve with a unitary local system.

In \cite{HKK} the classification theorem is stated for objects of the $\bZ$-graded category (where gradings are defined with respect to any choice of line field $\eta$). 
However, the argument above uses filtrations coming from the boundary
structure of the full formal arc system and from the differential of the
twisted complex, and \emph{not} from the $\bZ$ indexing degrees.
The argument of \cite[Section 4.4]{HKK} carries over without modification to the $\bZ/2$-graded case simply by
using $\pi_0(\partial\Sigma\setminus\Lambda)\times \bZ/2$ instead of
$\pi_0(\partial\Sigma\setminus\Lambda)\times \bZ$ as indexing set
for the net and reducing the second factor mod 2 in all the statements.
This yields geometricity for $\bZ/2$-graded twisted complexes.
(In the special case of a once-punctured torus, an explicit algorithm for producing the geometric replacement of a $\bZ/2$-graded twisted complex guaranteed by Theorem \ref{Thm:HKK} is given in \cite{HRW}.)
\end{proof}


\begin{Remark} \label{Rem:not_gradeable}
Fix a grading structure on $S$ and a full formal arc system $\scrA$.  
The category $\Tw_{gr}(\scrA)$ of graded twisted complexes over $\scrA$ is split-closed, since it admits a stability condition by \cite{HKK}; hence $\Tw_{gr}(\scrA) = \Tw_{gr}^{\pi}(\scrA)$. There is a commuting diagram
\[
\xymatrix{
\Tw_{gr}(\scrA) \ar[r] \ar[d] & \Tw^{\pi}(\scrA) \ar[d] \\
SH^*(S) \ar@{=}[r] & SH^*(S)
}
\]
where the top arrow collapses the $\bZ$-graded structure to its underlying $\bZ/2$-grading and the vertical arrows are open-closed maps (these factor through the localisation functors from partially to fully wrapped categories).  Both vertical maps hit the unit, by \cite{GPS:structural}, so the image of $\Tw_{gr}(\scrA)$ in $\Tw^{\pi}(\scrA)$ is a split-generating subcategory. Nonetheless, the only idempotents which are admitted in $\Tw_{gr}(\scrA)$ are those of degree zero. Concretely, a simple closed curve which separates a genus one subsurface of $S$  containing no punctures cannot be graded for \emph{any} choice of line field, since the winding number of the line field is necessarily non-zero by the Poincar\'e-Hopf theorem. Such a curve defines an object of $\Tw^{\pi}(\scrA)$ which does not lift to $\Tw^{\pi}_{gr}(\scrA) = \Tw_{gr}(\scrA)$ for any choice of grading structure. \end{Remark}

Theorem \ref{Thm:HKK} shows that objects of $\Tw\scrW(S)$ are geometric, but not that objects of the split-closure $\Tw^{\pi}\scrW(S)$ are geometric. For a line field $\eta$ on $S$ we have a $\bZ$-graded full subcategory $\Tw\scrW(S,\eta) \subset \Tw\scrW(S)$. The existence of stability conditions on $\Tw\scrW(S,\eta)$ (as constructed in \cite{HKK}) implies that it is split-closed. However,  Remark \ref{Rem:not_gradeable} implies that one cannot reduce geometricity of idempotent summands of $\Tw\scrW(S)$ to geometricity of  objects of $\Tw\scrW(S,\eta)$, since there are objects in the former which don't lift to the latter for any choice of $\eta$. 

Section \ref{Sec:split-closure}, following a strategy from \cite[Appendix B]{AAEKO}, uses homological mirror symmetry to prove that the $\bZ/2$-graded category $\Tw\scrW(S)$ is split-closed whenever $S \subset (\bC^*)^2$ is a very affine curve in a maximally degenerating family. The next section reviews the relevant mirror symmetry input, due to Heather Lee \cite{HLee}.

\subsection{Homological mirror symmetry for punctured surfaces}

We consider a finite subset $A\subset \bZ^2$ and a function $\rho: A \to \bR$ which is the restriction of a convex piecewise-linear function $\bar\rho: \mathrm{Conv}(A) \to \bR$.  We assume that the maximal domains of linearity of $\bar\rho$ are the cells of a ``regular" polyhedral decomposition $\scrP$ of $\mathrm{Conv}(A)$, i.e. one with vertex set $A$ and for which every maximal cell is congruent to a $GL(2,\bZ)$-image of the standard simplex. We consider a punctured surface
\[
S_t = \left\{ \sum_{a\in A} c_a t^{-\rho(a)} z^a = 0 \right\} \subset \bC^* \times \bC^*, \quad t\gg 0
\]
with its natural exact convex symplectic structure. Explicitly, we can take 
\[
\omega_t = \frac{i}{2|\log(t)|^2} \sum_{j=1}^2 d\log z_j \wedge d\log \bar{z}_j.
\] 
The regularity hypothesis ensures that the genus and number of punctures of $S_t$ are independent of the choice of $\rho$, and the wrapped Fukaya category $\scrW(S_t,\omega_t)$ is independent of $t\gg 0$ up to quasi-isomorphism.  Let
\[
\Log_t: (\bC^*)^2 \to \bR^2, \ (z_1,z_2) \mapsto \frac{1}{|\log (t)|} (\log |z_1|, \log |z_2|).
\]
As $t\to\infty$ the images $\mathrm{Log}_t(S_t)$ Gromov-Hausdorff converge to the ``tropical amoeba", the 1-dimensional polyhedral complex $\Pi$ which is the singular locus of the Legendre transform of the convex function $\rho$, defined by 
\begin{equation} \label{eqn:Lrho}
L_\rho: \bR^2 \to \bR, \quad \xi \mapsto \mathrm{max} \, \left\{\langle a,\xi\rangle - \rho(a) \, | \, a\in A \right\}.
\end{equation}
This is combinatorially the 1-skeleton of the dual cell complex of $\scrP$.  The regions $R^a \subset \bR^2\backslash \Pi$ in the complement of the tropical curve $\Pi$  are labelled by elements of $A$, according to which term in \eqref{eqn:Lrho} achieves the maximum. Let 
\begin{equation} \label{eqn:LG_mirror}
\Delta_{Z,\rho} = \left\{(\xi,\eta) \in \bR^2 \times \bR \, | \, \eta \geq L_{\rho}(\xi) \right\},
\end{equation}
let $Z$ be the corresponding 3-dimensional toric variety, and $W: Z\to \bC$ the function defined by the toric monomial $(0,0,1)$, which vanishes to order $1$ along each component of $W^{-1}(0)$.  The mirror to $(S_t,\omega_t)$ is the
Landau-Ginzburg model $(Z,W)$, i.e.\ the symplectic geometry of $S_t$ is
reflected in the singularities of the toric divisor $W^{-1}(0) = Z_0$. 
For further discussion and context, see \cite[Section 3]{AAK}.  In particular, we point out that the topology of $Z$ depends on the choice of polyhedral decomposition $\scrP$, with different choices differing by flops. The irreducible toric divisors of $Z$ are labelled by the components $R^a \subset \bR^2\backslash \Pi$
(whose closures are their moment polytopes), and their intersections are determined by the combinatorics of the tropical curve.  

\begin{Remark} \label{Rem:numerology}
By considering general elements of linear systems of  curves in $\bP^2$ of degree $d$ or in $\bP^1\times\bP^1$ of bidegree $(a,b)$, one obtains punctured surfaces of genus $g$ with $\ell$ punctures for pairs $(g,\ell)$ of the form $((d-1)(d-2)/2, 3d)$ and $((a-1)(b-1), 2(a+b))$.  
More generally, one can obtain punctured surfaces for any $(g,\ell)$ with $3\leq \ell\leq 2g+4$ 
by considering the family of tropical plane curves depicted in Figure \ref{Fig:tropicaltripod}.
The upper bound $\ell-3\leq 2g+1$ on the number of horizontal legs ensures that the two `north-east antlers' of the curve don't intersect near infinity.
\end{Remark}

\begin{figure}[ht]
\begin{center} 
\begin{tikzpicture}[scale=0.35]
\draw[semithick] (-1,-3) -- (0,-1) -- (0,1) -- (-0.5,1.5);
\draw[semithick, dotted] (-0.5,1.5) -- (-1,2);
\draw[semithick] (-1,2) -- (-1.5,2.5) -- (-3.5,3.5) -- (-5.75,4.25);
\draw[semithick] (0,-1) -- (2,1) -- (3,1.5);
\draw[semithick, dotted] (3,1.5) -- (4,2);
\draw[semithick] (4,2) -- (5,2.5) -- (8,3.5) -- (12,4.5);
\draw[semithick] (2,1)--(0,1);
\draw[semithick] (5,2.5)--(-1.5,2.5);
\draw[semithick] (8,3.5)--(-3.5,3.5);
\draw[decorate,decoration={brace,amplitude=4pt}]
    (7.5,2.5)--(0.5,-1.5) node [midway,xshift=0.5em,yshift=-1em] {\small $g$};
\draw[semithick] (-9,4.25)--(-5.75,4.25)--(-7.25,5)--(-8,5.75)--(-8,6.5)--(-7.25,7.25);
\draw[semithick] (-9,5)--(-7.25,5);
\draw[semithick] (-9,5.75)--(-8,5.75);
\draw[semithick] (-9,6.5)--(-8,6.5);
\draw[decorate,decoration={brace,amplitude=4pt}]
    (-9.7,4)--(-9.7,6.75) node [midway,xshift=-2em] {\small $\ell-3$};
\end{tikzpicture}
\end{center}
\caption{Tropical plane curves of genus $g$ with $\ell$ punctures \label{Fig:tropicaltripod}}
\end{figure}
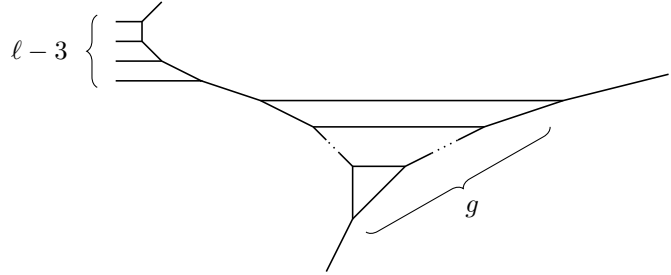

For any scheme $Z$, we will write $\mathrm{Perf}(Z)$ for the dg-category of perfect complexes over $Z$, which is the full subcategory of $D^b\Coh(Z)$ of objects admitting finite locally free resolutions. The dg-quotient of the latter by the former is the derived category of singularities $D_{sg}(Z)$. This is not in general split-closed and its split-closure is denoted $D^{\pi}_{sg}(Z)$.  The main result of \cite{HLee} asserts:

\begin{Theorem}[Heather Lee] \label{Thm:HMS-HLee}
There is an equivalence of $\bZ/2$-graded split-closed $\bC$-linear triangulated categories
\[
D^{\pi}\scrW(S_t) \simeq D^{\pi}_{sg}(Z_0).
\]
\end{Theorem}


Lee's proof of this theorem involves writing both sides as limits over restrictions to certain simple
pieces, and matching the two restriction diagrams in order to conclude that their
limits are equivalent. On one hand, the tropical curve $\Pi$ induces a decomposition of
$S_t$ into pairs of pants (indexed by the vertices of $\Pi$) glued together
along cylinders (corresponding to the bounded edges of $\Pi$); on the other hand, the
Landau-Ginzburg model $(Z,W)$ admits a matching decomposition into
toric affine charts $(\C^3,z_1z_2z_3)$. 
In Section \ref{Sec:restrict}, we will revisit the argument
in order to apply analogous technology to a closed symplectic surface.  For
now, we note the following consequence:

\begin{Proposition}\label{Prop:HMS-sub-HLee}
Let $S'_t\subset S_t$ be the union of the pairs of pants
corresponding to a given subset of the vertices of $\Pi$. 
Let $Z'$ be the union of the corresponding affine charts of $Z$, i.e\ the
toric 3-fold obtained from $Z$ by removing all the toric strata whose
closure does not contain any of the selected vertices, and $Z'_0=Z'\cap Z_0$. Then $D^\pi
\scrW(S'_t)\simeq D^{\pi}_{sg}(Z'_0)$.
\end{Proposition}

\begin{proof} This follows immediately from Lee's proof of Theorem
\ref{Thm:HMS-HLee} by considering only the parts of the restriction
diagrams that correspond to the pairs of pants and cylinders in $S'_t$ on
one hand, and to the affine charts of $Z'$ on the other hand. The limits of
these diagrams compute $D^\pi\scrW(S'_t)$ and $D^\pi_{sg}(Z'_0)$
respectively, which yields the result.
\end{proof}

\begin{Remark} \label{Rem:numerology2}
Proposition \ref{Prop:HMS-sub-HLee} allows us to apply Lee's result to punctured surfaces
of arbitrary genus and with any number $\ell\geq 3$ punctures. For $\ell>2g+4$
the graphs of Figure~\ref{Fig:tropicaltripod} are not entire tropical plane
curves because two of the legs would need to intersect each other outside of the depicted region
of the plane; however they describe subsurfaces $S'_t$ of higher genus curves
$S_t\subset (\bC^*)^2$, to which we can apply the Proposition.
The mirror configuration $Z'_0$ is again a union of smooth toric divisors of
$Z'$, as depicted in Figure \ref{Fig:tropicaltripod}; the only difference
with the previous setting is that one of the components of $Z'_0$
(corresponding to the upper-right region of the figure)
now arises as the complement of a toric divisor (e.g.\ a $\bP^1$ with negative
normal bundle) inside a compact component of $Z_0$;
this does not have any incidence on the properties of the derived category of singularities, and the results
below apply without modification to these examples.
\end{Remark}

Lee's proof furthermore matches certain specific objects on the two sides of the mirror.

Given any irreducible component $Z_0^a\subset Z_0$, and a line bundle
$\mathcal{L}_a\to Z_0^a$, the push-forward of $\mathcal{L}_a$ from 
$Z_0^a$ to $Z_0$ defines an object of $D_{sg}(Z_0)$, which by abuse of 
notation we also denote by $\mathcal{L}_a$. 

\begin{Proposition} \label{Prop:match}
The equivalence of Theorem \ref{Thm:HMS-HLee} matches the sheaves $\mathcal{L}_a$ with
properly embedded arcs or simple closed curves in $S_t$, hence with objects in the image
of\/ $\Tw\scrW(S_t) \to \Tw^{\pi}\scrW(S_t)$.
\end{Proposition}

\begin{proof}[Sketch]
This follows from the proof of Theorem \ref{Thm:HMS-HLee} as given in \cite{HLee}.  
Specifically, $\mathcal{L}_a$ is matched with
an arc or curve which lies on the portion of $S_t$ whose projection under
$\mathrm{Log}_t$ collapses, as $t\to\infty$, to the boundary of the 
corresponding region $R^a \subset \bR^2\backslash \Pi$. The specific arc 
or curve is determined up to Hamiltonian isotopy, and hence quasi-isomorphism 
in $\scrW(S_t)$,  by its winding over each of the cylindrical regions of $S_t$ which 
collapse to finite edges of $\partial R^a$, and by a normalization condition; 
the winding numbers are determined explicitly by the degrees of the
restriction of $\mathcal{L}_a$ to the corresponding projective lines 
in $Z_0^a$ (see \cite[Section 3.1]{HLee}). The details of the correspondence will not matter in the sequel.
(The discussion in \cite{HLee} concerns specifically those objects
$\mathcal{L}_a$ which arise from powers of the polarization determined by 
the polytope $\Delta_{Z,\rho}$, but the construction easily
extends to general line bundles.)
\end{proof}

\subsection{Split-closure\label{Sec:split-closure}}

Following a strategy from \cite{AAEKO}, this section will prove that the categories $\Tw\scrW(S)$ and $D_{sg}(Z_0)$ appearing in Theorem \ref{Thm:HMS-HLee} are in fact already idempotent complete, i.e. split-closed.  For a scheme $Z$, we write $K_j(Z)$ for $K_j(\mathrm{Perf}(Z))$. 

\begin{Proposition}
$D_{sg}(Z)$ is idempotent complete if and only if $K_{-1}(Z) = 0$.
\end{Proposition}

\begin{proof} This is \cite[Proposition B.1]{AAEKO}. \end{proof}

\begin{Lemma} \label{Lem:this_model_split_closed}
Let $W: Z \to \bC$ be as after \eqref{eqn:LG_mirror}. Then $D_{sg}(Z_0)$ is split-closed.
\end{Lemma}

\begin{proof}
Recall $W: Z \to \bC$ is a toric monomial morphism on a toric 3-fold $Z$ with central fibre $Z_0$ a union of (not necessarily compact) toric surfaces. Let $\Gamma \subset Z_0$ denote the one-dimensional subscheme of singular points of $Z_0$, and let $Z_0' \to Z_0$ be the normalisation and $\Gamma' = \Gamma\times_{Z_0} Z_0' \subset Z_0'$.  
Concretely, $Z_0'$ is the disjoint union of the toric surfaces appearing in $Z_0$, and $\Gamma'$
is the union of their toric boundaries.
 There is an exact sequence
\[
K_0(Z_0') \oplus K_0(\Gamma) \to K_0(\Gamma') \to K_{-1}(Z_0) \to K_{-1}(Z_0') \oplus K_{-1}(\Gamma) \to K_{-1}(\Gamma')
\]
(cf. the proof of \cite[Proposition B.2]{AAEKO}, which is itself inspired by
\cite{Weibel}).  In our case, $Z_0'$ is a union of smooth surfaces, so $K_{-1}(Z_0') = 0$. We claim that
\begin{enumerate}
\item the map $K_{-1}(\Gamma) \to K_{-1}(\Gamma')$ is an isomorphism;
\item  the map $K_0(Z_0') \oplus K_0(\Gamma) \to K_0(\Gamma')$ is surjective.
\end{enumerate}
By \cite[Lemma 2.3]{Weibel}, for the curves $\Gamma$ and $\Gamma'$ (whose irreducible components are all
$\bP^1$'s and $\mathbb{A}^1$'s) the $K_{-1}$-group is $\bZ^{b_1(\bullet)}$ where $b_1(\bullet)$ denotes the first Betti number of the curve. Suppose $S_t$ has genus $g$.   Then $b_1(\Gamma) = g = b_1(\Gamma')$, and the natural map $\Gamma' \to \Gamma$ identifies the corresponding cycles of $\bP^1$'s, which implies the first statement. 

For the second statement, recall that $\Gamma'$ is the disjoint union of the toric boundaries $\Gamma^a = \Gamma'\cap Z_0^a$
of the components of $Z_0$. For any non-compact component $Z_0^a \subset Z_0$, 
the map $K_0(Z_0^a) \to K_0(\Gamma^a)$ is surjective. There are $g$ compact components of $Z_0$, on each of which the corresponding map has rank two cokernel. 
This is essentially a cohomological computation, since the relevant $K_0$-groups for rational curves and toric surfaces are given by ranks of cohomology. 
By classical toric geometry, there is an exact sequence $$H^2(Z_0^a)\to H^2(\Gamma^a)
\to \Z^2\to 0,$$ where the second map sends each component of $\Gamma^a$ to
the primitive normal vector of  the corresponding facet of the moment
polytope. The cokernel of the first map is therefore generated by any two irreducible toric divisors whose
corresponding normal vectors form a basis of $\Z^2$: for instance, by the
Delzant condition, any two irreducible toric divisors which meet in one
point.

Given this, an easy inductive argument shows that the total map from
$K_0(Z'_0)\oplus K_0(\Gamma)$ to $K_0(\Gamma')=\bigoplus_a K_0(\Gamma^a)$ is surjective.  
Namely, pick an
ordering of the components of $Z_0$ such
that, for each compact component $Z_0^a\subset Z_0$, there exist two
components of $\Gamma^a$ whose normal vectors generate $\Z^2$ and which
arise as intersections of $Z_0^a$ with two other components of $Z_0$ which 
appear before it in the chosen ordering. (For example, order the components
by scanning $\R^2\setminus\Pi$ from
top to bottom: then the edges meeting at a top-most vertex of a compact component
have the requisite property.)  We then show that the map is surjective onto
each summand $K_0(\Gamma^a)$ by induction on $a$: for
non-compact components the map $K_0(Z_0^a)\to K_0(\Gamma^a)$ is surjective,
and for compact components our assumption yields two generators of $K_0(\Gamma)$ whose
images, after quotienting by the previously encountered summands of
$K_0(\Gamma')$, span the cokernel of $K_0(Z_0^a)\to K_0(\Gamma^a)$. This
implies surjectivity.
\end{proof}

\begin{Lemma}\label{Lem:find_generators}
The sheaves $\mathcal{L}_a$ of Proposition \ref{Prop:match} generate $D_{sg}(Z_0)$.
\end{Lemma}

\begin{proof} 
For each irreducible component $Z_0^a$ of $Z_0$, line bundles over the toric surface $Z_0^a$ generate
its derived category $D^b\Coh(Z_0^a)$ \cite[Corollary 4.8]{BorisovHorja}, so
their images under inclusion generate the full subcategory
$D^b_{Z_0^a}\Coh(Z_0)$ of complexes whose cohomology is supported on the
component $Z_0^a$. Considering all components, these sheaves taken together generate
$D^b\Coh(Z_0)$; the result follows.
\end{proof}

\begin{Corollary} \label{Cor:split-closed}
Let $S$ be a surface with $\ell \geq 3$ punctures. 
Then the category of $\bZ/2$-graded twisted complexes $\Tw\scrW(S)$ is split-closed.
\end{Corollary}

\begin{proof}
The hypotheses imply that $S$ can be realised as a hypersurface in $(\bC^*)^2$ defined by a 
Laurent polynomial as in the setting of Theorem \ref{Thm:HMS-HLee} and
Remark \ref{Rem:numerology}, or as a
subsurface as in Proposition \ref{Prop:HMS-sub-HLee} and Remark \ref{Rem:numerology2}. 
Lemma \ref{Lem:this_model_split_closed} (which applies equally well to the
examples of Remark \ref{Rem:numerology2}) shows that 
Theorem \ref{Thm:HMS-HLee} in fact gives an equivalence $\Tw^{\pi}\scrW(S) \simeq D_{sg}(Z_0)$. 
Furthermore, the right-hand side is generated (and not just split-generated)
by objects which lie in $\Tw\scrW(S)$, by Proposition \ref{Prop:match} and 
Lemma \ref{Lem:find_generators}. It follows that $\Tw\scrW(S)$ is split-closed.
\end{proof}

This incidentally shows that, for such a surface $S$, the category $D^{\pi}\scrW(S)$ has finite rank Grothendieck group; the corresponding result is false for a closed elliptic curve.

\begin{Corollary} \label{Cor:boundary_gives_geometric}
Let $S$ be a surface of genus $g$ with $\ell\geq 3$ punctures. 
Any irreducible object $X \in \scrW(S)^{per}$ with finite-dimensional endomorphism ring is 
quasi-isomorphic to a union of immersed closed curves with finite rank local system.
\end{Corollary}

\begin{proof} 
By Corollary \ref{Cor:split-closed}, $X$ is quasi-isomorphic to a twisted complex.  The geometricity
result, Theorem \ref{Thm:HKK}, then says that $X$ is quasi-isomorphic to the
direct sum of some immersed arcs and immersed closed curves with finite rank local systems.  
However, non-compact arcs have infinite-dimensional endomorphisms in the wrapped category, so cannot appear.
\end{proof}

\begin{Corollary} \label{Cor:spherical_on_punctured}
Let $S$ be a surface with at least one puncture. If $X \in D^{\pi}\scrF(S)$ is a spherical object such that
there exists an object $Y\in \scrW(S)$ with $\chi\Hom(X,Y)\neq 0$, then $X$ is quasi-isomorphic to a simple closed curve with rank one local system.
\end{Corollary}

\begin{proof} For surfaces with $\ell\geq 3$ punctures, this follows
directly from Corollary \ref{Cor:boundary_gives_geometric} and Corollary \ref{Cor:immersed_spherical}.
(The homological assumption on $X$ implies that the geometric replacement obtained by Corollary \ref{Cor:boundary_gives_geometric}
represents a non-zero class in $H_1(S;\bZ)$.)

If $S$ has fewer than three punctures, we reduce to the previous case by
considering the surface $S^+$ obtained by attaching a
4-punctured sphere $P$ to $S$ along an annular neighborhood $A$ of a
puncture, $S^+=S\cup_A P$. There is a fully faithful inclusion functor
$\scrF(S)\to \scrF(S^+)$ which comes from viewing $S$ as a subsurface of
$S^+$ and observing that none of the holomorphic polygons contributing to
the $A_\infty$-operations can escape into $S^+\setminus S$ (due to the open mapping
principle). This extends to a fully faithful functor $\scrF(S)^{per}\to
\scrF(S^+)^{per}$, and we can view $X$ as a spherical object of the latter category.
The inclusion into $S^+$ also preserves the property that $X$ has non-zero $\chi\Hom$ pairing
with some other object: if $Y$ is an arc with an end in the annulus $A$, we extend it in an arbitrary way 
across $P$ to obtain an arc in $S^+$.
Since $S^+$ has at least 3 punctures, we conclude that $X$ is
quasi-isomorphic to a simple closed curve $\gamma$ in $S^+$ with a rank one local
system~$\xi$. Given any properly embedded arc $\eta$ contained in $S^+\setminus S$
(e.g.\ connecting two punctures of~$P$), the vanishing of the wrapped Floer
cohomology $HW^*(X,\eta)$ implies that the geometric intersection number of
$\gamma$ with $\eta$ is zero; this in turn implies that $\gamma$ can be
isotoped away from $S^+\setminus S$. Alternatively, Lemma \ref{Lem:cut_annulus}
below (applied to the object $(\gamma,\xi)\in \scrF(S^+)$ and the waist curve of the annulus
$A$) implies that $\gamma$ is isotopic to a simple closed curve in the completion of $S$.
Either way, we conclude that $X$ is quasi-isomorphic to $(\gamma,\xi)$ in
$D^\pi\scrF(S)$.
\end{proof}

\begin{Remark}
The homological assumption in Corollary
\ref{Cor:spherical_on_punctured} is in fact equivalent to requiring that the geometric
replacement given by Corollary \ref{Cor:boundary_gives_geometric} represents
a non-zero class in $H_1(S;\bZ)$, as needed to apply Corollary
\ref{Cor:immersed_spherical}; the stronger assumption that some
object of $\scrF(S)$ has non-zero $\chi\Hom$ pairing with $X$ would amount to
non-vanishing in $H_1(S,\partial S;\bZ)$. These conditions are direct
analogues of the Chern character condition that appears in Theorem \ref{Thm:Main} for
closed surfaces; we have chosen this formulation in order to avoid a discussion of Chern characters and Mukai 
pairings for open surfaces, which would require another digression into
partially wrapped Fukaya categories.
\end{Remark}

To extend this result to closed surfaces,
we will use equivariant Floer theory  and restriction functors to subsurfaces to prove that a spherical object on a closed surface in fact comes from some open subsurface.

\section{Restriction to open subsurfaces}\label{Sec:restriction}

The first three subsections below review material from \cite{HLee}, which is subsequently applied in our setting.

\subsection{Dipping Hamiltonians}

Let $\Sigma$ be a surface (closed or with punctures) and $\sigma\subset \Sigma$ a  simple closed curve. (Given a finite union of disjoint curves $\sigma_j$ one can consider the corresponding Hamiltonians $H_k$ which dip near each; in an abuse of notation we will continue to write $A$ for the union of annular neighbourhoods of the $\sigma_j$, and refer to $A$ as an annulus.)
 Let $\sigma \subset A \subset \Sigma$ be an annular neighbourhood of $\sigma$.  Following \cite{HLee}, we consider a sequence of functions 
\[
H_k: \Sigma \to \bR
\]
which are small perturbations of the constant (say zero) function on $\Sigma \backslash A$ but ``dip" inside the annulus.
We work in co-ordinates $(r,\theta) \in (-2,2)\times S^1 = A$ such that the symplectic form is
$\omega = c\,dr\wedge d\theta$ for some constant $c>0$, and define $H_k(r,\theta) = c\,f_k(r)$ where 
\[
f_k(r) = \begin{cases} -k.\pi (r+2)^2 & -2<r<-1 \\ k.\pi r^2 -2k.\pi& -1 \leq r \leq 1 \\ -k.\pi(r-2)^2 & 1 <r<2 \end{cases}
\]  The time-$1$ Hamiltonian flow of $H_k$ lifted to the universal cover $(-2,2)\times\bR$ of the annulus $A$ is then given by
\[
\phi_{H_k}(r,\theta) = \begin{cases} (r, \theta-2k\pi.(r+2)) & -2<r<-1 \\ (r,\theta+2k\pi.r) & -1\leq r\leq 1 \\ (r,\theta - 2k\pi.(r-2)) & 1<r<2 \end{cases}
\]
Since $f_k = k.f_1$, the time-$1$ flow $\phi_{H_k}^1$ of $H_k$ is exactly the time-$k$ flow of $H_1$, and indeed there is a well-defined time-$t$ flow $\Phi_{H_1}^t$ for non-integer times $t$ which interpolates between the time-$1$  flows of the $H_k$. 
If $\gamma \subset A$ is an arc $\{\theta = \mathrm{constant}\}$ crossing the annulus, the time-$1$ Hamiltonian flow of $H_k$ applied to $\gamma$ yields an arc which wraps $k$ times clockwise around $A$, then $2k$ times anticlockwise, and then $k$ times clockwise again.  (Note that clockwise corresponds to negative Reeb flow and anticlockwise to positive Reeb flow.)   It will be important later to divide
\[
A = A_{in} \cup A_{out}
\]
into the inner region $A_{in} = (-1,1)\times S^1$ in which the wrapping is by positive Reeb flow, and the outer region $A_{out} = ((-2,-1) \sqcup (1,2)) \times S^1$ in which the wrapping is negative.  

For a pair of Lagrangians $L, L' \subset (X,\omega)$, and a Hamiltonian $H: X \to \bR$, denote by
\[
CF^*(L,L';H) := CF^*(\phi_H^1(L), L')
\]
given by flowing $L$ by the time-$1$ flow of $H$; the group is generated by time-$1$ chords of $H$ from $L$ to $L'$, or equivalently by intersections of $\phi_H^1(L)$ and $L'$.  For a given finite collection of Lagrangians, a generic small perturbation of  $H$ (which we shall suppress from the notation) will make all such chords non-degenerate.  In the setting at hand, given a pair of distinct arcs $\gamma_0, \gamma_1 \subset \Sigma$ which both cross the annulus $A$, the set of intersections $\phi_{H_k}^1(\gamma_0) \cap \gamma_1$ will grow in size with $k$, as more and more intersections appear in the ``wrapping" regions.

\begin{Lemma} \label{Lem:persists} There is $n(\gamma_0,\gamma_1) > 0$ with the following 
property: for any integer $w>n(\gamma_0,\gamma_1)$, any point $p\in \phi_{H_w}^1(\gamma_0) \cap \gamma_1$ belongs to a unique smooth arc 
\begin{equation} \label{eqn:maximal_smooth}
[w,\infty) \to \Sigma, \quad t \mapsto p(t) \in \phi_{H_t}^1(\gamma_0) \cap \gamma_1
\end{equation}
of transverse intersections between the time $t\geq w$ flow by $H_1$ of $\gamma_0$ and $\gamma_1$.
\end{Lemma}  

\begin{proof} See \cite[Section 3.5, properties (P1,2)]{HLee}. \end{proof}

Lemma \ref{Lem:persists} means that, once any pair of arcs has been sufficiently wrapped, their intersection points persist (and remain transverse) for all further time, even though new 
intersections keep being created (at non-integer times).   By using a cascade model for continuation maps of Floer complexes, as in \cite[Section 10e]{Seidel:FCPLT}, in which one counts exceptional holomorphic discs and flow-trees for isolated times $(J_t,H_t)$ in a one-parameter family of almost complex structures and Hamiltonians, Lee infers:

\begin{Lemma} \label{Lem:continuation_is_diagonal}
For $N>n>n(\gamma_0,\gamma_1)$, the continuation map 
\[
CF^*(\gamma_0,\gamma_1;H_n) \to CF^*(\gamma_0,\gamma_1;H_N)
\]
maps each generator $p$ to the summand generated by the intersection point which lies on the smooth arc $p(t)$ of \eqref{eqn:maximal_smooth}.
\end{Lemma}

Similarly, because of the non-existence of exceptional holomorphic discs on a Riemann surface, a cascade model for higher continuation maps shows that continuation-type products
\[
CF^*(\gamma_{k-1},\gamma_k;H_n) \otimes \cdots \otimes CF^*(\gamma_0,\gamma_1;H_n) \to CF^*(\gamma_0,\gamma_k; H_N)[1-k]
\]
with $k\geq 2$ inputs vanish whenever $N > kn > n(\gamma_0,\ldots,\gamma_k)$ is sufficiently large. 
This leads to a well-defined $A_{\infty}$-inclusion functor $CF^*(\gamma_0,\gamma_1;H_n) \hookrightarrow CF^*(\gamma_0,\gamma_1;H_N)$, for $N>n$ sufficiently large, which has no higher-order (nonlinear) terms.  It is simplest to formalise this by passing to telescope models for wrapped Floer complexes.

\subsection{Telescope models and $A_{\infty}$-ideals}

 The telescope complex for exact manifolds comes from \cite{Abouzaid-Seidel}, and a detailed exposition in the monotone case (under geometric hypotheses which also apply in the case of a closed surface) is given in \cite{Ritter-Smith}.   We will incorporate an action-rescaling of generators of Floer complexes, so we briefly review the set-up; for simplicity we suppress local systems, which are discussed in  \cite[Section 3.17]{Ritter-Smith}.

Let $X$ be closed or convex at infinity, and fix a pair of Lagrangian branes $L_i, L_j \in \scrW(X)$, which might be compact or cylindrical at infinity.   If outside a compact set $(X,\omega) \cong (\partial X \times [1,\infty), d(r\cdot \alpha))$ for a contact form $\alpha \in \Omega^1(\partial X)$ and co-ordinate $r \in [1,\infty)$, we will fix a Hamiltonian $H$ with $H(y,r) = r$, and which has no integer-length chords from $L_i$ to $L_j$.  The telescope model for the wrapped Floer complex is then 
 \begin{equation} \label{eqn:telescope_complex}
 CW^*(L_i,L_j)=\bigoplus_{w=1}^{\infty} CF^*(L_i,L_j;wH)[\mathbf{q}]
 \end{equation}
 where $\mathbf{q}$ is a formal variable of degree $-1$ satisfying $\mathbf{q}^2=0$,
equipped with the differential 
\begin{equation} \label{eqn:telescope_differential}
\mu^1(x+\mathbf{q}y) = (-1)^{|x|}\mathfrak{d} x + (-1)^{|y|}(\mathbf{q}\mathfrak{d} y + \mathfrak{K}  y -y) 
\end{equation}
where $\mathfrak{d}$ denotes the usual Floer boundary operator, and 
where $\mathfrak{K}$ denotes the Floer continuation map 
\[
\mathfrak{K}: CF^*(L_i,L_j;wH) \to CF^*(L_i,L_j;(w+1)H).
\]
The part of this complex which does not involve $\mathbf{q}$ is the direct
sum of Floer complexes $CF^*(L_i,L_j;wH)$, with the usual Floer differential.
For a $\mathfrak{d}$-cocycle $y$,  $\mathbf{q}y$ serves to identify $y$ and $\mathfrak{K} y$ in cohomology, as expected in the cohomology direct limit
\[
HW^*(L_i,L_j)=\varinjlim HF^*(L_i,L_j;wH)
\]

Let $u: \bR \times [0,1] \to X$ be a non-constant isolated (modulo $\bR$-translation) solution of Floer's equation $\partial_s u+J(\partial_t u - w\,X_H)=0$, with  Lagrangian boundary conditions $u(\cdot,0)\in L_i$ and $u(\cdot,1)\in L_j$, and asymptotic conditions $u\to x,y$ as $s\to -\infty,+\infty$ respectively, where $x, y$ are chords of $X_H$. 
Usually the differential $\mathfrak{d}$ counts isolated such solutions $u: S \to X$ weighted by their ``topological energy" $E_{\mathrm{top}}(u)\geq 0 \in \bR$, defined by:
\[
E_{\mathrm{top}}(u) 
= \int_S u^*\omega - d(u^*(wH)\, dt) 
= \int_S u^*\omega +wH(x)- wH(y),
\] 
which is also equal to the geometric energy
\[
E_{\mathrm{geo}}(u)= \frac{1}{2}\int_{S}\|du-wX_H\otimes dt\|^2\, ds\wedge dt.
\]
(There is also an orientation sign, which we will not discuss: for comprehensive treatments see \cite{Abouzaid-Seidel, Abouzaid:generation}.)  
Similarly the continuation map $\mathfrak{K}$ counts isolated solutions to the corresponding equation where $J$ and $wH$ (in fact, in our case
just $w$) depend on the parameter $s$ (so there is no $\bR$-reparametrization), again weighted by topological energy
(which now provides an upper bound on the geometric energy, up to an additive constant given by the minimum of the Hamiltonian).
Similar remarks apply to the counts of isolated curves contributing to the rest of the $A_{\infty}$-structure, cf. \cite[Section 3.11--15]{Ritter-Smith}.

\begin{Remark}
Floer's equation $\partial_s u+J(\partial_t u-w\,X_H)=0$ for maps $u:\bR\times [0,1]\to X$
with boundary on $L_i$ and $L_j$ can be recast in terms of 
pseudo-holomorphic curves with boundary on $\phi_{wH}^1(L_i)$ and $L_j$, 
by considering $\tilde{u}(s,t)=\phi_{wH}^{1-t}(u(s,t))$.
Continuation maps are then naturally defined via pseudo-holomorphic curves
with moving boundary conditions, or (under suitable conditions), as mentioned in the previous section,
cascades of honest pseudo-holomorphic curves.
The appropriate notion of topological energy is then obtained by transcribing 
the above definition through the dictionary between the two viewpoints.
\end{Remark}

Specialise now to the case where $X = \Sigma$ is a surface and we have functions $H_w = w\cdot H$ for a Hamiltonian which dips in some collection of annuli, as in the previous section. 
For a given collection of Lagrangians $\gamma_j$, Lemma \ref{Lem:continuation_is_diagonal} asserts that the continuation maps in the telescope complex \eqref{eqn:telescope_complex}, \eqref{eqn:telescope_differential} are eventually diagonal inclusions of based vector spaces.  To make them inclusions of subcomplexes, for a Hamiltonian $H_w$-chord $x \subset A$ lying in the cylinder at radial parameter $r \in (-2,2)$,  introduce the action values
\[
\scrA_w(x) = H_w(r) - r\cdot H_w'(r)
\]
(i.e. the intercept of the tangent line to the graph of the dipping function at $r$ with the vertical axis). 

\begin{Lemma} \label{Lem:inclusions}
Consider the rescaled bases of Floer complexes $CF^*(\gamma_j, \gamma_k; H_w)$ in which the chords $x$ inside the cylindrical regions $A$ are rescaled by their action $\scrA_w(x)$. Then for $N$ sufficiently large, the continuation maps $\mathfrak{K}$ are inclusions for $w>N$.
\end{Lemma}

\begin{proof}
Given a perturbed holomorphic polygon $u: S \to \Sigma$ with input chords $x_1^{in},\ldots, x_j^{in}$ with weights $w_1^{in},\ldots, w_j^{in}$ and output chord $x^{out}$ with weight $w^{out}$, and supposing the image $u(S)$ is contained in a single connected component of $A$, then
\[
E_{top}(u) = \scrA_{w^{out}}(x^{out}) - \sum_{i=1}^j \scrA_{w_i^{in}}(x_i^{in}).
\]
The Floer solution $u$ contributes to the usual continuation map with weight $q^{E_{top}(u)}$, which means that it takes rescaled generators to rescaled generators.
\end{proof}

The key property of dipping Hamiltonians, explained in \cite[Section 3.8]{HLee}, is that they lead to a ``localisation" result for holomorphic polygons, which precludes non-trivial Floer products from crossing the inner region $A_{in} \subset A$ of the annulus in which one wraps.
To state this precisely, denote by $S$ the surface obtained from $\Sigma\setminus A_{in}$ by
reattaching a copy of the corresponding annular region $A_{in}$ to each end
(see Definition \ref{defn:subsurfaceS} below):
$S$ is a (possibly disconnected) surface with boundary, and
there is a natural surjective map $S\to \Sigma$ 
which is two-to-one over $A_{in}$ and one-to-one outside of $A_{in}$.

Fix $\gamma_0,\ldots,\gamma_k \subset \Sigma$ which
intersect the boundary of $A_{in}$ minimally; equivalently, we require that the preimages
of $\gamma_i$ in $S$ do not contain any trivial arcs connecting $\partial S$ to
itself. Assume moreover that, if $\gamma_i$ 
is parallel to the annulus $A_{in}$ (i.e., isotopic to a curve contained in
$A_{in}$), then it is actually contained in $A_{in}$.
Also fix a collection of Floer generators $p_0,\ldots, p_{k-1}$ with $p_i \in CF^*(\gamma_i,\gamma_{i+1}; H_n)$.  By Lemma \ref{Lem:persists}, there are unique corresponding generators $\kappa(p_i) \in CF^*(\gamma_i, \gamma_{i+1};H_N)$ for any $N>n$, once $n$ is sufficiently large. 

\begin{Proposition} \label{Prop:isolate_polygons} Fix a collection of Floer generators $\{p_0,\ldots,p_{k-1}\}$ as above.
\begin{enumerate}
\item For $N> n$ sufficiently large,  every rigid perturbed holomorphic polygon in $\Sigma$ with inputs $\kappa(p_i) \in CF^*(\gamma_i,\gamma_{i+1}; H_N)$
lifts to a perturbed holomorphic polygon inside $($a single component of\/$)$ $S$.
\item For $N>n$ sufficiently large, if a rigid perturbed holomorphic polygon in $\Sigma$ 
with inputs $\kappa(p_i) \in CF^*(\gamma_i,\gamma_{i+1};H_N)$ is not entirely
contained in $A_{in}$ $($for example if at least one $p_i$ lies outside of $A_{in}$$)$
then its output also lies in the complement of $A_{in}$. 
\end{enumerate}
\end{Proposition}

\begin{proof} These follow respectively from \cite[Lemma 3.7]{HLee}, which
states that rigid perturbed holomorphic polygons cannot cross the inner
annular region $A_{in}$, and \cite[Lemma 3.5]{HLee}, which states that if
part of the disc lies outside of $A_{in}$ then so does the output. 
Lee states and proves these results for a specific collection of objects of $\scrF(\Sigma)$
(and \cite[Lemma 3.5]{HLee} is stated under the stronger assumption that one input lies outside of
$A_{in}$), but the same arguments apply verbatim to our setting, as we now
explain.

The proof of \cite[Lemma 3.5]{HLee} considers a rigid holomorphic polygon whose
output lies in $A_{in}$, and shows by contradiction that it must be entirely contained in
$A_{in}$ using a two-step argument. Lee first shows (``Case 1''
of the proof in \cite{HLee}) that the
boundary of the polygon cannot backtrack in $A\setminus A_{in}$; thus, if
the polygon is not entirely contained in $A_{in}$, some part of
it must lie outside of $A$. 
Our assumptions on the $\gamma_i$ ensure that, inside $A\setminus A_{in}$,
they are locally graphs $r=f(\theta)$, just like the curves considered in
\cite{HLee}; this (and the ordering on the slopes of these graphs imposed
by the negative wrapping) is what prevents the backtracking. 
Next, Lee argues (``Case 2'' of the proof) that the portion of the polygon which lies outside of $A$
determines the width of the interval(s) along which it intersects $\partial
A$, and as the amount of negative wrapping in $A\setminus A_{in}$ increases, 
the width of the corresponding interval(s) along which the polygon
intersects $\partial A_{in}$ decreases and eventually becomes negative, 
preventing it from entering $A_{in}$ altogether.
 Our assumption that all $\gamma_i$ intersect $\partial A_{in}$ minimally
ensures that no portion of disc crossing into $A$ can look
like a strip with both of its boundaries on the same $\gamma_i$; this in
turn ensures that negative wrapping does indeed decrease the available width at the boundary of
$A_{in}$. Moreover, excluding curves which are parallel to the annulus
but lie outside of $A_{in}$ ensures that
the width of the intersection of the polygon with $\partial
A$ is completely determined by the collection of input generators in $\Sigma\setminus
A$. (By contrast, polygons with corners at fixed generators on an annulus-parallel curve outside $A_{in}$
can enter the annulus with an arbitrarily large width, as the
boundary of the polygon could wrap around the annulus-parallel curve arbitrarily
many times.)

The proof of \cite[Lemma 3.7]{HLee} uses similar considerations to show that
a rigid holomorphic polygon cannot cross completely through $A_{in}$. First, Lee
observes that, when the $\gamma_i$ are locally given by graphs
$r=f(\theta)$, due the ordering of the slopes imposed by the positive wrapping
inside $A_{in}$, any input contained in $A_{in}$ forces the boundary of the
polygon to backtrack, which prevents it from crossing from one end of
$A_{in}$ to the other. This remains true if we allow curves
$\{r=r_0\}$ contained in $A_{in}$, as the local convexity property
of rigid holomorphic polygons ensures that if part of the boundary of
the polygon lies on such a curve then it must backtrack.
The rest of the argument is then similar to the proof of \cite[Lemma
3.5]{HLee}: if a polygon crosses $A_{in}$, then regardless of whether $A_{in}$ contains the
output, its width at the center of $A_{in}$ is determined by the inputs on
one side of $A_{in}$ or on the other, and is not affected by
the wrapping, while the available width at one of the two
boundaries of $A_{in}$ decreases with the amount of wrapping and eventually
becomes negative.
\end{proof}

\subsection{Lee restriction functors\label{Sec:restrict}}

Fix a finite set  $\gamma_1,\ldots,\gamma_{\ell}$ of split-generators  of $\scrF(\Sigma)$ each of which meets $\sigma$
minimally, and none of which is parallel to $\sigma$ (unless it is $\sigma$ itself).
The category $\scrF(\Sigma)$ embeds into the category of modules over the $A_{\infty}$-algebra $\oplus_{i,j} CF^*(\gamma_i,\gamma_j;H_k)$ for any fixed $k$, and hence into modules over the telescope algebra $\oplus_{i,j} CW^*(\gamma_i,\gamma_j)$. 

 Fix an annular neighbourhood $\sigma \subset A \subset \Sigma$ of a simple closed curve; as in the previous section we regard $\sigma = \{0\}\times S^1 \subset (-2,2)\times S^1 = A$ as divided into three sub-annuli, the outer of which comprise $A_{out}$ and the inner of which is labelled $A_{in}$. 
As above, we construct a surface $S$ by reattaching a copy of $A_{in}$ to
each end of $\Sigma\setminus A_{in}$:
\begin{Definition}\label{defn:subsurfaceS} \ 
\begin{itemize}
\item If $\sigma$ is separating, we view $\Sigma$ as a union of two subsurfaces with non-empty boundary which overlap in the central ``positive wrapping" region $A_{in} \subset A$, i.e. the subannulus $(-1,2)\times S^1$ lies inside one subsurface and $(-2,1)\times S^1$ lies inside the other. Write $S = S_{left} \sqcup S_{right}$ for the disjoint union of these subsurfaces.
\item If $\sigma$ is non-separating, we define $S$ to be a surface with two boundary components, neighborhoods of which are respectively modelled on $(-2,1) \times S^1 \subset A$ and $(-1,2) \times S^1 \subset A$.  Thus, $S$ is not strictly a subsurface of $\Sigma$, but there is an obvious map $S \to \Sigma$ which is two-to-one over the central subannulus  $A_{in} \subset A$.
\end{itemize}
\end{Definition}
 Given any finite collection of disjoint simple closed curves in $\Sigma$ and annular neighbourhoods of those curves, there are associated open surfaces given by cutting along the annuli in a similar manner. For our purposes, it is sufficient to focus on the case of a single annulus.

\begin{figure}[ht]
\begin{center} 
\begin{tikzpicture}[scale=0.65]

\draw[semithick] ellipse  (4.5 and 3);
\draw[semithick] (-3,0) arc (200:340:0.5);
\draw[semithick] (2.2,0) arc  (200:340:0.5);
\draw[semithick] (-2.87,-0.15) arc (150:30:0.4);
\draw[semithick] (2.33,-0.15) arc (150:30:0.4);

\draw[semithick] (1.68,2.77) arc (10:-10:16);
\draw[semithick, dashed] (1.68,2.77) arc (170:190:16);

\draw[semithick] (-1.41,2.77) arc (10:-10:16);
\draw[semithick, dashed] (-1.41,2.77) arc (170:190:16);

\draw[semithick] (0.7,2.95) arc (10:-10:17);
\draw[semithick] (-0.68,2.96) arc (10:-10:17);
\draw[semithick, dashed] (-0.68,2.96) arc (170:190:17);
\draw[semithick, dashed] (0.7,2.95) arc (170:190:17);

\draw [thick, decorate,decoration={brace,amplitude=10pt,mirror},xshift=0.4pt,yshift=-0.4pt](-4.5,-3.2) -- (0.6,-3.2)  node[black,midway,yshift=-0.6cm] {\footnotesize $S_{left}$};
\draw [thick, decorate,decoration={brace,amplitude=10pt,mirror},xshift=0.4pt,yshift=-0.4pt](-0.5,-3.5) -- (4.5,-3.5)  node[black,midway,yshift=-0.6cm] {\footnotesize $S_{right}$};
\draw (0.1,2) node {$A_{in}$};

\end{tikzpicture}
\end{center}
\label{Subsurfaces}
\caption{Restriction to subsurfaces}
\end{figure}
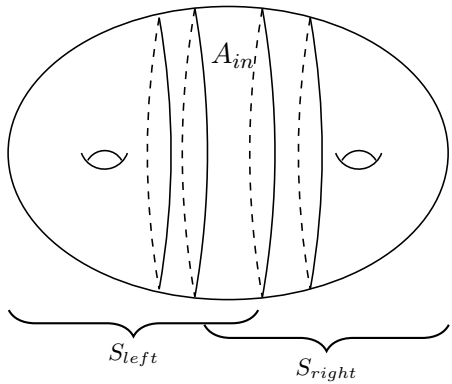

Note that the Hamiltonians $H_w$ considered previously induce the ``standard" positive wrapping at the ends of $S$ and $A_{in}$, hence are appropriate for defining their wrapped categories.

\begin{Lemma} \label{Lem:restriction_exists}
Let $A\subset \Sigma$ be a single annulus. There is an $A_{\infty}$-restriction functor $\Phi_{\Sigma,S}:\scrF(\Sigma) \to \scrW(S)$, and a pair of restriction functors $\Phi_{S,A_{in}}^{left/right}:\scrW(S) \to \scrW(A_{in})$ coming from the inclusions of $A_{in}$ into the two ends of $S$.
\end{Lemma}

\begin{proof} It suffices to construct the functors on a finite set of split-generators
$\gamma_i$; recall that we do not allow any of these to be
parallel to the annulus (but see Example
\ref{Ex:restriction_example} below for a description of the functor on
annulus-parallel curves).
We take telescope models for the respective morphism groups (even for the compact surface $\Sigma$) and note that, for any given set of inputs, replacing the telescope complex by a quasi-isomorphic truncation allows us to only consider arbitrarily large wrapping Hamiltonians. 

Proposition \ref{Prop:isolate_polygons} implies that the generators which lie outside of $A_{in}$ form an ideal for the $A_{\infty}$-operations in either $\Sigma$ or $S$, 
and that quotienting the Floer complex by this ideal recovers the wrapped Floer 
complex inside $A_{in}$ and its $A_\infty$-operations. Thus the quotient maps
\[
\Phi_{\Sigma,A_{in}}:CW^*_{\Sigma}(\gamma_i,\gamma_j) \to CW^*_{A_{in}}(\gamma_i,\gamma_j)
\ \text{ and } \ 
\Phi^{left/right}_{S,A_{in}}:CW^*_S(\gamma_i,\gamma_j)\to CW^*_{A_{in}}(\gamma_i,\gamma_j),
\]
where the subscripts denote the surfaces in which we compute Floer cohomology,
define restriction $A_\infty$-functors with no higher order terms.  

Next, we consider the pullback map 
\[
\Phi_{\Sigma,S}:CW_\Sigma^*(\gamma_i,\gamma_j)\to CW_S^*(\gamma_i,\gamma_j)
\]
mapping each Floer generator to its preimage under the map $S\to \Sigma$ if 
it lies outside of $A_{in}$, or to the sum of its two preimages if it lies
in $A_{in}$. Proposition \ref{Prop:isolate_polygons} implies that this map is
compatible with the $A_\infty$-operations (again after restricting 
to sufficiently large Hamiltonians), and defines a $A_\infty$-functor with
no higher order terms, identifying the wrapped Floer complex
in $\Sigma$ with the sub-algebra of the wrapped Floer complex in $S$ consisting
of elements in which pairs of generators in the two copies of $A_{in}\subset S$ appear with the same coefficients.
\end{proof}

\begin{Corollary}\label{Cor:pullback}
$\scrF(\Sigma)^{per}$ is the pullback in the diagram of restriction functors 
\[
\xymatrix{
\Tw^{\pi}\scrF(\Sigma) \ar[r] & \Tw^{\pi}\scrW(S) \ar@<-.5ex>[r] \ar@<.5ex>[r] & \Tw^{\pi}\scrW(A_{in}).
}
\]
\end{Corollary}

\begin{proof}
This follows directly from Lemma \ref{Lem:restriction_exists}. After truncating the telescope complexes to only consider Floer complexes $CF^*(\gamma_i,\gamma_j;H_n)$ where the wrapping $n$ is sufficiently large, $CW_\Sigma^*(\gamma_i,\gamma_j)$ is exactly the subcomplex of $CW^*_S(\gamma_i,\gamma_j)$ of elements which restrict compatibly under the two restrictions to $A_{in} \subset S$, i.e. the equalizer of the diagram $\xymatrix{ \scrW(S) \ar@<-.5ex>[r] \ar@<.5ex>[r] & \scrW(A_{in})}.$ This description is compatible with the $A_{\infty}$-operations since the arrows are given by quotienting by an ideal.  \end{proof}

\begin{Example} \label{Ex:restriction_example} 
Let $a$ and $b$ be two simple closed curves which cross an annulus $A$, meeting once just outside the annulus,
and differ from each other by
a Dehn twist parallel to the annulus (cf.\ the first part
of Figure~\ref{Fig:Cross_twice}). We consider the
twisted complex $X = a \xrightarrow{\ \ p\ \ } b$, which is isomorphic to a simple closed
curve $\gamma'$ running parallel to the annulus, as can be seen by considering the Lagrange surgery of $a$ 
and $b$ at $p$ (the red curve in Figure \ref{Fig:Cross_twice}).
(This is the most natural way in which an annulus-parallel curve can be replaced by
a twisted complex to which the machinery of Lee restriction functors can be applied).
\begin{figure}[ht]
\begin{tikzpicture}[scale=0.8]

\draw[semithick] (-8,4) -- (8,4);
\draw[semithick] (-8,1) -- (8,1);

\draw[semithick] (-4,4) arc (165:195:5.8);
\draw[semithick, dashed] (-4,4) arc (15:-15:5.8);
\draw[semithick] (4,4) arc (165:195:5.8);
\draw[semithick, dashed] (4,4) arc (15:-15:5.8);

\draw[semithick, blue] (-7.5,1.5) .. controls(-5,3.5) .. (0,3.5) .. controls(5,3.5) .. (7.5,1.5);
\draw[semithick, blue,rounded corners] (-7.5,3.5) .. controls(-5,1.5) .. (4.5,1.5) -- (5,1);
\draw[semithick, blue, dashed] (5,1)--(6,4);
\draw[semithick, blue] (6,4)--(7.5,3);

\draw[fill] (-6.15,2.5) circle (0.1); 
\draw (-6.15,2.1) node {$p$};
\draw (-7.8,3.5) node {$b$};
\draw (-7.8, 1.5) node {$a$};
\draw[semithick,red] (-7,3.1) arc (50:-52:0.77);
\draw[semithick,red] (-5.3,3.03) arc (120:242:0.55);

\draw[semithick, dashed,color=gray] [ ->] (-4,0.5) -- (-5.2,1.9) ;
\draw (-3.8,0.2) node {$\mathrm{Surgery} \ X = \mathrm{Cone}(a\stackrel{p}{\longrightarrow} b)$};

\draw[semithick] (-8,-1) -- (8,-1);
\draw[semithick] (-8,-4) -- (8,-4);
\draw[semithick] (-4,-1) arc (165:195:5.8);
\draw[semithick, dashed] (-4,-1) arc (15:-15:5.8);

\draw[semithick] (4,-1) arc (165:195:5.8);
\draw[semithick, dashed] (4,-1) arc (15:-15:5.8);
\draw[semithick, blue, rounded corners] (-7.5,-1.5) .. controls(-5,-3.5) .. (5,-3.5) -- (5.5,-4);
\draw[semithick, blue, dashed] (5.5,-4)--(6.5,-1);
\draw[semithick, blue] (6.5,-1)--(7.8,-2);
\draw (-7.8,-1.5) node {$b$};
\draw (-7.8, -3.5) node {$\phi_{H}(a)$};

\draw[semithick, blue] (-4,-4) parabola (-3.3,-1);
\draw[semithick, blue] (-4,-4) parabola (-4.7,-1);
\draw[semithick, blue, dashed] (-4.7,-1) -- (-5.1, -2.5) -- (-5.5,-4);
\draw[semithick,blue,rounded corners] (-7.8,-3.2) -- (-6.5,-2) -- (-6,-2) -- (-5.5,-4);
\draw[semithick, blue, dashed] (-3.3,-1) -- (-2.9,-4);
\draw[semithick,blue,rounded corners] (-2.9,-4) -- (-2,-2.5) -- (2.5,-2.5) -- (2.9,-1);
\draw[semithick, blue] (4,-1) parabola (3.3,-4);
\draw[semithick, blue] (4,-1) parabola (4.4,-4);
\draw[semithick, blue, dashed] (3.3,-4) -- (2.9,-1);
\draw[semithick, blue, dashed] (4.4,-4) -- (4.8,-1);
\draw[semithick,blue,rounded corners] (4.8,-1) -- (5.3,-2) -- (7.5,-3.5);

\draw[fill] (-6.65,-2.15) circle (0.1); 
\draw[fill] (-5.82,-2.67) circle (0.1); 
\draw[fill] (-2.55,-3.4) circle (0.1);
\draw[semithick, dashed,color=gray] [ ->] (-7,-4.5) -- (-6.65,-2.4) ; 
\draw (-7.1,-4.9) node {$p_\text{main}$};
\draw[semithick, dashed,color=gray] [ ->] (-6.1,-4.8) -- (-5.9,-2.85) ; 
\draw (-5.65, -5.1) node {$\partial(p_\text{cyl})$};
\draw[semithick,dashed, color=gray] [ ->] (-2,-4.5) -- (-2.45,-3.5);
\draw (-2,-4.9) node {$p_\text{cyl}$};
\draw (2.5, -6) node {$\mathrm{Continuation} \ \kappa: p \mapsto p_\text{main} + p_\text{cyl} +p'_\text{cyl}$};
\draw[fill] (3.35,-3.5) circle (0.1);
\draw[semithick,dashed, color=gray] [ ->] (2.5,-4.5) -- (3.25,-3.6);
\draw (2.45,-4.9) node {$p'_\text{cyl}$};
\draw[fill] (4.35,-3.5) circle (0.1);
\draw[semithick,dashed, color=gray] [ ->] (4.9,-4.5) -- (4.45,-3.6);
\draw (4.9,-4.9) node {$\partial(p'_\text{cyl})$};

\end{tikzpicture}
\caption{Continuation maps associated to dipping Hamiltonians\label{Fig:Cross_twice}}
\end{figure}
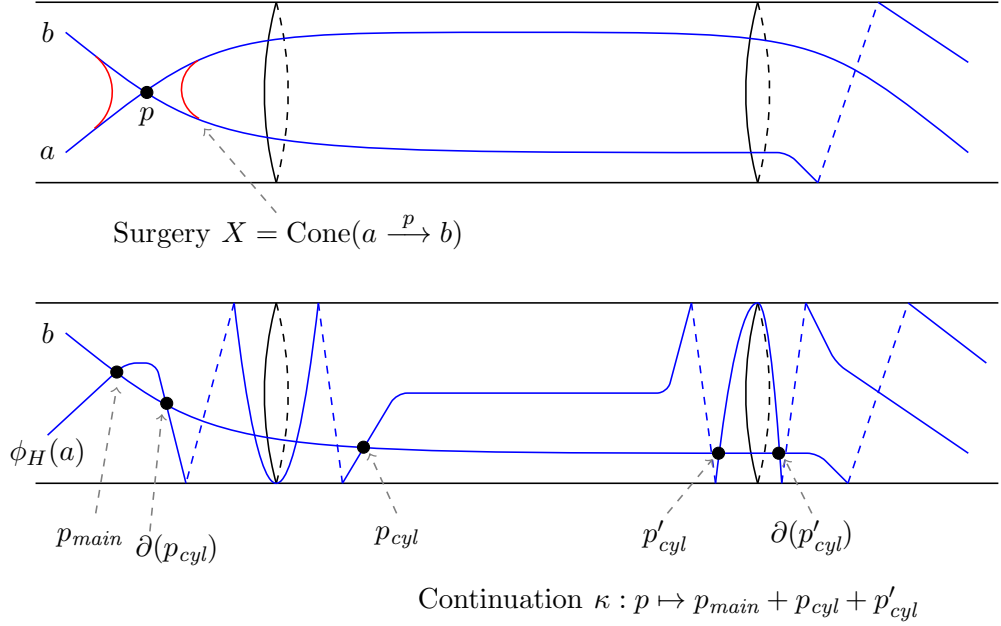
The Lee deformation effects the change in the second part of Figure \ref{Fig:Cross_twice}, which gives a quasi-isomorphic model 
$$X \ \simeq \ a \xrightarrow{p_\text{main}+p_\text{cyl}+p'_\text{cyl}} b.$$
Indeed, the deformation creates two holomorphic strips contributing to $\partial(p_\text{main})$, and the image of $p$
under the continuation map is $\kappa(p)=p_\text{main}+p_\text{cyl}+p'_\text{cyl}$
(up to a rescaling of the generators).
Restricting $a$ and $b$ to the inner annulus $A_{in}$ yields two isotopic arcs, 
and $\kappa(p)$ restricts to $p_\text{cyl}+p'_\text{cyl}$, so
$\Phi_{\Sigma,A_{in}}(X)=Cone(p_\text{cyl}+p'_\text{cyl})\in \Tw\scrW(A_{in})$ is isomorphic to a simple closed curve $\gamma'_{in}$ in
the Liouville completion of $A_{in}$. Just like the simple closed curve $\gamma'$ that
represents $X$ in $\Sigma$, this curve may or may not actually lie inside $A_{in}$; 
in the latter case, the difference with the naive restriction of $\gamma'$ illustrates
the need to exclude such annulus-parallel curves in the construction of Lee restriction functors.

The curves $a$ and $b$ restrict (or rather lift) to arcs in the open surface
$S$, representing isomorphic objects of $\scrW(S)$. Denoting by $p_\text{left}$ and $p_\text{right}$, resp.\ $p'_\text{left}$ and $p'_\text{right}$ the two
preimages of $p_\text{cyl}$, resp.\ $p'_\text{cyl}$ in the ends of the open surface $S$,
the twisted complex $\Phi_{\Sigma,S}(X)$ is the mapping cone
of $p_\text{main}+p_\text{left}+p'_\text{left}+p_\text{right}+p'_\text{right}$,
which is isomorphic to the direct sum of two boundary-parallel simple closed
curves in the completion of $S$. (If $\gamma'$ lies inside $A_{in}$ then these
are the lifts of $\gamma'$ to $S$; otherwise one curve is the
preimage of $\gamma'$ and the other one lies past the boundary of $S$ inside
the other end of the completion). Further restriction to either copy of $A_{in}$ yields the
mapping cone of $p_\text{cyl}+p'_\text{cyl}$ in $\Tw \scrW(A_{in})$, as
expected.

By contrast, if $a$ and $b$ did not run parallel to each other
in $\Sigma\setminus A$, so that no holomorphic strip connected $p_\text{main}$
to $\partial(p'_\text{cyl})$ in the second part of Figure \ref{Fig:Cross_twice},
then $\kappa(p)$ would be $p_\text{main}+p_\text{cyl}$, whose restriction to $A_{in}$
is $p_\text{cyl}$, a quasi-isomorphism between the restrictions of $a$ and $b$;
$\Phi_{\Sigma,A_{in}}(X)$ would then be
acyclic. This is as expected, since in that case the Lagrange surgery of 
$a$ and $b$ at $p$ can be pulled through $A_{in}$ and away from it altogether.
Meanwhile, the twisted complex $\Phi_{\Sigma,S}(X)$ becomes the mapping cone
of $p_\text{main}+p_\text{left}+p_\text{right}$, where $p_\text{left}$ and
$p_\text{right}$ are the preimages of $p_\text{cyl}$ in the ends of $S$;
it can be checked that this is isomorphic to a simple closed curve in 
$\Sigma\setminus A_{in}\subset S$, obtained from $a$ and $b$ by performing
the surgery in $\Sigma$ and sliding the result through $A_{in}$ and away from it.
\end{Example}

\subsection{Applications of restriction}\label{Sec:restriction_appl}
We continue with the notation of the previous subsection. 
Given a simple closed curve $\gamma\subset \Sigma$, we consider an annulus $A$ with waist curve $\gamma$,
and the Lee restriction functors $\Phi_{\Sigma,S}:\scrF(\Sigma) \to \scrW(S)$, and similarly 
$\Phi^{left/right}_{S,A_{in}}$ and $\Phi_{\Sigma,A_{in}}$.


\begin{Lemma}\label{Lem:cut_annulus}
Let $Y\in \Tw\scrF(\Sigma)$. 
If $HF^*(Y,\gamma) = 0$ then $\Phi_{\Sigma,S}(Y)$ and $\Phi_{\Sigma,A_{in}}(Y)$ are isomorphic to direct sums of
immersed closed curves with local systems in the completions of $S$ and $A_{in}$.
\end{Lemma}

\begin{proof}
As in Section \ref{Sec:restrict}, we consider the equalizer diagram 
\[
\xymatrix{
CW^*_{\Sigma}(Y,\gamma)\ar[r] & CW^*_S(\Phi_{\Sigma,S}(Y), \Phi_{\Sigma,S}(\gamma)) \ar@<-.5ex>[r] \ar@<.5ex>[r] & CW^*_{A_{in}}(\Phi^{left/right}_{S,A_{in}}(Y), \Phi^{left/right}_{S,A_{in}}(\gamma)),
}
\]
and recall that the restriction to $A_{in}$ is the quotient of $CW^*_\Sigma(Y,\gamma)$ by the
$A_\infty$-ideal spanned by the generators which lie outside of $A_{in}$.
Since $\gamma\subset A_{in}$, all the generators of
$CW^*_\Sigma(Y,\gamma)$ lie in $A_{in}$,
and Proposition \ref{Prop:isolate_polygons}~(2) ensures that the bigons contributing to 
the Floer differential are also entirely contained in $A_{in}$.
Therefore, $CW^*_{A_{in}}(\Phi_{\Sigma,A_{in}}(Y),\Phi_{\Sigma,A_{in}}(\gamma))$ is
isomorphic to $CW^*_\Sigma(Y,\gamma)$, hence acyclic.
It then follows from the equalizer diagram that $CW^*_S(\Phi_{\Sigma,S}(Y),
\Phi_{\Sigma,S}(\gamma))$ is also acyclic.

Next, we observe that $\Phi_{\Sigma,S}(\gamma)\simeq \gamma_\text{left}\cup
\gamma_\text{right}$ is the disjoint union of the two boundary-parallel simple
closed curves obtained by lifting $\gamma$ to the open surface $S$.
On the other hand, the restriction $\Phi_{\Sigma,S}(Y) \simeq \Gamma_Y$ has a geometric replacement $\Gamma_Y$ in the Liouville completion of $S$, which is a union of immersed arcs and curves with local systems, by Theorem \ref{Thm:HKK}.  
The vanishing of $HF^*(Y,\gamma)$ implies that $CW^*(\Gamma_Y,\gamma_\text{left})$
and $CW^*(\Gamma_Y,\gamma_\text{right})$ are acyclic.

We claim that $\Gamma_Y$ contains no arcs. Indeed,
Assume that $\Gamma_Y$ contains an immersed arc $\eta$ which reaches the boundary component of $S$ adjacent to $\gamma_\text{left}$.
As noted in Remark \ref{Rmk:find_sigma}, the argument of Lemma \ref{Lem:find_sigma} constructs an immersed arc $\eta'$ 
which is unobstructedly regular homotopic and hence quasi-isomorphic to $\eta$ and intersects $\gamma_\text{left}$ minimally, so that there are no
bigons contributing to the Floer differential on $CF^*(\eta',\gamma_\text{left})$.
However, the acyclicity of $CW^*(\Gamma_Y, \gamma_\text{left})$ implies that of $CF^*(\eta,\gamma_\text{left})$, since $\eta$ is a direct summand in $\Gamma_Y$.
Therefore $\eta'\cap \gamma_\text{left}=\emptyset$, and the arc $\eta'$ is entirely
contained in the cylindrical region that lies beyond $\gamma_\text{left}$ in the completion of $S$,
hence trivial as an object of $\scrW(S)$.
Arcs which reach the other boundary component of $S$  are similarly excluded
by considering $CW^*(\Gamma_Y, \gamma_\text{right})$.

It follows that 
the geometric replacement $\Gamma_Y \in \scrW(S)$ is a union of  immersed closed curves with local systems
in the completion of $S$. Moreover, as noted at the end of the proof of
Theorem \ref{Thm:HKK}, it can be assumed
that all non boundary-parallel curves in $\Gamma_Y$ are strictly contained
in the interior of $S$, or more precisely, in the subsurface $\Sigma\setminus A_{in}\subset S$ (see Remark \ref{Rmk:sweeparea}); whereas the boundary-parallel
curves can be ``straightened'' to run parallel to $\gamma$. 

Next we apply the restriction functors $\Phi_{S,A_{in}}^{left/right}$ to
$\Gamma_Y$. The non boundary-parallel summands are represented by curves
in $\Sigma\setminus A_{in}$, and hence mapped to zero. Meanwhile, the boundary-parallel curves which lie outside of $A_{in}$
need to be resolved by mapping cones as in Example \ref{Ex:restriction_example} in order
to apply the machinery of Section \ref{Sec:restrict}; the upshot is
that curves which are parallel to $\gamma_\text{left}$ (resp.\ $\gamma_\text{right}$) are mapped by
$\Phi_{S,A_{in}}^{left}$ (resp.\ $\Phi_{S,A_{in}}^{right}$)
to closed curves in the completion of $A_{in}$.
\end{proof}

\begin{Corollary}\label{Cor:cut_annulus}
Let $Y\in \Tw\scrF(\Sigma)$. 
If $HF^*(Y,\gamma) = 0$ then $Y$ is quasi-isomorphic to a direct sum of
immersed closed curves with local systems in $\Sigma$.
\end{Corollary}

\begin{proof}
We start again from the geometric replacement $\Gamma_Y \in \scrW(S)$
of $\Phi_{\Sigma,S}(Y)$ constructed in the proof of the previous Lemma, 
which is a direct sum of immersed closed curves with local systems in
the completion of $S$. Since $\Phi_{S,A_{in}}^{left}(\Gamma_Y)\simeq
\Phi_{S,A_{in}}^{right}(\Gamma_Y)\in \scrW(A_{in})$, the boundary-parallel
summands of $\Gamma_Y$ are ``the same'' curves near the two boundaries of~$S$.

Define $\hat{Y}\in \scrF(\Sigma)$ to be the {\em reduced} projection of $\Gamma_Y$ from
$S$ to $\Sigma$, i.e.\ the direct sum of the non-boundary-parallel summands
of $\Gamma_Y$ (which by Remark \ref{Rmk:sweeparea} can be assumed to lie
in $\Sigma\setminus A_{in}$), and for each pair of boundary-parallel summands
of $\Gamma_Y$, a curve in $\Sigma$ which runs parallel to the annulus $A$ and
differs from the waist curve $\gamma$ by the same amount of symplectic area.
(Recall that ``parallel'' means ``homotopic to a curve in $A$'': if $\gamma$ is separating, achieving the desired amount of
symplectic area might require a non-embedded regular
homotopy, cf.\ Lemmas \ref{Lem:sweeparea} and \ref{Lem:space}).
By construction, the restriction $\Phi_{\Sigma,S}(\hat{Y})$ (defined by first resolving
the annulus-parallel summands of $\hat{Y}$ as in Example \ref{Ex:restriction_example})
is isomorphic to $\Gamma_Y$. Thus, the 
two restriction diagrams
\[
\xymatrix{
\Phi_{\Sigma,S}(Y) \ar@<-.5ex>[r] \ar@<.5ex>[r] & \Phi_{\Sigma,A_{in}}(Y)
& \text{and} &
\Phi_{\Sigma,S}(\hat{Y}) \ar@<-.5ex>[r] \ar@<.5ex>[r] & \Phi_{\Sigma,A_{in}}(\hat{Y})
}
\]
involve quasi-isomorphic objects of $\scrW(S)$ and $\scrW(A_{in})$ (namely, $\Gamma_Y$ and its
restriction to $A_{in}$). In order to conclude from Corollary \ref{Cor:pullback} that $Y$ and $\hat{Y}$ are
quasi-isomorphic in $\Tw \scrF(\Sigma)$, we need to verify that the restriction maps in
the two diagrams are also the same.

More precisely, $\Phi_{\Sigma,S}(Y)\simeq \Gamma_Y$ determines a Yoneda module over $\scrW(S)$, 
whose pullback along $\Phi_{\Sigma,S}$ is a $\scrF(\Sigma)$-module which we
denote by $\scrY_S$. Similarly, $\Phi_{\Sigma,A_{in}}(Y)\in \scrW(A_{in})$ 
determines a Yoneda module over $\scrW(A_{in})$, and we denote by $\scrY_{\!A_{in}}$ its pullback along
$\Phi_{\Sigma,A_{in}}$. The functors $\Phi_{S,A_{in}}^{left/right}$ induce
two $\scrF(\Sigma)$-module homomorphisms 
\begin{equation}\label{eq:equalizer_module}
\xymatrix{\phi^{left/right}_Y:\scrY_S\ar@<-.5ex>[r] \ar@<.5ex>[r] &  \scrY_{A_{in}}\!\!},
\end{equation}
and Corollary \ref{Cor:pullback} states that the Yoneda module associated to $Y$
is isomorphic to the equalizer of this diagram in the category of $\scrF(\Sigma)$-modules.

Assume for now that all annulus-parallel curves in $\Gamma_Y$ actually lie inside $A_{in}$
(near both ends of $S$).
Then, using the correspondence between holomorphic polygons in $\Sigma$ and
in $S$ given by Proposition \ref{Prop:isolate_polygons}, the module $\scrY_S\in \text{mod-}\scrF(\Sigma)$ 
can be represented by a direct sum $\tilde{Y}_S$ of curves with local systems in $\Sigma$, namely the (total) projection
of $\Gamma_Y$ under the map $S\to \Sigma$. This differs from $\hat{Y}$ in that $\tilde{Y}_S$ contains
two copies of each of the curves which lie in $A_{in}$. Meanwhile, $\scrY_{A_{in}}$ is
represented by the direct sum $\tilde{Y}_{in}$ of the curves contained in a single copy of $A_{in}$.
Since the non-boundary-parallel summands of $\Gamma_Y$ are orthogonal to
the curves in $A_{in}$ (i.e., their Floer cohomology vanishes), the morphisms
from $\tilde{Y}_S$ to $\tilde{Y}_{in}$ which represent the module homomorphisms
$\phi_Y^{left/right}$ must be supported on the summands which
are parallel to $A_{in}$. For each annulus-parallel curve $\sigma\subset A_{in}$ which
appears as a summand of $\hat{Y}$,
the restriction $\Phi_{\Sigma,S}(\sigma)$ is a disjoint union $\sigma_\text{left}\cup \sigma_{\text{right}}$ of the
two preimages of $\sigma$ in the ends of $S$ (each of which appears as a summand in $\Gamma_Y$, and hence in $\tilde{Y}_S$). By considering the diagram
\begin{equation}\label{Eq:restr_diagram_annular}
\xymatrix{
CW^*_{\Sigma}(Y,\sigma)\ar[r] & CW^*_S(\Gamma_Y, \sigma_\text{left}\cup\sigma_\text{right}) \ar@<-.5ex>[r] \ar@<.5ex>[r] & CW^*_{A_{in}}(\Phi_{\Sigma,A_{in}}(Y), \sigma),
}
\end{equation}
we find that $\phi_Y^{left}$ maps the summand $\sigma_\text{left}$ of $\tilde{Y}_S$
 isomorphically to the summand $\sigma$ of $\tilde{Y}_{in}$, and vanishes on $\sigma_\text{right}$;
and vice-versa for $\phi_Y^{right}$. This implies in turn that the equalizer of the diagram \eqref{eq:equalizer_module} is isomorphic to $\hat{Y}$.

When $\Gamma_Y$ contains annulus-parallel curves which lie outside of $A_{in}$, 
the argument is essentially the same, but now requires a detour via the construction of Example \ref{Ex:restriction_example} to show that the modules
$\scrY_S$ and $\scrY_{in}$ are represented by the total projection $\tilde{Y}_S$ of $\Gamma_Y$ (containing two
copies of each annulus-parallel curve) and by $\tilde{Y}_{in}$ as in the previous case;
and once again to analyze the restriction diagrams \eqref{Eq:restr_diagram_annular}
for each annulus-parallel summand $\sigma$ of $\hat{Y}$.
\end{proof}

\begin{Remark}
Instead of dealing with annulus-parallel curves which lie outside of
$A_{in}$ via Example \ref{Ex:restriction_example}, one might try to
simply enlarge the annulus $A_{in}$ inside $\Sigma$ 
in order to ensure that it contains all of the annulus-parallel curves 
whose Floer cohomology with $Y$ is nonzero. However, this is not always
possible, as the area of an
immersed cylinder bound by a pair of
annulus-parallel curves can be greater
than that of the whole surface $\Sigma$.
\end{Remark}

When $Y$ is an idempotent summand rather than a twisted
complex, a similar result holds provided that $Y$ pairs trivially with {\em
two} simple closed curves:

\begin{Lemma}\label{Lem:cut_two_annuli}
Let $Y\in \Tw^\pi \scrF(\Sigma)$. Let $\gamma_1,\gamma_2\subset \Sigma$ be
two disjoint simple closed curves whose homology classes are linearly independent,
and denote by $A_1,A_2$ (disjoint) annular neighborhoods of $\gamma_1,\gamma_2$,
and by $S$ the surface obtained by cutting $\Sigma$ open along both annuli.
If $HF^*(Y,\gamma_1) = HF^*(Y,\gamma_2)=0$ then $\Phi_{\Sigma,S}(Y)$ and
$\Phi_{\Sigma,A_{i,in}}(Y)$ are isomorphic to direct sums of immersed closed
curves with local systems in the completions of $S$ and $A_{i,in}$.
\end{Lemma}

\begin{proof}
The argument proceeds as in the proof of Lemma \ref{Lem:cut_annulus},
considering two annuli parallel to $\gamma_1$ and $\gamma_2$. The
assumption that $\gamma_1$ and $\gamma_2$ are homologically independent
ensures that cutting $\Sigma$ open along both annuli yields a connected surface with four punctures, so that
the existence of a geometric replacement $\Gamma_Y$ for $\Phi_{\Sigma,S}(Y)$
is guaranteed by Corollary \ref{Cor:boundary_gives_geometric}.
\end{proof}

\begin{Corollary}
Let $Y \in \Tw^{\pi}\scrF(\Sigma)$. Suppose there are disjoint,
homologically independent simple closed curves $\gamma_1,\gamma_2 \subset \Sigma$ with
$HF(Y,\gamma_i) = 0$. 
Then $Y$ is quasi-isomorphic to a direct sum of immersed closed curves with local systems in $\Sigma$.
\end{Corollary}

\begin{proof} This follows from
Lemma \ref{Lem:cut_two_annuli} by the same argument as Corollary \ref{Cor:cut_annulus}.
\end{proof}

\begin{Corollary} \label{Cor:non-filling_implies_geometric}
Let $X \in D^{\pi}\scrF(\Sigma)$ be spherical with $ch(X)\neq 0$. 
If there are two disjoint homologically independent simple closed curves $\gamma_j$, $j=1,2$, with $HF^*(X,\gamma_j) = 0$, then 
$X$ is quasi-isomorphic to an embedded simple closed curve with a rank one
local system. 
\end{Corollary}

\begin{proof}
The previous corollary gives a geometric replacement for $X$; since a spherical
object is indecomposable, the replacement consists of a single immersed
curve $\sigma$ with local system. The assumption that
$ch(X)\in H_1(\Sigma;\bZ)$ is non-zero implies that $[\sigma]\in H_1(\Sigma;\bZ)$ is
non-zero.
Corollary  \ref{Cor:immersed_spherical} then implies that $\sigma$ is embedded and the
local system has rank one. 
\end{proof}

\section{Group actions}

\subsection{Finite group actions on categories}

Let $G$ be a finite group.  Following \cite{Seidel:Am_milnor_fibres, Seidel:categorical_dynamics}, we say $G$ acts strictly on a strictly unital $A_{\infty}$-category $\scrA$ if there is an action of $G$ on $\Ob\,\scrA$, and corresponding maps between morphism spaces which strictly satisfy the relations of $G$ and for which the $A_{\infty}$-operations are strictly equivariant.  An action of $G$ on $\scrA$ induces one on the category $\scrA^{mod}$ of $A_{\infty}$-modules over $\scrA$, which is a strictly unital $A_{\infty}$-category.  Necessarily $G$ will then preserve the strict units.  

A \emph{strictly equivariant} object $Y$ of $\scrA$ is one for which we have closed morphisms
\[
\rho^1_Y(g) \in hom^0_{\scrA}(g(Y),Y), \quad \rho^1_Y(e_G) = e_Y
\]
satisfying
\[
\mu^2_{\scrA}(\rho^1_Y(g),g\cdot \rho^1_Y(h)) = \rho^1_Y(gh)
\]
for all $g,h\in G$.  Any object $Y \in \scrA$ which is fixed by $G$, meaning $g(Y) = Y$ for every $g\in G$, defines a strictly equivariant object  for each character $\chi: G \to \bK^*$ via $\rho^1_Y(g) = \chi(g) e_Y$.

In the case of $A_\infty$-modules, a strictly equivariant structure on
$M\in \scrA^{mod}$ is given
by $A_\infty$-homomorphisms $\rho^1_M(g)\in \hom^0_{\scrA^{mod}}(g(M),M)$
which generally include higher order terms. The special case where the
module homomorphisms $\rho^1_M(g)$ are ordinary linear maps, with all higher order
terms identically zero, corresponds to the situation where $M$ is
equivariant in the \emph{naive} sense, i.e., there is
a linear action of $G$ on the vector spaces underlying $M$, with
respect to which the structure maps of $M$ are equivariant.

A \emph{weakly equivariant} object is one which satisfies the cohomological analogue
of strict equivariance, i.e. there are classes  $[\rho^1_Y(g)] \in  \Hom^0_{H(\scrA)}(g(Y),Y)$ which satisfy
\[
[\rho^1_Y(gh)] = [\rho^1_Y(g)]\cdot (g\cdot[\rho^1_Y(h)]), \qquad [\rho^1_Y(e_G)] = e_Y.
\]
If $Y$ and $Y'$ are both weakly equivariant, then $H^*(\hom_{\scrA}(Y,Y'))$ becomes a $G$-representation via $a \mapsto [\rho^1_{Y'}(g)]\cdot
(g\cdot a) \cdot [\rho^1_{Y}(g)]^{-1}$.

\begin{Lemma} [Seidel] \label{Lem:weaktonaive}
Assume $\bK$ has characteristic zero.  If $Y$ is a weakly equivariant object
of $\scrA$ or $\scrA^{per}$, then the Yoneda module of $Y$ is quasi-isomorphic to a
naively equivariant module.
\end{Lemma}

\begin{proof} This follows from \cite[Propositions 14.5 \& 14.7]{Seidel:categorical_dynamics}
(which in turn are direct analogues  of \cite[Lemmas 8.2 \& 8.3]{Seidel:Am_milnor_fibres}; in particular they do not rely on the problematic Lemma 13.7 in
\emph{op.\ cit.}). Proposition 14.5 is essentially an obstruction theory calculation
(see also \cite[Section 8c]{Seidel:Am_milnor_fibres}) showing that, if
$H^r(G;\Hom_{H(\scrA)}^{1-r}(Y,Y)) = 0$ for $r \geq 2$, then a weakly
equivariant structure can be upgraded to a {\em homotopy equivariant} (also known
as {\em coherently equivariant}) structure. In our case, the cohomology $H^r(G;M)$ 
vanishes for $r>0$ whenever $G$ is finite and $M$ is a $G$-module in which $|G|$ is invertible in $M$, in particular for all modules in characteristic zero. 
Once this is done, Proposition 14.7 of \cite{Seidel:categorical_dynamics} constructs a naively
equivariant $A_\infty$-module (see also \cite[Section
8b]{Seidel:Am_milnor_fibres}; the rationality requirement which imposes an
extra assumption on \cite[Lemma 8.2]{Seidel:Am_milnor_fibres} is not
relevant here).
\end{proof}

\subsection{Finite actions and coverings\label{Sec:finite_covers}}

Let $G$ be a finite abelian group with dual group $G^{\vee} = \Hom(G,\bC^*)$.  
Given a surface $\Sigma$, recall that $H^1(\Sigma;\bC^*)$ acts on the
Fukaya category $\scrF(\Sigma)$ by tensoring by flat (Novikov-unitary)
line bundles. Thus,
a homomorphism from $G$ to $H^1(\Sigma;\bC^*)$
determines an action of $G$ on the Fukaya category $\scrF(\Sigma)$, where
each $g\in G$ acts by twisting by a local system $\zeta^g$
(with $\zeta^{g_1g_2}\simeq \zeta^{g_1}\otimes \zeta^{g_2}$).
The main example that we consider is the following:

\begin{Lemma}\label{Lem:Zpaction}
A non-zero class $a\in H^1(\Sigma;\bZ)$ defines a $\bZ/p$-action on $\scrF(\Sigma)$, for any $p \geq 2$.
\end{Lemma}

\begin{proof} The action on objects is given by tensoring by flat unitary line bundles with holonomy given by the  class defined by $a$ in $\Hom(\pi_1(\Sigma);\Z/p)$.  
Choosing the Floer perturbation data to be independent of local systems in the
construction of $\scrF(\Sigma)$ ensures that this action is strict.
(Note that, in general, there is a standard trick to make a finite action strict on a quasi-isomorphic model of $\scrF(\Sigma)$, by passing to a category whose objects are pairs of an object of $\scrF(\Sigma)$ and an element of the given finite group, and choosing perturbation data independently for such pairs.  See e.g. \cite[Appendix A]{Sheridan:CY} for a closely related
case.)
\end{proof}

A homomorphism from $G$ to $H^1(\Sigma;\C^*)$ is the same thing as an
element of $H^1(\Sigma;G^\vee)$, or a homomorphism $\mu:\pi_1(\Sigma)\to
G^\vee$. Thus, it determines a finite Galois covering $\tilde{\Sigma}\to
\Sigma$ with deck group $G^\vee$. 


Given an immersed curve with a local system 
$(\xi,\gamma)\in \scrF(\Sigma)$, the action of $g\in G$ twists $\xi$
by the rank one local system $\zeta^g_{|\gamma}$, whose holonomy is $\langle \mu([\gamma]),g\rangle$.
Thus, if $\xi$ has rank one then the $G$-action preserves the isomorphism
class of $(\xi,\gamma)$ if and only if $[\gamma]\in \Ker\mu$, i.e.\ if and
only if $\gamma$ lifts to the covering $\tilde{\Sigma}$. The set of
$G$-equivariant structures on $(\xi,\gamma)$ is then a $G^\vee$-torsor, and so
is the set of lifts of $(\xi,\gamma)$ to $\tilde{\Sigma}$.
To be explicit, fix a base point in $\Sigma$, and a trivialization of the $G$-family of
local systems $\{\zeta^g\}$ at the base point. The choice of an arc connecting $\gamma$ to
the base point then determines on one hand, a trivialization of the collection of local
systems $\{\zeta^g\}$ over $\gamma$, and hence a $G$-equivariant structure 
on the object $(\xi,\gamma)\in \scrF(\Sigma)$ induced by the 
isomorphisms $\xi\otimes \zeta^g_{|\gamma}\stackrel{\sim}{\rightarrow} \xi$;
and on the other hand, a lift $\tilde\gamma$
of $\gamma$ to $\smash{\tilde\Sigma}$. Moving the base arc by a loop along which
$\mu$ takes the value $\chi\in G^\vee$, we obtain another $G$-equivariant 
structure on $(\xi,\gamma)$, which we denote by $(\xi,\gamma)^\chi$, 
where the isomorphism $\xi\otimes \zeta^g_{|\gamma}
\stackrel{\sim}{\rightarrow}\xi$ is modified by $\chi(g)\in \C^*$;
and another lift ${\tilde\gamma}^\chi$ of $\gamma$ to $\smash{\tilde\Sigma}$, which
differs from $\tilde\gamma$ by the deck transformation $\chi$.

Let $\{\gamma_i\}$ be a finite collection of split-generators of
$\scrF(\Sigma)$, whose homotopy classes all lie in $\Ker\mu$. 
(One way to construct such $\gamma_i$ is to choose simple closed curves in
$\tilde\Sigma$ which satisfy Abouzaid's split-generation criterion for
$\scrF(\tilde\Sigma)$, and project them down to $\Sigma$. Compatibility 
of the open-closed map with the projection $\tilde\Sigma\to\Sigma$ implies
that the corresponding immersed curves in $\Sigma$ split-generate
$\scrF(\Sigma)$.) 
Fixing base arcs as above, we equip each $\gamma_i$ with a preferred
(strict) $G$-equivariant structure, and a preferred lift $\tilde\gamma_i$ to
$\tilde\Sigma$.

The chosen $G$-equivariant structures on $\gamma_i$ equip
the $A_\infty$-algebra $\A=\bigoplus_{i,j} CF(\gamma_i,\gamma_j)$ with
an action of $G$. Explicitly, $g\in G$ acts on each Floer generator
$p\in \gamma_i\cap \gamma_j$ by multiplication by $\chi_p(g)$, where $\chi_p\in G^\vee$ is the image under
$\mu$ of the loop formed by connecting the intersection point $p$
to the base point along the base arcs for $\gamma_i$ and
$\gamma_j$. The $G$-equivariant part of the Floer complex,
$CF^G(\gamma_i,\gamma_j)$, is then generated by those intersections for
which $\chi_p=1$. Those correspond exactly to the intersections between
the chosen lifts of $\gamma_i$ and $\gamma_j$ to $\tilde{\Sigma}$; therefore
$CF^G(\gamma_i,\gamma_j)\simeq CF_{\tilde\Sigma}(\tilde\gamma_i,\tilde\gamma_j)$.
Varying the $G$-equivariant structures by $\chi_i,\chi_j\in G^\vee$,
$CF^G(\gamma_i^{\chi_i},\gamma_j^{\chi_j})$
is generated by those intersections for which $\chi_p=\chi_i^{-1}\chi_j$, and
isomorphic to $CF_{\tilde\Sigma}(\tilde\gamma_i^{\chi_i},\tilde\gamma_j^{\chi_j})$.
These isomorphisms are compatible with the $A_\infty$-operations
(if the Floer perturbation data are chosen consistently in $\Sigma$ and
$\tilde\Sigma$), and give an isomorphism of $A_\infty$-algebras
\begin{equation}\label{eq:Equivariant_vs_Lift}
\tilde{\A}=\bigoplus_{i,j,\chi_i,\chi_j}
CF_{\tilde\Sigma}(\tilde\gamma_i^{\chi_i},\tilde\gamma_j^{\chi_j})\simeq
\bigoplus_{i,j,\chi_i,\chi_j} CF^G(\gamma_i^{\chi_i},\gamma_j^{\chi_j}).
\end{equation}

Recall from \cite[Ch.~4c]{Seidel:HMSquartic} that the semidirect product
$A_\infty$-algebra $\A\rtimes G$ is defined by considering the tensor
product $\A\otimes_{\bK} \bK[G]$ of $\A$ with the group ring of $G$, 
with the operations
$$\mu^d_{\A\rtimes G}(a_d\otimes g_d,\dots,a_1\otimes g_1)=
\mu^d_\A(a_d,g_d\cdot a_{d-1},(g_dg_{d-1})\cdot a_{d-2},\dots)\otimes
(g_d\dots g_1).$$
Recalling that each $\chi\in G^\vee$ determines an idempotent
$e_\chi=|G|^{-1} \sum \chi(g)\,g\in
\bK[G]$, an easy calculation shows that the linear map $\varphi:\tilde\A
\to \A\rtimes G$ which takes
$a\in CF^G(\gamma_i^{\chi_i},\gamma_j^{\chi_j})\subset \tilde\A$ to
$\varphi(a)=a\otimes e_{\chi_i^{-1}}\in CF(\gamma_i,\gamma_j)\otimes
e_{\chi_i^{-1}}\subset \A\rtimes G$ defines an isomorphism of
$A_\infty$-algebras. Summarizing, we have:

\begin{Proposition}[{Seidel \cite[Ch. 4 \& 8b]{Seidel:HMSquartic}}]
$\A \rtimes G \simeq \tilde\A$.
\end{Proposition}

An immersed curve with local system $Y=(\xi,\gamma)\in\scrF(\Sigma)$, with a $G$-equivariant
structure (i.e., a choice of isomorphisms $\xi\otimes
\zeta^g_{|\gamma}\stackrel{\sim}{\rightarrow} \xi$), determines a module
$\scrY_{(\xi,\gamma)}=\bigoplus_i CF(\gamma_i,Y)$ over the
$A_\infty$-algebra $\A$, which is $G$-equivariant in the naive sense
(i.e.\ $G$ acts linearly on $\scrY_{(\xi,\gamma)}$ and all the module maps are equivariant).
A naive $G$-equivariant structure on an $\A$-module $M$ equips $M$ with the
structure of a module over $\A\rtimes G$, with structure maps given by 
$$\mu^{d+1}(m,a_d\otimes g_d,\dots,a_1\otimes g_1)=
(g_1^{-1}\dots g_d^{-1})\cdot\mu^{d+1}(m,a_d,g_d\cdot a_{d-1},
(g_dg_{d-1})\cdot a_{d-2}, \dots).$$
In the case of $\scrY_{(\xi,\gamma)}$, this has a geometric interpretation
in terms of the lift $\tilde{Y}$ of $Y$ to $\tilde\Sigma$:
$$\scrY_{(\xi,\gamma)}\simeq \bigoplus_{i,\chi_i} CF^G(\gamma_i^{\chi_i},Y)
\simeq \bigoplus_{i,\chi_i}
CF_{\tilde\Sigma}(\tilde\gamma_i^{\chi_i},\tilde{Y})$$
as modules over $\A\rtimes G\simeq \tilde\A$.


Transcribing these statements in the language of modules over
$A_\infty$-categories, rather than the endomorphism algebra of a given
set of generators, a $G$-equivariant structure (in the naive sense) on a module over
$\scrF(\Sigma)$ determines a lift to a module over $\scrF(\tilde\Sigma)$,
and for Yoneda modules of equivariant objects of $\scrF(\Sigma)$ this
coincides with the geometric lifting under the covering map $\tilde\Sigma\to
\Sigma$.

\begin{Example} \label{Ex:easy}
Consider the action of $G=\bZ/p$ on $\scrF(\Sigma)$ determined by a
class $a\in H^1(\Sigma;\bZ)$ as in Lemma \ref{Lem:Zpaction}, and let
$(\xi,\gamma)\in \scrF(\Sigma)$ be an immersed curve with a rank one local system.
\begin{enumerate}
\item If $\langle a, [\gamma]\rangle = 0$, then $(\xi,\gamma)$ is strictly equivariant for
$G$, and its Yoneda module $\scrY_{(\xi,\gamma)}$ is equivariant in the
naive sense. A choice of equivariant structure corresponds to a choice of lift of
$\gamma$ to $\tilde\Sigma$, and the corresponding lift of $\scrY_{(\xi,\gamma)}$ to
$\scrF(\tilde\Sigma)^{mod}$ is then the Yoneda module for the chosen lift of
$(\xi,\gamma)$.
\item If $\langle a,[\gamma]\rangle = 1$ (or is coprime with $p$), then the 
the Yoneda module of $(\xi,\gamma)$ is not equivariant, but its
full orbit $\scrY_{(\xi,\gamma)}\rtimes G=\bigoplus_{g\in G}
\scrY_{g(\xi,\gamma)}$ admits a unique equivariant structure. The lift of
this equivariant module to
$\scrF(\tilde\Sigma)^{mod}$ is the Yoneda module for the lift of
$(\xi,\gamma)$ to $\tilde\Sigma$ (which consists of a single curve covering
$\gamma$ $p$-fold).
\end{enumerate}
\end{Example}

Generalising Example \ref{Ex:easy}, we wish to prove that an object $X \in D^{\pi}\scrF(\Sigma)$ with $\langle a, ch(X)\rangle = 0$ admits a $G$-equivariant structure. However, it is not obvious why the purely homological hypothesis on $X$ should force it to be even weakly  equivariant.  To prove this we will embed the finite group action into a $\bG_m$-action, where the homological condition will yield infinitesimal equivariance.  General machinery due to Seidel then implies 
weak $\bG_m$-equivariance, and \emph{a posteriori} $G$-equivariance.

\subsection{Categorical  $\bG_m$-actions}\label{Sec:rationalactions}

Let $\scrA$ be a proper $\bZ/2$-graded $A_{\infty}$ category with a strict $\bG_m$-action. 
There are several notions of a $\bG_m$-action on a module $M\in \scrA^{mod}$.  Briefly, following \cite{Seidel:Am_milnor_fibres} one says:
\begin{enumerate}
\item $\bG_m$ acts \emph{naively} if it acts linearly on the underlying vector spaces of $M$, with all
the structure maps being equivariant.
\item $\bG_m$ acts \emph{strictly} if there are higher order (in $\scrA$) terms to the action, i.e. for each $g\in \bG_m$ we have 
\[
\rho^1_g = \{ \rho_g^{1,d+1}: M \otimes \scrA^d \to g^*M\}_{d\geq 0} \in \hom^0_{\scrA^{mod}}(M,g^*M)
\]
satisfying unitality and cocycle conditions,
and a rationality condition that the maps $\rho^{1,d+1}_g$ are coherently induced from a single map $\rho^{1,d+1}: M \otimes \scrA^d \to \bK[\bG_m] \otimes M[-d]$.
\item $\bG_m$ acts \emph{weakly} if there are module homomorphisms
$\rho^1_g$ as above, for which the
unitality and cocycle conditions hold at the cohomological level (in $H^0(\scrA^{mod})$).
\item $\bG_m$ acts \emph{up to homotopy} if there are higher order (in $\bG_m$) terms to the action, i.e. we now have maps $\rho^i_{(g_{i},\ldots,g_1)} \in \hom_{\scrA^{mod}}^{1-r}(M, g_1^*g_2^*\cdots g_{i}^*M)$, which should again satisfy appropriate unitality, cocycle and rationality
conditions. 
\end{enumerate}
A naive action yields a strict action with no higher order terms (so $\rho^{1,d+1}_g = 0$ for $d> 0$), whilst a strict action on a module gives a naive action on the quasi-isomorphic module $M \otimes_{\scrA} \scrA$, so these are essentially equivalent notions.

Unless $\scrA$ admits a set of generators which are strictly fixed by $\bG_m$,
the notion of rationality is best formulated in the algebro-geometric language of \cite[Appendix A]{Seidel:PL_dilating}.
We will consider situations where $\bG_m$ acts freely on the objects of $\scrA$
and strictly on morphisms. Concretely, in the case of Fukaya categories, this
is achieved by setting the objects to be pairs $(L,g)$ where $L$ is a Lagrangian brane and
$g$ is an element of the group; the object $(L,g)$, also denoted $g(L)$, is
obtained from $L$ by the action of $g$ and equipped with Floer perturbation data which
is pulled back from that for $L$.

A $\bK[\bG_m]$-module is just a quasicoherent sheaf over $\bK^*$, hence has stalks at points $g\in \bK^*$.  We will identify $\bK[\bG_m]^{\otimes d}$ with $\bK[\bG_m^{\times d}]$, and note there is a natural morphism $\bK[\bG_m] \to \bK[\bG_m]^{\otimes d}$ dual to the total multiplication map $\bG_m^{\times d} \to \bG_m$.
A rational $\bG_m$-action on $\scrA$ (with $\bG_m$ acting freely on objects and
strictly on morphisms) is then, by definition, an $A_\infty$-category
$\underline{\scrA}$, with the same objects as $\scrA$, in which
\begin{enumerate}
\item  the morphism groups $\hom_{\underline\scrA}(X_0,X_1)$ are $\bZ/2$-graded projective $\bK[\bG_m]$-modules, whose fibre at $g\in \bG_m$ is $\hom_{\underline\scrA}(X_0,X_1)_g:=\hom_\scrA(g(X_0),X_1)$;
using strictness, $\hom_{\scrA}(g_0(X_0),g_1(X_1))=\hom_{\underline\scrA}(X_0,X_1)_{g_1^{-1}g_0}$ for all $X_0,X_1\in \mathrm{ob}\,\scrA$ and $g_0,g_1\in \bG_m$;

\item the $A_{\infty}$-operations define $\bK[\bG_m^{\times d}]$-module maps 
\[
\mu^d_{\underline\scrA}: \hom_{\underline\scrA}(X_{d-1},X_d) \otimes_{\bK} \cdots \otimes_{\bK} \hom_{\underline\scrA}(X_0,X_1) \to \bK[\bG_m^{\times d}] \otimes_{\bK[\bG_m]} \hom_{\underline\scrA}(X_0,X_d)[2-d]
\]
which satisfy appropriate associativity equations, and a unitality condition (fiberwise, these are just the usual axioms for $\scrA$ to be an $A_\infty$-category);

\item for all $X_0,X_1\in \mathrm{ob}\,\scrA$ and $g_0,g_1\in \bG_m$, the $\bK[\bG_m]$-module $\hom_{\underline\scrA}(g_0(X_0),g_1(X_1))$ is the pullback of
$\hom_{\underline\scrA}(X_0,X_1)$ under multiplication by $g_1^{-1}g_0$, and
the $A_\infty$-operations on these modules strictly coincide under these identifications.
\end{enumerate}

(A small difference between our exposition and \cite[Appendix A]{Seidel:PL_dilating} is that
Seidel takes objects of $\scrA$ to be pairs consisting of an object of
$\underline\scrA$ and an element of $\bG_m$, i.e.\ the objects of
$\underline\scrA$ are $\bG_m$-orbits of objects of $\scrA$; the additional
objects in our version of $\underline\scrA$ contain no additional
information, since their morphism spaces and the $A_\infty$-operations on
those are completely determined by condition (3)).

We will consider the case where $\underline\scrA$ is {\em proper}, i.e.\ the cohomology groups
$H^*(\hom_{\underline\scrA}(X_0,X_1))$ are bounded $\bK[\bG_m]$-modules and finitely generated in each degree;
this implies properness of $\scrA$.

Informally, an object $X \in \scrA$ is homotopy $\bG_m$-equivariant if it is isomorphic to all of its images $g(X)$ for $g\in \bG_m$, in a manner which is coherent up to higher homotopy data. 
Formally, a (rational) \emph{homotopy equivariant} structure on $X$ is a sequence 
\[
\rho^i_X \in \bK[\bG_m^{\times i}] \otimes_{\bK[\bG_m]} \hom_{\underline\scrA}^{1-i}(X,X), \quad i\geq 1
\]
which, stalkwise, give elements
\[
\rho^i_{X, g_{i},\ldots,g_1} \in \hom^{1-i}_{\scrA}(g_{i}\cdots g_1 (X),X)
\]
which satisfy the following associativity equations
\cite[Appendix]{Seidel:PL_dilating}:
\[
\sum_{\substack{d\geq 1\\ i_1+\dots+i_d=i}}
\mu^d_{\scrA}(\rho^{i_d}_{X,g_{i},\ldots,g_{i_1+\dots+i_{d-1}+1}},\ldots, \rho^{i_1}_{X,g_{i_1},\ldots, g_1}) +
\sum_{1\leq k<i} (-1)^{k} \rho^{i-1}_{X,g_i,\ldots, g_{k+1}g_k, \ldots, g_1} = 0
\]
and the unitality condition $[\rho^1_{X,e}] = [e_X] \in H^0(\hom_{\scrA}(X,X))$. 

In this language, a weakly $\bG_m$-equivariant structure on $X$ is the data
of $\rho^1_X$ which satisfies unitality, and for which there exists
$\rho^2_X$ such that the first two associativity equations hold, namely
\[
\mu^1_\scrA(\rho^1_{X,g})=0 \quad \text{and} \quad
\mu^2_\scrA(\rho^1_{X,g_2},\rho^1_{X,g_1})-\rho^1_{X,g_2g_1}+\mu^1_\scrA(\rho^2_{X,g_2,g_1})=0.
\]

The machinery in \cite[Section 8]{Seidel:Am_milnor_fibres} upgrades weak
equivariant structures to homotopy equivariant structures using an
obstruction theory argument, and turns homotopy equivariant objects into
naively equivariant modules.
For applications, the essential point is therefore to find criteria for the
existence of weak actions.

\subsection{Infinitesimal equivariance}

Denoting by $\mathfrak{g}$ the Lie algebra of $\bG_m$, we have a short exact sequence of $\bK[\bG_m]$-modules
\[
0 \to \mathfrak{g}^{\vee} \to \bK[\bG_m]/\calI_e^2 \to \bK \to 0
\]
where $\calI_e$ is the maximal ideal at the point $e\in \bG_m$, and
$\bK=\bK[\bG_m]/\calI_e$.
Tensoring this with the (projective, hence flat) module $\hom_{\underline\scrA}(X,X)$
and considering the resulting long exact sequence in cohomology, the image of
$[e_X]\in H^0(\hom_{\scrA}(X,X))$ under the connecting morphism yields an element
\begin{equation} \label{eqn:deformation_class}
\Def^0_X \in \mathfrak{g}^{\vee} \otimes H^1(\hom_{\scrA}(X,X)).
\end{equation}

If this vanishes then we say $X$ is ``infinitesimally equivariant".

\begin{Proposition}\label{Prop:inf_enough} Suppose $\scrA$ is proper, has a rational
$\bG_m$-action (in the sense of the preceding section), and $\bK$ has characteristic zero. If $X$ is infinitesimally equivariant and simple, i.e. $H^0(\hom_{\scrA}(X,X)) = \bK$, then $X$ admits a
weakly equivariant structure.
\end{Proposition}

\begin{proof} 
The proof follows along the same lines as Seidel's argument for the case $\bK = \bC$,
which is stated as \cite[Theorem 2.7 and Corollary A.12]{Seidel:PL_dilating}
for $A_\infty$-categories with naive and rational $\bC^*$-actions,
respectively.
Rationality of the action and properness imply that the cohomology $H^0(\hom_{\underline\scrA}(X,X))$ defines a coherent sheaf over
$\bG_m$, and infinitesimal equivariance equips that sheaf with an algebraic
connection. More precisely,  a choice of primitive 
\[
def^0_X \in \mathfrak{g}^{\vee} \otimes \hom_{\scrA}^0(X,X) \quad \mathrm{for} \quad \Def^0_X = 0 \in \mathfrak{g}^{\vee} \otimes H^1(\hom_{\scrA}(X,X))
\]
equips the sheaf  with a distinguished connection (cf.\ \cite[Lemma A.3]{Seidel:PL_dilating} and \cite[Section 7a]{Seidel:Am_milnor_fibres}).
In characteristic zero, any coherent sheaf admitting a connection is locally free \cite[Corollary 2.5.2.2]{Andre}, so the sheaf
$\calF=H^0(\hom_{\underline\scrA}(X,X))$ is locally free of rank one, i.e.\ a line bundle.

Over $\bC$, the construction in \cite[Lemma 7.12]{Seidel:Am_milnor_fibres}  uses surjectivity of the exponential map to
modify the algebraic connection and trivialise the monodromy. A covariant constant section taking the value $[e_X]$ at $e\in \bG_m$ then provides a
$\bG_m$-family of cohomological isomorphisms which obey the group law, and hence defines a weak $\bG_m$-action
(which can be lifted to a homotopy $\bG_m$-action using a general obstruction theory argument).

Following  \cite[Remark 14.22]{Seidel:categorical_dynamics}, we instead argue as follows.
The total space of the line bundle $\calF$ carries an action of an algebraic group $\mathcal{G}$
which fits into an extension 
\[
1 \to \mathrm{Aut}(X) \to \mathcal{G} \to \bG_m \to 1
\]
where the subgroup $\mathrm{Aut}(X)=H^0\hom_\scrA(X,X)^*$ acts by multiplication on the fibres
$\calF_g=H^0\hom_\scrA(gX,X)$, and the quotient $\bG_m$ acts on the base by multiplication.

Since $X$ is simple, $\Aut(X) = \bG_m$.  Over any perfect (e.g. characteristic zero) field, the group of self-extensions $\Ext^1(\bG_m,\bG_m) = 0$, cf. for instance \cite[Chapter 5]{Brion}.  
Therefore $\mathcal{G} \cong \bG_m\times\bG_m$, the extension splits,
and the action of $\bG_m$ on itself by multiplication admits a lift to an
action of $\bG_m$ on the total space of $\calF$. The orbit of $[e_X]\in
\calF_e$ then provides the desired $\bG_m$-family of cohomological
isomorphisms.
\end{proof}

\subsection{$\bG_m$-actions from the Picard group}

We now apply the machinery of rational $\bG_m$-actions to the setting of 
Fukaya categories of surfaces. We first state the result for a closed surface $\Sigma$ of genus $g\geq
2$, which is our main focus.

\begin{Proposition}\label{Prop:giveanaction}
A choice of class $a \in H^1(\Sigma;\bZ)$ defines a rational $\bG_m$-action on
$\scrF(\Sigma)$, for which $\underline\scrF(\Sigma)$ is proper in the sense
of Section \ref{Sec:rationalactions}.
\end{Proposition}

\begin{proof}
Choose a closed differential form $\alpha \in \Omega^1(\Sigma;\bR)$ representing $a$, and let $V \in
\Gamma(T\Sigma)$ be the symplectic vector field $\omega$-dual to $\alpha$.   Consider the actions of $s \in \bR$ and of $\lambda \in U_{\Lambda}$ on Lagrangian branes in $\Sigma$ given by 
\begin{itemize}
\item the symplectomorphisms $\phi^s$ associated to the time $s$ flow of $V$, with flux $s\cdot a$, and 
\item twisting by the line bundle $\xi_{\lambda} \to \Sigma$ which is topologically trivial and has connection $d+\lambda\cdot\alpha$ and hence holonomy $\exp(\lambda\cdot\alpha)$.
\end{itemize}

Using the fact that $\mathcal{L}_V(\alpha) = 0$, one can check that the actions of $\bR$ and $U_{\Lambda}$ \emph{strictly} commute, and define an action of $\bG_m$ on the set of Lagrangian submanifolds with local systems, where
$q^s\lambda\in \bG_m$ acts by $\phi^s \circ (\otimes \xi_{\lambda})$. 

We now enlarge the Fukaya category, following \cite[Section 10b]{Seidel:FCPLT}, to allow pairs $(L,g)$ with $L$ a brane and $g\in \bG_m$, where the perturbation data for $(L,g)$ is the $g$-pullback of that for $L$. 
Then $\bG_m$ acts strictly on Lagrangian branes, via $A_{\infty}$-functors with no higher order terms.
The ideas of Section \ref{Sec:analyticity}, cf.  \cite{Seidel:flux} and \cite{Kartal}, and in particular  
Theorem \ref{Thm:Floer_analytic} imply that the morphism groups $CF^*(g(L),L')$ form analytic sheaves of complexes over
$\bG_m$. 
We will explain below, using ideas from Section \ref{Sec:restriction}, that
these morphism groups can in fact arranged to be the stalks of (infinite rank but cohomologically proper) projective
$\bK[\bG_m]$-modules.
The fact that the $A_{\infty}$-operations are compatible with the group operation follows from strictness of the action. 

The reason why the projective $\bK[\bG_m]$-modules $\hom_{\underline\scrF(\Sigma)}(L,L')$
typically have infinite rank, is that the Floer complexes $CF^*(g(L),L')$ have
unbounded rank, since deforming $L$ by an increasingly large amount of flux
may create unboundedly many new intersections with $L'$. These new intersections come in
cancelling pairs, and because $\Sigma$ has genus at least two, only finitely
many such pairs of intersections ever contribute to the Floer cohomology
$HF^*(g(L),L')$, with the rest belonging to summands which remain acyclic for all $g$.
Still, in order to define the chain-level $A_\infty$-structure it is simpler
to work with the whole Floer complex. In our setting, there is a direct
geometric approach to turning the Floer complexes $CF^*(g(L),L')$ into
projective (in fact, free) $\bK[\bG_m]$-modules.

Namely, represent the Poincar\'e dual to $a$ (or some multiple of it, if $a$
is not a primitive element of $H^1(\Sigma;\bZ)$) by a simple closed curve
$\sigma$, and pick the 1-form $\alpha$ and vector field $V$ to be supported in
an annular neighbourhood $\sigma\subset A\subset \Sigma$, which we identify
with $(-2,2)\times S^1$ as in Section \ref{Sec:restrict}, with $\alpha$
pulled back from $(-2,2)$ so that $V$ is everywhere parallel to the $S^1$-factor.
As in Section \ref{Sec:restrict}, only consider objects of $\scrF(\Sigma)$
which are not parallel to the annulus $A$ and intersect its boundary
minimally (annulus-parallel curves get replaced by twisted complexes as in
Example \ref{Ex:restriction_example}).  Then the Lee perturbations commute
with the $\bG_m$-action, and the various properties of the Floer complexes
$CF^*(L,L';wH)$ for large $w$ carry over to $CF^*(g(L),L';wH)$ as long as
$w$ is sufficiently large compared to the valuation of $g\in \bG_m$.
Because the continuation maps $CF^*(g(L),L';wH)\to CF^*(g(L),L';(w+1)H)$ for sufficiently large $w$ are chain-level inclusions which
define strict $A_\infty$-homomorphisms with no higher order terms, the telescope models
for $CW^*(g(L),L')=\varinjlim CF^*(g(L),L';wH)$ can be replaced with naive
limits: define $\hom(g(L),L')$ to be the union of the increasing sequence of complexes
$CF^*(g(L),L';wH)$ for $w=N,N+1,\dots$ starting from sufficiently large
$N\gg \mathrm{val}(g)$, identifying each with a subcomplex of the next by inclusion.
Equivalently, declare the generators of
$\hom(g(L),L')$ to be the arcs of intersections given by Lemma
\ref{Lem:persists},  with $A_\infty$-operations given by counts of
perturbed holomorphic discs for any value of $w$ which is sufficiently large
relative to the given inputs and the valuation of $g$, corrected by action rescalings as
in Lemma \ref{Lem:inclusions}. There is an obvious identification between
generators for varying values of $g$, and the fact that the $A_\infty$-operations vary
algebraically with $g\in \bG_m$ is a direct consequence of Lemma
\ref{Lem:analytic}; thus the Floer complexes constructed in this manner
naturally define projective $\bK[\bG_m]$-modules underlying a rational
$\bG_m$-action.

Even though the $\bK[\bG_m]$-modules $\hom_{\underline\scrF(\Sigma)}(L,L')$
constructed in this manner are not finitely generated, the cohomology
$H^*\hom_{\underline\scrF(\Sigma)}(L,L')$ is a finitely generated
$\bK[\bG_m]$-module, because all but finitely many of the generators created
upon wrapping come in pairs which are connected by a single bigon
contributing to the Floer differential, hence lie in acyclic summands of the
chain-level $\bK[\bG_m]$-module. Here we use crucially the fact that
$\Sigma$ is not a (closed) torus. We conclude that $\underline\scrF(\Sigma)$ is proper. 
\end{proof}

\begin{Remark}
Proposition \ref{Prop:giveanaction} also holds for a surface $S$ of genus $\geq 1$ with non-empty
boundary, under the additional assumption that the restriction of $a\in H^1(S;\bZ)$ to $\partial S$
vanishes. This ensures that the class $a$ can be represented by a one-form supported in
an annular neighborhood of a simple closed curve, and the argument then
proceeds exactly as in
the case of closed surfaces. Moreover, the construction applies equally well to
the wrapped Fukaya category, and gives a rational $\bG_m$-action on $\scrW(S)$
(however $\underline\scrW(S)$ is not proper).
\end{Remark}

\begin{Example} \label{Ex:deformationclass_scc}
 Suppose $\gamma \subset \Sigma$ is a simple closed curve, and $a\in H^1(\Sigma;\bZ)$.
Then it is clear from the definition of the $\bG_m$-action that the class $\Def^0_{\gamma} \in
\mathfrak{g}^\vee\otimes HF^1(\gamma,\gamma) \simeq H^1(\gamma;\bK)$ is
$a|_{\gamma}$, or more precisely, the homomorphism from $\mathfrak{g}\simeq \bK$ to
$HF^1(\gamma,\gamma)$ given by multiplication by~$a|_{\gamma}$.
Thus, $\gamma$ is infinitesimally equivariant if and only if $\langle a,\gamma\rangle =
0$, as expected.
\end{Example}

The rational $\bG_m$-action on $\scrF(\Sigma)$ immediately extends to
twisted complexes and to the idempotent completion $\scrF(\Sigma)^{per}$;
the properness of $\underline\scrF(\Sigma)$ implies that of
$\underline\scrF(\Sigma)^{per}$.

\begin{Lemma} \label{Lem:deformationclass}
Let $X\in \scrF(\Sigma)^{per}$ be a spherical object with $ch(X)\neq 0$, and $a\in
H^1(\Sigma;\bZ)$. Then $X$ is infinitesimally equivariant with respect to the action of
Proposition \ref{Prop:giveanaction} if and only if $\langle a,ch(X)\rangle=0$.
\end{Lemma}

\begin{proof}
The deformation classes $\Def^0_X$, $X\in
\mathrm{ob}\,\scrF(\Sigma)$ are induced by a Hochschild cohomology class
$\mathbf{Def}\in \mathfrak{g}^\vee\otimes
HH^1(\scrF(\Sigma),\scrF(\Sigma))$, which measures the
infinitesimal action of $\bG_m$ on morphisms
\cite{Seidel:PL_dilating}. 
(Since $\mathfrak{g}^\vee\simeq \bK$, we can think of this as a single
Hochschild cohomology class, and henceforth we drop $\mathfrak{g}^\vee$ from
the notation).
Recall that the closed-open map $CO:H^1(\Sigma;\Lambda)\to
HH^1(\scrF(\Sigma),\scrF(\Sigma))$ is an isomorphism. Thus, the fact that
$\Def^0_\gamma=a_{|\gamma}$ for all simple closed curves implies that
$\mathbf{Def}=CO(a)$. 

The Yoneda product makes Hochschild cohomology a unital algebra, over which
the {\em cap product} endows the Hochschild homology $HH_*(\scrF(\Sigma),\scrF(\Sigma))$
with a module structure. It is well-known to experts 
that the closed-open map is a ring homomorphism, and that the open-closed
map $OC:HH_*(\scrF(\Sigma),\scrF(\Sigma))\to H^{*+1}(\Sigma;\Lambda)\simeq
H_{1-*}(\Sigma;\Lambda)$ is a homomorphism
of $H^*(\Sigma;\Lambda)$-modules (see e.g.\ \cite{Ganatra2013}, \cite{Ritter-Smith}).
With this understood, we can view $\Def^0_X\in H^1\hom(X,X)$ as a Hochschild
cycle, and the corresponding class $[\Def^0_X]\in HH_1(\scrF(\Sigma),\scrF(\Sigma))$ can be expressed
as $\mathbf{Def}\cap [e_X]=CO(a)\cap [e_X]$, which implies that
\[
OC([\Def^0_X])=OC(CO(a)\cap [e_X])=a\cap OC([e_X])=a\cap ch(X)\in H_0(\Sigma;\Lambda).
\]
Since the open-closed map is an isomorphism, we conclude that $[\Def^0_X]\in
HH_1(\scrF(\Sigma),\scrF(\Sigma))$ vanishes if and only if $\langle
a,ch(X)\rangle=0$. In order to reach the same conclusion for $\Def^0_X\in H^1\hom(X,X)$,
it remains to verify that the map $H^1\hom(X,X)\to
HH_1(\scrF(\Sigma),\scrF(\Sigma))$ is an isomorphism. Since $X$ is
spherical, both sides have rank one, and it is enough to show that the map
does not vanish identically; since $ch(X)\neq 0$ this follows from the
existence of classes $a$ for which $\langle a,ch(X)\rangle\neq 0$.
\end{proof}

We note that the assumption that $ch(X)\neq 0$ can be removed by using the
Calabi-Yau structure on $\scrF(\Sigma)$; more generally the statement is
expected to hold for all indecomposable objects of $\scrF(\Sigma)^{per}$.

We also remark that a similar criterion for infinitesimal
equivariance can be formulated for spherical objects on a surface with boundary, 
using inclusion into a larger closed surface to reduce to
Lemma \ref{Lem:deformationclass}.

\section{Conclusions}

\subsection{Spherical objects revisited} We can now conclude the proof of Theorem \ref{Thm:Main}.

\begin{Corollary} Let $X \in \scrF(\Sigma)^{per}$ be spherical, with $ch(X)\neq 0$. Suppose there is a class $a \in H^1(\Sigma;\bZ)$ with the property that $\langle a, ch(X)\rangle  = 0$. 
Then for each $p$, $X$ is quasi-isomorphic to an object of
$\scrF(\Sigma)^{per}$ which admits a naive equivariant structure for the $\bZ/p$-action of Lemma \ref{Lem:Zpaction}
and lifts to the degree $p$ covering $\tilde{\Sigma} \to \Sigma$ as an
object $\hat{X} \in \scrF(\tilde{\Sigma})^{per}$.
\end{Corollary}

\begin{proof} By Lemma \ref{Lem:deformationclass}, the hypothesis $\langle a,ch(X)\rangle  = 0$
implies that $X$ is infinitesimally equivariant for the $\bG_m$-action
associated to $a$ by Proposition \ref{Prop:giveanaction}. 
Proposition \ref{Prop:inf_enough} then shows that $X$ is weakly $\bG_m$-equivariant in
the sense of Section \ref{Sec:rationalactions}. Restricting to the finite
subgroup of $p$-th roots of unity, we find that $X$ is weakly $\bZ/p$-equivariant
with respect to the action of Lemma \ref{Lem:Zpaction}, and so by Lemma \ref{Lem:weaktonaive} 
its Yoneda module is quasi-isomorphic to a naively equivariant module.  
One can then appeal to the results of Section \ref{Sec:finite_covers} and
reinterpret this naively $\bZ/p$-equivariant $\scrF(\Sigma)$-module as a
module over $\scrF(\tilde\Sigma)$. 

Since $\scrF(\Sigma)$ is proper, the Yoneda module of $X$ is proper,
and so are the equivariant module and its lift; the homological smoothness
of $\scrF(\tilde\Sigma)$ then implies that the lifted module is perfect,
hence can be represented by an object $\hat{X}\in
\scrF(\tilde\Sigma)^{per}$. Finally, projecting $\hat{X}$ back to $\Sigma$ yields
a naively equivariant object of $\scrF(\Sigma)^{per}$ which is
quasi-isomorphic to $X$.
\end{proof}
 
We now show that certain objects are supported on subsurfaces of $\tilde{\Sigma}$, in the following sense.

\begin{Corollary} \label{Cor:lift_disjoint_from_curve}
If $X \in \scrF(\Sigma)^{per}$ is spherical, with $ch(X)\neq 0$, we can find a cover $\tilde{\Sigma} \to \Sigma$ and a lift $\hat{X}$ of $X$ to
$\tilde{\Sigma}$ and a simple closed curve $\gamma \subset \tilde{\Sigma}$ for which $HF^*(\hat{X}, \gamma) = 0$.
\end{Corollary}

\begin{proof} Pick a class $a \in H^1(\Sigma;\bZ)$ with $\langle a, ch(X)\rangle = 0$.
By the previous Corollary, $X$ is $G=\bZ/p$-equivariant for the action of
Lemma \ref{Lem:Zpaction}, and can be lifted to the degree $p$ covering
$\tilde\Sigma\to\Sigma$, giving an object $\hat{X}\in \scrF(\tilde\Sigma)^{per}$.

Let $\delta \subset \Sigma$ be any simple closed curve such that $\langle
a,[\delta]\rangle=0$: then $\delta$ lifts to a simple closed curve
$\hat\delta$ in $\tilde\Sigma$, and by the results of Section
\ref{Sec:finite_covers} we have \begin{equation} \label{eqn:HF_upstairs}
HF^*(X,\delta) = \textstyle\bigoplus\limits_{\chi\in G^{\vee}} HF^*(\hat{X},
\hat\delta^\chi),
\end{equation}
where we recall that the objects $\hat\delta^\chi$, $\chi\in G^\vee$ are the $p$ lifts of
$\delta$ to $\tilde\Sigma$. 
Then for sufficiently large $p$ ($p>\mathrm{rank}\,HF^*(X,\delta)$) some summand on the RHS vanishes, 
giving a simple closed curve $\gamma=\hat\delta^\chi\subset \tilde{\Sigma}$ with $HF^*(\hat{X},\gamma) = 0$. \end{proof}

\begin{Lemma} \label{Lem:lift_with_many_punctures}
In the situation of Corollary \ref{Cor:lift_disjoint_from_curve}, one can choose
$\tilde\Sigma$ so that there are two disjoint homologically independent
simple closed curves $\gamma_1,\gamma_2\subset \tilde\Sigma$ with
$HF^*(\hat{X},\gamma_i)=0$.
\end{Lemma}

\begin{proof}
We proceed as in the proof of Corollary \ref{Cor:lift_disjoint_from_curve}.
Since the genus of $\Sigma$ is
at least two, we can find two simple closed curves $\delta,\delta'$ in
$\Sigma$ which have geometric intersection number one and both pair
trivially with $a$. The $p$ lifts of $\delta$ to $\tilde\Sigma$ are disjoint simple closed
curves, and each of them intersects
precisely one of the $p$ lifts of $\delta'$, so their homology
classes are linearly independent in $H_1(\tilde\Sigma;\bZ)$. The result now follows by arguing as in
the previous Corollary and taking $p$
sufficiently large to force the vanishing of at least two of the Floer
cohomology groups $HF^*(\hat{X},\hat\delta^\chi)$.
\end{proof}

\begin{Corollary}[=Theorem \ref{Thm:Main}]
A spherical object $X \in D^{\pi}\scrF(\Sigma)$ with non-zero Chern character is quasi-isomorphic to a simple closed curve equipped with a rank one local system.
\end{Corollary}

\begin{proof}
By Lemma \ref{Lem:lift_with_many_punctures}, we can find an action of
$G=\bZ/p$ and a finite covering
$\tilde\Sigma\to \Sigma$ for which $X$ is $G$-equivariant and lifts to an object $\hat{X}\in
\scrF(\tilde\Sigma)^{per}$, and two disjoint, homologically independent
simple closed curves $\gamma_1,\gamma_2\subset \tilde\Sigma$ with
$HF^*(\hat{X},\gamma_i)=0$. We now check that $\hat{X}$ is spherical and
$ch(\hat{X})$ is non-zero.

Recalling that the
Fukaya category of $\tilde\Sigma$ gives a model for the $G$-equivariant
Fukaya category of $\Sigma$ (cf.\ Section \ref{Sec:finite_covers}),
$HF^*(\hat{X},\hat{X})$ is isomorphic to the $G$-invariant summand of
$HF^*(X,X)\simeq H^*(S^1;\bK)$. General properties of equivariant
objects imply that the unit $e_X$ is $G$-invariant,
so $HF^0(\hat{X},\hat{X})$ has rank one, and Poincar\'e duality (or
vanishing of the skew-symmetric pairing $\chi\Hom(\hat{X},\hat{X})$) implies that
$HF^1(\hat{X},\hat{X})$ also has rank one. Hence $\hat{X}$ is spherical.
Meanwhile, the assumption that $ch(X)\in H_1(\Sigma;\bZ)$ is non-zero means that there exists a
simple closed curve $\eta$ such that $\chi\Hom(X,\eta)\neq 0$. Denoting by
$\hat\eta$ the total preimage of $\eta$ in $\tilde\Sigma$ (which may consist of one or more
simple closed curves), we find that $\chi\Hom(\hat{X},\hat\eta)=
\chi\Hom(X,\eta)\neq 0$, so $ch(\hat{X})$ is a non-zero element of
$H_1(\tilde\Sigma;\bZ)$.

Corollary \ref{Cor:non-filling_implies_geometric} now implies that $\hat{X}$ is quasi-isomorphic to a simple closed curve
$(\hat\xi,\hat\sigma)$ with rank one local system in $\scrF(\tilde\Sigma)$. 
The last step is to descend back from $\tilde\Sigma$ to $\Sigma$.
The projection of $\hat\sigma$ to $\Sigma$ is a closed ({\em a priori}
immersed) curve $\sigma$ in $\Sigma$, whose homology class satisfies $\langle
a,[\sigma]\rangle=0$, and the local system $\hat\xi$ descends to a rank one
local system $\xi$ on $\sigma$. As noted in Section \ref{Sec:finite_covers},
the object $(\xi,\sigma)\in \scrF(\Sigma)$ admits $p$ distinct $G$-equivariant structures,
and one of these corresponds to the lift $(\hat\xi,\hat\sigma)\in \scrF(\tilde\Sigma)$.
A quasi-isomorphism between $\hat{X}$ and $(\hat\xi,\hat\sigma)$ then
descends to a $G$-equivariant quasi-isomorphism between $X$ and
$(\xi,\sigma)$ (i.e., a quasi-isomorphism which lies in the $G$-invariant part of the
Floer complex); from which it follows that $X$ and $(\xi,\sigma)$ are also
quasi-isomorphic (non-equivariantly) in $\scrF(\Sigma)^{per}$. Corollary
\ref{Cor:immersed_spherical} then implies that $\sigma$ is quasi-isomorphic
to an embedded simple closed curve.
\end{proof}

\begin{Remark}
The exotic spherical object of Figure \ref{Fig:bad} lifts to an embedded simple closed curve on a double cover of the surface, so the ability to descend from
$\tilde{\Sigma}$ back to $\Sigma$ above is again making essential use of the homological hypothesis on $X$.  
\end{Remark}

\subsection{A Floer-theoretic Schmutz graph}\label{Sec:Schmutz}
The ``Schmutz graph", introduced in \cite{SchmutzSchaller}, has vertices non-separating simple closed curves up to isotopy, and two such are joined by an edge exactly when they have geometric intersection number one. 
The main theorem of \emph{op. cit.} asserts that the group of simplicial automorphisms of the Schmutz graph is the
extended mapping class group (of both orientation-preserving and reversing automorphisms) for a surface of genus $g\geq 3$, and the quotient of the mapping class group by the hyperelliptic involution $\iota$ when $g=2$.

We now introduce a Floer-theoretic analogue. By Theorem \ref{Thm:Main},
spherical objects of $D^\pi \scrF(\Sigma)$ with non-zero Chern character
correspond to homologically essential simple closed curves with rank one
local systems. Moreover, the simple closed curves
underlying two spherical objects $X_1,X_2$ have geometric intersection number one
if and only if the rank of $HF^*(X_1,X_2)$ is equal to 1, by Corollary \ref{Cor:rankHF}.
We introduce an equivalence relation on the set of spherical objects by
declaring that $X\sim X'$ if and only if they have Floer cohomology rank one
with the same set of spherical objects:
\begin{equation}\label{eq:HFequivalence}
X\sim X'\quad \Longleftrightarrow \quad \left(\text{for all spherical } Y,
\ \mathrm{rk}\,HF^*(X,Y)=1 \Leftrightarrow
\mathrm{rk}\,HF^*(X',Y)=1\right).
\end{equation}
 
\begin{Lemma}\label{Lem:isotopy_classes}
Two spherical objects $X,X'\in D^\pi\scrF(\Sigma)$ with non-zero Chern characters 
are equivalent if and only if the simple closed curves $\gamma,\gamma'$
underlying $X,X'$ are isotopic.
\end{Lemma}

\begin{proof}
Using Corollary \ref{Cor:rankHF}, $X\simeq X'$ if and only if $\gamma,\gamma'$ have 
geometric intersection number one with the exact same set of simple closed curves. 
It is a general fact of surface topology that this can only happen when
$\gamma,\gamma'$ are isotopic. For example, if $\gamma$ and $\gamma'$ were
not isotopic, then the Schmutz graph would admit an automorphism which exchanges the vertices
corresponding to $\gamma$ and $\gamma'$ while preserving every other vertex;
this cannot be induced by an element of the extended mapping class group.
\end{proof}

Since our equivalence relation on spherical objects is clearly preserved by
auto-equivalences of the Fukaya category, Lemma \ref{Lem:isotopy_classes}
implies that auto-equivalences act on the set of isotopy classes of
non-separating simple closed curves. This remains true for surfaces with
boundary, by the same argument; cf.\ Remark \ref{Rmk:Schmutz_open}. (However, spherical objects supported on boundary-parallel
curves can be permuted arbitrarily by auto-equivalences of the compact
Fukaya category; auto-equivalences of the wrapped Fukaya category are
better-behaved.)

On closed surfaces, the machinery of \cite{Seidel:flux} actually allows us
to obtain a stronger result: recalling that every homologically essential spherical object of $\scrF(\Sigma)$
comes in a $\bG_m$-family, any auto-equivalence must act on such objects in
a way that maps $\bG_m$-families to $\bG_m$-families. Since we will not need
this result, we only sketch the argument.

\begin{proof}[Sketch]
Given a spherical object $X\in \scrF(\Sigma)$ with $ch(X)\neq 0$ (hence
primitive in $H_1(\Sigma;\bZ)$ as a consequence of Theorem \ref{Thm:Main}),
and a class $a\in H^1(\Sigma;\bZ)$ with $\langle
a,ch(X)\rangle=1$, the rational $\bG_m$-action of Proposition
\ref{Prop:giveanaction} yields an object of $\underline\scrF(\Sigma)$
representing the $\bG_m$-orbit of $X$, which in turn determines an object
$\scrX$ of the category $\scrF^{per}_{\bK^*}$ of
perfect families of $\scrF(\Sigma)$-modules over the base $\bK^*$
\cite[Section 1f]{Seidel:flux}. 
Example \ref{Ex:deformationclass_scc} and the
proof of Lemma \ref{Lem:deformationclass} show that this family of modules follows
a deformation field $[\gamma]\in HH^1(\scrF^{per}_{\bK^*},\Omega^1_{\bK^*}\otimes\scrF^{per}_{\bK^*})$ 
 which is constant over the base $\bK^*$
and given fibrewise by $CO(a)\in HH^1(\scrF(\Sigma),\scrF(\Sigma))$.
Now, an autoequivalence $G$ of $\scrF(\Sigma)$ induces a functor
$\scrG^{per}$ on $\scrF^{per}_{\bK^*}$, and $\scrG^{per}(\scrX)$ is a perfect family
of modules which follows the deformation field $[\gamma']=G_*([\gamma])$,
which is constant over the base and given fibrewise by $G_*(CO(a))$
\cite[Section 1i]{Seidel:flux}.
Since the closed-open map is an isomorphism,  there exists $a'\in
H^1(\Sigma;\Lambda)$ such that $G_*(CO(a))=CO(a')$; in fact $a'\in H^1(\Sigma;\bZ)$,
because $\langle a',ch(G(Y))\rangle=\langle a,ch(Y)\rangle\in \Z$ for all
spherical objects $Y$.
Repeat the above construction for the spherical object $X'=G(X)$ and the
$\bG_m$-action determined by the cohomology class $a'$, to build a perfect
family of modules $\scrX'\in \scrF^{per}_{\bK^*}$, which follows the deformation
field $[\gamma']$. The two families of modules $\scrG^{per}(\scrX)$ and $\scrX'$ 
agree over the origin and both follow the deformation field $[\gamma']$, so
by \cite[Proposition 1.21]{Seidel:flux} their fibres are quasi-isomorphic at
every point of $\bK^*$.
\end{proof}

The group of auto-equivalences $\Auteq(\scrF(\Sigma))$ contains a subgroup isomorphic to 
$H^1(\Sigma, \Lambda^{\ast})$, given by symplectic isotopies of arbitrary flux and 
tensoring by flat line bundles. We expect that a further elaboration on the
above argument implies that this is a normal subgroup.

We now return to the proof of Corollary \ref{Cor:Main}:
   
\begin{Proposition} \label{Prop:split}
There is a natural homomorphism $\Auteq(D^{\pi}\scrF(\Sigma)) \to \Gamma_g$, which is split by the (non-canonical) homomorphism $\Gamma_g \to \Auteq(D^{\pi}\scrF(\Sigma))$ constructed in Section \ref{Sec:Monodromy}.
\end{Proposition}

\begin{proof}
Define a graph $\Upsilon(\Sigma)$ as follows:
\begin{itemize}
\item vertices are spherical objects $X \in \scrF(\Sigma)^{per}$ with
non-zero Chern character, modulo 
the equivalence relation \eqref{eq:HFequivalence}; 
\item two distinct vertices $X_1$ and $X_2$ are joined by an edge whenever $HF(X_1, X_2) $ has rank 1
(this is clearly invariant under the equivalence relation).
\end{itemize}

Theorem \ref{Thm:Main} and Lemma \ref{Lem:isotopy_classes} imply that the vertices of the graph are in bijection with isotopy 
classes of homologically essential simple closed curves on $\Sigma$, while
Corollary \ref{Cor:rankHF} shows that the edges correspond to pairs of
curves with geometric intersection number one.
Thus, $\Upsilon(\Sigma)$ agrees with the Schmutz graph from \cite{SchmutzSchaller}.  
On the other hand, it is manifest that $\Auteq(\scrF(\Sigma)^{per})$ acts on $\Upsilon(\Sigma)$ by simplicial automorphisms.
Thus, we obtain a homomorphism from $\Auteq(\scrF(\Sigma)^{per})$ to the
extended mapping class group $\Gamma_g^\pm$ for $g\geq 3$, or to the
quotient $\Gamma_2^\pm/\langle \iota\rangle$ when $g=2$.

Since any autoequivalence preserves the pairing $\chi HF(\cdot,\cdot)$, i.e.\ algebraic intersection
numbers of simple closed curves, its action on the set of isotopy classes 
cannot be that of an orientation-reversing diffeomorphism. Therefore, the
homomorphism actually takes values in the ordinary (oriented) mapping class group
$\Gamma_g$ for $g\geq 3$, or in $\Gamma_2/\langle \iota\rangle$ when $g=2$.
When $g=2$, one can consider the action of $\Auteq(\scrF(\Sigma_2)^{per})$ on
the graph $\Upsilon^+(\Sigma)$ whose vertices are spherical objects
modulo a graded version of \eqref{eq:HFequivalence} which requires
$HF^*(X,Y)$ and $HF^*(X',Y)$ to be in the same degree when they both have
rank 1;
i.e., we now consider isotopy classes of {\em
oriented} simple closed curves. Since the hyperelliptic involution reverses the orientation on all simple closed curves, 
this allows one to lift the homomorphism from $\Gamma_2/\langle \iota\rangle$ to $\Gamma_2$.  The fact that the action on the Floer-theoretic Schmutz graph is split by the construction of Section \ref{Sec:Monodromy} is straightforward.
\end{proof}

\begin{Remark} \label{Rmk:Schmutz_open}
The Schmutz graph can be defined analogously for isotopy classes of
non-separating simple closed curves on a punctured surface of genus
$g\geq 1$ with $n$ punctures.
It is known \cite{SchmutzSchaller} that its simplicial isometry group reproduces the extended mapping class group, modulo the hyperelliptic involution for $(g,n) \in \{(1,1),
(1,2)\}$.  Starting from here, and using the geometricity of spherical objects obtained in Corollary \ref{Cor:spherical_on_punctured}, one sees that the analogue of Proposition \ref{Prop:split} holds for punctured surfaces, i.e. the autoequivalence group of the compact Fukaya category determines (and surjects onto) the mapping class group.  \end{Remark}


\section{An application to symplectic mapping class groups\label{Sec:application}}

In this final section, we prove Theorem \ref{Thm:largeMCG} as an application of Corollary \ref{Cor:Main}.  The argument is a fairly straightforward adaption of ideas from \cite{Seidel:flux, Smith:HFQuadrics}, but involves somewhat different technology from that in the main body of the paper, so we will be relatively brief. We also leave the realm of strictly unobstructed Lagrangians; our main examples satisfy a weak monotonicity property, but we will not labour foundational aspects of the Fukaya category.

\subsection{Fukaya category summands}

Let $(M,\omega)$ be a closed symplectic manifold.  Define the preliminary category $\scrF_{pr}(M)$ to be a curved $\bZ_2$-graded linear $A_{\infty}$-category over  $\Lambda_{\geq 0}$ which has objects oriented spin Lagrangian submanifolds $L\subset M$ equipped with finite-dimensional $\Lambda_{\geq 0}$-local systems $\xi \rightarrow L$.  The morphism groups, in the two most important cases, are given by 
\begin{equation}
hom_{\scrF_{pr}(M)}((L,\xi),(L',\xi')) =  \begin{cases}  \ C^*(L; Hom(\xi,\xi')) &  L = L' \\  \ \oplus_{x\in L\cap L'} \xi_x^{\vee} \otimes \xi'_x & L \pitchfork L'  \end{cases} \end{equation}
where in the first case we take  any fixed finite-dimensional chain-level model for the classical cohomology of $L$ with coefficients in the bundle $\Hom(\xi,\xi')$.  When $\xi=\xi'$  we take this chain-level model to be strictly
unital, and denote the unit by $1_{(L,\xi)}$ (or $1_L$ if we suppress local systems from the notation).  Floer theory defines a curved $A_{\infty}$-structure $\{\mu^d\}_{d \geq 0}$ on $\scrF_{pr}(M)$. 

Let $\lambda \in \Lambda_{>0}$. The Fukaya category $\scrF(M;\lambda)$ has objects pairs $(L,\alpha)$ where $L\in \mathrm{Ob}(\scrF_{pr}(M))$, where $\alpha \in hom^1_{\scrF_{pr}(M)}(L,L)$  vanishes in  $hom^1_{\scrF_{pr}(M)}(L,L) \otimes_{\Lambda_0} \bC$, and where $\alpha$ is a solution of the \emph{weakly unobstructed} Maurer-Cartan equation
\[
\mu^0 + \mu^1(\alpha) + \mu^2(\alpha,\alpha) + \cdots = \lambda\cdot 1_L \in hom^0_{\scrF_{pr}(M)}(L,L).
\]
The morphism spaces in $\scrF(M;\lambda)$ are given by  Floer cochains
\[
hom_{\scrF(M;\lambda)}(L, L') = hom_{\scrF_{pr}(M)}(L,L') \otimes_{\Lambda_0} \Lambda
\]
and inherit a non-curved $A_{\infty}$-structure obtained by all possible insertions of Maurer-Cartan elements; thus, the differential in the Floer complex for $(L_0,\alpha_0)$ and $(L_1,\alpha_1)$ is given by
\begin{equation} \label{eq:deformed-mu1}
\begin{aligned}
\mu^1_{\scrF(M;\lambda)}(x) & = \mu^1_{\scrF_{pr}(M)}(x) + \mu^2_{\scrF_{pr}(M)}(\alpha_1,x) + \mu^2_{\scrF_{pr}(M)}(x,\alpha_0) + \mu^3_{\scrF_{pr}(M)}(\alpha_1,\alpha_1,x) \\ & \qquad + \mu^3_{\scrF_{pr}(M)}(\alpha_1,x,\alpha_0) + \mu^3_{\scrF_{pr}(M)}(x,\alpha_0,\alpha_0) + \cdots
\end{aligned}
\end{equation}
The Floer differential squares to zero, i.e. $\scrF(M;\lambda)$ has vanishing curvature, since  $\lambda\cdot 1_L$ is central,  because $\scrF_{pr}(M)$ was assumed to be strictly unital.
The mapping class group $\pi_0\Symp(M)$ acts on $\scrF(M;\lambda)$ for each $\lambda \in \Lambda_{>0}$ separately.

The open-closed map $OC:HH_*(\scrF(M;\lambda),\scrF(M;\lambda))\to QH^*(M;\Lambda)$ takes
values in the generalized $\lambda$-eigenspace of quantum
multiplication by $c_1(TM)$, which is a subring
\[
QH^*(M)_\lambda=\{a\in QH^*(M):(c_1(TM)-\lambda)^N*a=0\text{ for some }N\in
\bN\}\subset QH^*(M).
\]
The analogue of Abouzaid's generation criterion
in this setting states that, if the restriction of the open-closed map to a full subcategory of 
$\scrF(M;\lambda)^{per}$ hits an invertible element of $QH^*(M)_\lambda$, then the
full subcategory split-generates $\scrF(M;\lambda)^{per}$ \cite[Theorem 11.3]{Ritter-Smith}.

\subsection{Relative parallel transport}

 Let $(X,\omega)$ be a symplectic manifold and $(Y_t)_{t\in [0,1]}$ a smooth family of symplectic submanifolds of $X$.

\begin{Lemma}\label{Lem:relative}
There is a time-dependent symplectomorphism $\Phi_t: X \rightarrow X$ with
$\Phi_0=\id$ and $\Phi_t(Y_0) = Y_t$, well-defined up to isotopy through symplectomorphisms with the same property.
\end{Lemma}

\begin{proof}
The existence is \cite[Proposition 4]{Auroux:gafa}. For uniqueness up to
isotopy, observe that two choices $\Phi_{0,t}$ and $\Phi_{1,t}$ differ by a
time-dependent symplectomorphism $\rho_{1,t}=\Phi_{0,t}^{-1}\circ \Phi_{1,t}$ 
which preserves $Y_0$ setwise, i.e.\ a path based at the origin in the group 
$\Symp(X,Y_0)$ of symplectomorphisms preserving $Y_0$ setwise. The path
$\rho_{1,t}$ can
be deformed continuously to the constant path $\rho_{0,t}\equiv \id$ (e.g.\ setting
$\rho_{s,t}=\rho_{1,st}$), and $\Phi_{s,t}=\Phi_{0,t}\circ \rho_{s,t}$ gives
an isotopy between $\Phi_{0,t}$ and $\Phi_{1,t}$.
\end{proof}

Recall that a symplectic fibration $\scrX \rightarrow B$ with fibre $(X,\omega)$ is a smooth fibre bundle with a globally closed 2-form $\Omega_{\scrX}$
such that $(X_b,\Omega_{\scrX}|_{X_b}) \cong (X,\omega)$ for each $b\in B$.

\begin{Corollary}
Given a symplectic fibration $\scrX \rightarrow B$ with fibre $X$ and a locally trivial symplectic subfibration $\scrY \hookrightarrow \scrX$ with fibre $Y \subset X$ at a base-point $b\in B$, there is a relative monodromy representation 
\[
\pi_1(B;b) \longrightarrow \pi_0\Symp(X,Y)
\]
into the mapping class group of the group of symplectomorphisms preserving $Y$ setwise.
\end{Corollary}

\begin{proof}
Fix a 1-dimensional submanifold  $\gamma \subset B$ and trivialise the bundle $\scrX|_{\gamma}$ over $\gamma$ by symplectic parallel transport. This brings us into the situation of Lemma \ref{Lem:relative}, meaning that we have a one-parameter family of symplectic embeddings (parametrized by a co-ordinate $t\in\gamma$) of $Y$ into a fixed $(X,\omega)$. 
By differentiating the relative Moser maps $\Phi_t$ of Lemma \ref{Lem:relative}, we obtain closed 1-forms $a_t \in \Omega^1(X)$ for which the $\omega$-dual vector fields $Z_t$  flow the submanifolds $Y_t$ into one another.  If we subtract $dt \wedge da_t$ from $\Omega_{\scrX}$, we obtain a new globally closed 2-form on $\scrX|_{\gamma}$ with the correct fibrewise restriction and for which parallel transport preserves the subfibration $\scrY|_{\gamma}$.
We can apply the preceding construction to the 1-skeleton of $B$ to obtain relative parallel transport maps for loops generating $\pi_1(B)$. The uniqueness up to isotopy in Lemma \ref{Lem:relative} shows the construction descends to a representation of $\pi_1(B)$. 
\end{proof}

Let $\omega_{st}$ denote the standard constant coefficient K\"ahler form on the four-torus $T^4$. Consider $(\Sigma_2 \times T^4, \omega \oplus \omega_{st})$. We fix a sufficiently small $\varepsilon> 0$ and let $p: Z \to \Sigma_2\times T^4$ denote the $\varepsilon$-symplectic blow-up of $\Sigma_2 \times T^4$ along the symplectic submanifold $C = \Sigma_2 \times \{0\}$. The exceptional divisor $E = C\times \bP^1$ is canonically a product; indeed $Z$ is just the product $\Sigma_2 \times Bl_{pt}(T^4)$, and carries a symplectic form $\Omega$ with cohomology class $p^*[\omega\oplus\omega_{st}] - \varepsilon\cdot E$. 

The cohomology of $Z$ admits a splitting
\[
H^*(Z;\bZ) = H^*(\Sigma_2\times T^4;\bZ) \oplus H^*(\Sigma_2;\bZ)\cdot u
\]
 where $u = -PD(E)$ has degree $2$. Let $\{\eta_j\}$ denote a basis for $H^1(Z;\bZ)$, and set 
\[
\Omega^{\delta}_{irr} = \Omega  + \delta \cdot \sum_{i.j} c_{ij} \eta_i \wedge \eta_j 
\]
for coefficients $c_{ij} \in (0,1)$. If $\delta > 0$ is sufficiently small then $\Omega^{\delta}_{irr}$ is a symplectic form on $Z$, because the symplectic condition is open and the $c_{ij}$ are bounded. 

\begin{Lemma}  \label{Lem:TrivialOnH*}
If the coefficients $c_{ij}$ are linearly independent over $\bQ$, then every symplectomorphism of $(Z,\Omega^{\delta}_{irr})$ acts trivially on $H^*(Z)$.
\end{Lemma}

\begin{proof} Assume that the $c_{ij}$ are rationally linearly independent.  We will show that any diffeomorphism of $Z$ preserving the cohomology class $[\Omega_{irr}]$ acts trivially on cohomology.  As a ring, $H^*(Z)$ is generated by  
$H^1(Z)\simeq H^1(\Sigma_2)\oplus H^1(T^4)$ and by the class $u\in H^2(Z)$.  Note that $\pi_2(Z)$ is generated by a fibre  $F \subset E$, so any diffeomorphism preserves the class $[F] \in H_2(Z;\bZ)$ and its intersection Poincar\'e dual $E \in H_4(Z;\bZ)$. 
The action on $H^2(Z;\bR)$ of a diffeomorphism of $Z$ which fixes $[\Omega^{\delta}_{irr}]$
has the eigenvalue $1$ appearing with multiplicity at least two (since $u$ and $[\Omega^{\delta}_{irr}]$ are both preserved).
However, the action on $H^2(Z) / \langle u\rangle$ is the action on $\Lambda^2(H^1(Z))$. Since the action on $H^1$ is through an integral matrix, the coefficients of any eigenvector
for the eigenvalue 1 must be linearly dependent over $\bQ$. Therefore, preservation of $[\Omega^{\delta}_{irr}]$ implies that the diffeomorphism acts trivially on $H^1(Z)$ and hence
on~$H^*(Z)$.
\end{proof}

The Torelli group $I_2 \leq \Gamma_2$ is an infinitely generated free group, generated by the Dehn twists on separating simple closed curves.   The construction of relative parallel transport applied to a family of blow-ups yields a representation
\[
\Gamma_2 \to \pi_0\Symp(Z, \Omega)
\]
which depends on the same kinds of choice as in Section \ref{Sec:Monodromy}.  Any element of $\Gamma_2$ acting non-trivially on cohomology cannot deform to a symplectomorphism with respect to the perturbed symplectic structure $\Omega^{\delta}_{irr}$, by Lemma \ref{Lem:TrivialOnH*}.  A given element of the Torelli group, however, will deform for $\delta$ sufficiently small. 

\begin{Corollary}\label{Cor:I2deforms}
Given $N>0$, there is $\delta(N) > 0$ such that $I_2 \to \pi_0\Symp(Z,\Omega)$ deforms on a rank $N$ free subgroup $\bF_N \leq I_2$ to  $\bF_N \to \pi_0\Symp(Z,\Omega^{\delta}_{irr})$
for all $\delta\in (0,\delta(N))$.
\end{Corollary}

\begin{proof}
The graph of $f \in I_2$ defines a Lagrangian submanifold $\Gamma(f) \subset (Z\times Z, \Omega \oplus -\Omega)$. Since the
cohomology classes of the perturbing forms $\eta_i \wedge \eta_j$ restrict trivially to $\Gamma(f)$, using that
$f^*([\eta]) = [\eta]$ for all $\eta$, if $\delta$ is sufficiently small there is a Lagrangian isotopy of $\Gamma(f)$ to a submanifold Lagrangian with respect to $\Omega^{\delta}_{irr}$.  Since being graphical is an open condition, if $\delta$ is sufficiently small this is again the graph of a symplectomorphism $f^{irr}$.  Pick $N$ separating simple closed curves $\sigma_j$ on $\Sigma_2$. Then the corresponding elements $f_{\sigma_j}$  admit common deformations $f_{\sigma_j}^{irr}$ to symplectomorphisms of $(Z,\Omega^{\delta}_{irr})$ if $\delta > 0$ is sufficiently small.
({\em A priori} the size of $\delta$ depends on geometric bounds on the Dehn twists about $\sigma_j$,
hence cannot be made uniform as $N\to \infty$.)
\end{proof}

Subsequently we will show that the homomorphism $\bF_N \to \pi_0\Symp(Z,\Omega^{\delta}_{irr})$ is faithful.

 \subsection{Unbounded rank}

Let $\scrL = \mathcal{O}_{\bP^1}(-1)$ denote the $\varepsilon$-blow-up of $\bC^2$ at the origin, equipped with its toric K\"ahler form in cohomology class $\varepsilon\cdot u$, where $u=-PD(E)$ is the negative of the Poincar\'e dual to the exceptional divisor (zero-section). The Gromov invariant of $E$  is non-trivial, and 
\[
QH^*(\scrL;\Lambda) = \Lambda[u]/\langle u(u+q^\varepsilon)\rangle
\] 
Note that $c_1(T\scrL) = u$. Implanting the local model into the four-torus, one finds that if $Y = Bl_{pt}(T^4)$ with the natural K\"ahler form $p^*\omega_{st} + \varepsilon\cdot u$, then
\[
QH^*(Y;\Lambda) \cong H^*(T^4;\Lambda) \oplus \Lambda\cdot u
\]
and the first Chern class $c_1(Y) = u = -PD(E)$ acts, under quantum multiplication, nilpotently on
all cohomology classes of positive degree in $H^*(T^4;\Lambda)$ and invertibly on the second summand.

\begin{Lemma} \label{Lem:summand} $\scrF(Y;-q^\varepsilon)^{per}$ 
is semisimple and generated by an idempotent summand $T^+$ of a Lagrangian torus $T\subset Y$.
\end{Lemma}

\begin{proof} $\scrL$ contains a (monotone) Lagrangian torus $T \subset \mathcal{O}_{\bP^1}(-1)$, which is 
the orbit of the torus action corresponding to the unique critical point of the toric potential function $W(x,y) = x + y + q^{-\varepsilon} xy$.  
The torus $T$ is weakly unobstructed, defines an object of $\scrF(\scrL;-q^\varepsilon)$
whose Floer cohomology is semisimple in characteristic zero, and splits into the direct sum of two idempotent summands $T^{\pm}$, which are isomorphic up to
shift \cite[Section 4.4]{Smith:HFQuadrics}.

A neighborhood of the zero section in $\scrL$ (large enough to
contain $T$) embeds into a neighborhood of the exceptional divisor in
$Y=Bl_{pt}(T^4)$; since all holomorphic discs bounded by $T$
in $Y$ must be contained inside the neighborhood of the exceptional divisor,
the Floer cohomology of the torus $T$ in $Y$ is exactly as in $\scrL$. Hence 
$T$ also defines an object of $\scrF(Y;-q^\varepsilon)$ with semisimple Floer
cohomology, which splits into two idempotent summands~$T^\pm$.
An explicit calculation shows that the images of the two idempotents of $HF(T,T)$
under the open-closed map are $\pm u=\mp PD(E)$;
the generation criterion \cite[Theorem 11.3]{Ritter-Smith} then implies
that $T$ split-generates $\scrF(Y;-q^\varepsilon)^{per}$.
\end{proof}

\begin{Proposition}
There is a fully faithful functor $\scrF(\Sigma_2) \to \scrF(Z,\Omega^{\delta}_{irr};\lambda)^{per}$
whose image split-generates $\scrF(Z,\Omega^{\delta}_{irr};\lambda)^{per}$ for
the eigenvalue $\lambda=-q^\epsilon$.
\end{Proposition}

\begin{proof}[Sketch]
For the product form $\Omega$ on $Z$, there is a  K\"unneth functor associated to the $A_{\infty}$-tensor product $\scrF(Z;\lambda)
\simeq \scrF(\Sigma_2) \otimes \scrF(Y;\lambda)$ and the semisimple piece of the second factor afforded by Lemma \ref{Lem:summand}.  Note that $T \subset Y$ survives arbitrary small perturbations of the given K\"ahler form $p^*\omega_{st} + \varepsilon\cdot u$ on $Y$, in the sense that it deforms as a Lagrangian to any sufficiently nearby symplectic form, since  the restriction map $H^2(Y) \to H^2(T)$ vanishes.  Since its Floer cohomology with respect to the initial symplectic form is semisimple, it must remain semisimple after small deformation. 

Choose the 1-forms $\eta_i$ on $Z=\Sigma_2\times Y$ so that
that $\eta_1,\dots,\eta_4$ are the pullbacks of closed 1-forms
$\alpha_1,\dots,\alpha_4$ on $\Sigma_2$ representing a basis of
$H^1(\Sigma_2;\bZ)$, and $\eta_5,\dots,\eta_8$ are the pullbacks of closed
1-forms on $Y$ which vanish everywhere in a neighborhood of the
exceptional divisor (and in particular on the torus $T$).

Then, for a fixed simple closed curve $\gamma \subset \Sigma_2$, the 
submanifold $\gamma\times T$ is Lagrangian not just for the product form
$\Omega$, but also for $\Omega^\delta_{irr}=\Omega + \delta \sum_{ij} c_{ij} \eta_i \wedge \eta_j$.

The association
\[
\gamma \mapsto \gamma \times T^+
\]
is globally realised by an $A_{\infty}$-functor associated to a Lagrangian correspondence
\[
G \subset \Sigma_2^- \times Z = \Sigma_2^- \times (\Sigma_2 \times Bl_{pt}(T^4))
\]
which fibres over the diagonal of $\Sigma_2^- \times \Sigma_2$ with fibre $T \subset \scrL \subset Y$. Note that $\Delta_{\Sigma_2} \times T$ remains Lagrangian after
perturbing the symplectic form on $Z$ by $\delta\sum_{ij} c_{ij}\eta_i\wedge
\eta_j$, and that on $\Sigma_2$ by $\delta\sum_{ij} c_{ij}\alpha_i\wedge
\alpha_j$. 
Since the correspondence is globally a product, the local-to-global spectral sequence $H^*(\Sigma_2; HF^*(T,T)) \to HF^*(G,G)$
degenerates; therefore the correspondence $G$ itself has an idempotent summand $G^+$ associated to a choice of idempotent for $T$, compare to \cite[Section 5.4]{Seidel:flux}. This yields the desired functor
$\scrG^+$ from $\scrF(\Sigma_2)$ (where we suppress from the notation the fact that the symplectic form depends on $\delta$) 
to $\scrF(Z,\Omega^\delta_{irr};-q^\varepsilon)^{per}$.  
This functor maps every object of $\scrF(\Sigma_2)$ to its product with
$T^+$; since we are in a product situation and $\End_{\scrF(Y;-q^\varepsilon)^{per}}(T^+) = \Lambda$, the functor is fully faithful.  

The fact that the image of the functor $\scrG^+$ split-generates $\scrF(Z;-q^\varepsilon)^{per}$
follows from Ganatra's automatic generation result \cite{Ganatra2016},
since  $HH^*(\scrF(\Sigma_2),\scrF(\Sigma_2))\simeq H^*(\Sigma_2;\Lambda)$ has the
same rank as $QH^*(Z)_{-q^\varepsilon}\simeq u\cdot H^*(\Sigma_2;\Lambda)$,
and $\scrF(\Sigma_2)$ is homologically smooth.
One can also proceed more directly: consider a collection of curves $\gamma_i\in\scrF(\Sigma_2)$ which satisfy Abouzaid's
split-generation criterion, i.e.\ the full subcategory with this set of objects
has a Hochschild cycle $\alpha$ which maps to the unit $1\in H^*(\Sigma_2;\Lambda)$ 
under the open-closed map. Then the Hochschild cycle
$\alpha^+=\scrG^+_*(\alpha)$, formed by replacing every morphism which
appears in $\alpha$ with its tensor product with $1_{T^+}$, maps to $1\otimes u\in QH^*(Z;\Lambda)$
(using the fact that we are locally in a product situation and $OC(1_{T^+})=u$).
Since $u$ is invertible in $QH^*(Z)_{-q^\varepsilon}$, \cite[Theorem
11.3]{Ritter-Smith} implies that the objects $\gamma_i\times T^+=\scrG^+(\gamma_i)$
split-generate $\scrF(Z;-q^\varepsilon)^{per}$. 
\end{proof}

Together with Corollary \ref{Cor:Main}, this proposition yields a natural map
\[
\Auteq(\scrF(Z,\Omega^{\delta}_{irr};-q^\varepsilon)^{per}) \cong \Auteq(\scrF(\Sigma_2)^{per}) \to \Gamma_2,
\]
which then induces a map 
\[
\Auteq_{HH}(\scrF(Z,\Omega^{\delta}_{irr};-q^\varepsilon)^{per}) \to I_2
\]
where the domain denotes those autoequivalences which act trivially on Hochschild cohomology.  We know that the closed-open map 
\[
H^*(\Sigma_2)\cdot u = QH^*(Z)_{-q^\varepsilon} \to HH^*(\scrF(Z,\Omega^{\delta}_{irr};-q^\varepsilon)^{per}) 
\]
is an isomorphism, and the map 
\[
\pi_0\Symp(Z,\Omega^{\delta}_{irr}) \to \Auteq(\scrF(Z,\Omega^{\delta}_{irr};-q^\varepsilon)^{per})
\]
lands in the subgroup $\Auteq_{HH}$ by Lemma \ref{Lem:TrivialOnH*}. 
Combining this with Corollary \ref{Cor:I2deforms}, we therefore obtain a map
\[
\bF_N \to \pi_0\Symp(Z,\Omega^{\delta}_{irr}) \to \Auteq_{HH}(\scrF(Z,\Omega^{\delta}_{irr};-q^\varepsilon)^{per})\to I_2 \cong \bF_{\infty}
\]
which one can compose with the quotient map $q: \bF_{\infty} \to \bF_N$ which kills all but the finite set of chosen generators (twists in all other separating simple closed curves).  
Chasing through the stages, the composite map is the natural inclusion of a finite rank free subgroup of the infinitely generated free group, or the identity after composition with $q$.  It follows that $\pi_0\Symp(Z,\Omega^{\delta}_{irr})$ surjects onto a free group of rank $N = N(\delta)$ which tends to infinity as $\delta \to 0$. This completes the proof of Theorem \ref{Thm:largeMCG}.


\bibliographystyle{amsplain}
\bibliography{mybib}

\end{document}